\documentclass{cedram-afst}

\usepackage{bm}
\usepackage[T1]{fontenc}
\usepackage[english]{babel}
\usepackage[utf8]{inputenc}
\usepackage[mathscr]{euscript}
\usepackage{textcomp,multicol,enumitem}
\usepackage[papersize={180mm,240mm},top=2cm, bottom=2cm, left=2cm , right=2cm]{geometry}

% ...

% ...

\title
[Strong Frobenius structures associated with q-difference operators]
{Strong Frobenius structures associated with q-difference operators}

\author{\firstname{Daniel} \middlename{} \lastname{vargas-Montoya}}

\thanks{This project has received funding from the European Research Council (ERC) under the European Union's Horizon 2020 research and innovation programme under the Grant Agreement No 648132}

  
\begin{document}
%% Abstracts must be placed before \maketitle
\begin{abstract}

The notion of strong Frobenius structure is classically studied in the theory of $p$-adic differential operators. In the present work, we introduce a new definition of the notion of strong Frobenius structure for $q$-difference operators. The relevance of this definition is supported by two main results. The first one deals with \emph{confluence}. We show that if the $q$-difference operator $L_q$ has a strong Frobenius structure for a prime $p$ with period $h$ and if $L$ is the $p$-adic differential operator obtained from $L_q$ by letting $q$ tend to 1, then $L$ has a strong Frobenius structure for $p$ with period $h$. The second one deals with congruence modulo cyclotomic polynomials. We show that if $f(q,z)\in\mathbb{Z}[q][[z]]$ is a solution of a $q$-difference operator having strong Frobenius structure for $p$ then $f(q,z)$ satisfies some congruences modulo the $p$-th cyclotomic polynomial. Another definition of strong Frobenius structures associated with $q$-difference operators has been introduced by André and Di Vizio and we also point out why their definition is not suitable for our applications: confluence and  congruence modulo cyclotomic polynomials. Finally, we show that  some $q$-hypergeometric operators of order 1 have a strong Frobenius strong for infinitely many primes numbers.

\end{abstract}

\maketitle

%\tableofcontents

\section{Introduction}

The notion of strong Frobenius structure for a differential operator was introduced by Dwork in \cite{DworksFf} and it is usually studied in the context of $p$-adic differential equations. Given a prime number $p$, the field $\mathbb{Q}_p$ of $p$-adic numbers is the completion of the field of rational numbers $\mathbb{Q}$ with respect to the $p$-adic norm, the ring $\mathbb{Z}_p$ is the set of of elements of $\mathbb{Q}_p$ whose $p$-adic norm is less than or equal to 1, and the field $\mathbb{C}_p$ is the completion of the algebraic closure of $\mathbb{Q}_p$ with respect to the $p$-adic norm. The field $\mathbb{C}_p(z)$ is endowed with the Gauss norm, $$\left|\frac{\sum_{i=0}^{n}a_iz^i}{\sum_{j=0}^{m}b_jz^j}\right|_{\mathcal{G}}=\frac{\sup\{a_i\}}{\sup\{b_j\}}.$$
The field $E_p$ of analytic elements is the completion of $\mathbb{C}_p(z)$ with respect to the Gauss norm. The field $E_p$ is equipped with the derivation $\delta=z\frac{d}{dz}$.

Let $L$ be a differential operator in $E_p[\delta]$ and let $A$ be the companion matrix of $L$. We say that $L$ has a \emph{strong Frobenius structure} with period $h$ if there are an integer $h>0$ and a matrix $H\in\rm{ GL}_n(E_p)$ such that $$\delta H=AH-p^hHA(z^{p^h}).$$

The matrix $H$ is called a \emph{Frobenius transition matrix}  for the operator $L$.	

Note that $\mathbb{Q}(z)\subset E_p$.  According to \cite[Chap. 22, Theorem 22.2.1]{Kedlaya}\footnote{See also p. 351 of \cite{Kedlaya} and the references given there in.} or \cite[Chap. V, p. 111 ]{andre},  the \emph{Picard-Fuchs} equations over $\mathbb{Q}(z)$ have a strong Frobenius structure for almost every prime number $p$. In \cite{vmsff},  we proved that, under some reasonable conditions, the differential operators over $\mathbb{Q}(z)$ whose monodromy group is rigid have a strong Frobenius structure for almost every prime number $p$ (a similar result is proved by Crew in \cite{crew}). For example, there are a large number of hypergeometric operators with rational parameters that satisfy these hypotheses.  We remind the reader that for any couple of vectors $\underline{\alpha}=(\alpha_1,\ldots,\alpha_n)$ and $\underline{\beta}=(\beta_1,\ldots,\beta_{n-1},1)$ in $(\mathbb{Q}\setminus\mathbb{Z}_{\leq0})^n$, we associate the hypergeometric operator
 $$\mathcal{H}(\underline{\alpha},\underline{\beta})=\prod_{i=1}^{n}(\delta+\beta_i-1)-z\prod_{i=1}^{n}(\delta+\alpha_i)\in\mathbb{Q}[z][\delta],\quad\delta=z\frac{d}{dz}.$$
Now, let $R$ be the algebraic closure of $\mathbb{Q}(q)$, where $q$ is a variable, and let $\delta_q$ be the automorphism of $R(z)$ given by $\delta_q(f(q,z,))=\frac{f(q,qz)-f(q,z)}{q-1}$. Let us consider the $q$-hypergeometric operator  $$\mathcal{ H}_{\delta_q}(\underline{\alpha},\underline{\beta})=\prod_{i=1}^n\left(q^{\beta_i-1}\delta_q+\frac{q^{\beta_i-1}-1}{q-1}\right)-z\prod_{i=1}^n\left(q^{\alpha_i}\delta_q+\frac{q^{\alpha_i}-1}{q-1}\right)\in R(z)[\delta_q].$$

The operator $\mathcal{ H}_{\delta_q}(\underline{\alpha},\underline{\beta})$ is a $q$-analog of the operator $\mathcal{ H}(\underline{\alpha},\underline{\beta})$ because by replacing $\delta_q$ by $\delta$ and $q$ by 1 into $\mathcal{ H}_{\delta_q}(\underline{\alpha},\underline{\beta})$ we recover $\mathcal{ H}(\underline{\alpha},\underline{\beta})$.

In \cite{vmsff}, we proved that if $\alpha_i-\beta_j\notin\mathbb{Z}$ for all $i,j\in\{1,\ldots,n\}$ and if $p$ is a prime number such that the $p$-adic norm of $\alpha_i$ and $\beta_j$ is less than or equal to 1 for all $i,j\in\{1,\ldots,n\}$, then the hypergeometric operator $\mathcal{H}(\underline{\alpha},\underline{\beta})$ has a strong Frobenius structure for $p$. As we have previously mentioned, the $q$-hypergeometric operator $\mathcal{H}_{\delta_q}(\underline{\alpha},\underline{\beta})$ is a $q$-analog of the hypergeometric operator $\mathcal{H}(\underline{\alpha},\underline{\beta})$. Therefore, an interesting problem is to determine whether there is a notion of strong Frobenius structures associated with $q$-difference operators such that we can recover the strong Frobenius structure of $\mathcal{H}(\underline{\alpha},\underline{\beta})$ from the strong Frobenius structure of $\mathcal{H}_{\delta_q}(\underline{\alpha},\underline{\beta})$ when specializing $q=1$.

\subsection{ Strong Frobenius structure associated with $q$-difference operators: congruences and specialization}

 We want to define a notion of strong Frobenius structure associated with $q$-difference operators satisfying the following two properties. The first one deals with \emph{confluence} or \emph{specialization}. That is, we want to show that if the $q$-difference operator $L_q$ has a strong Frobenius structure for a prime $p$ with period $h$ and if $L$ is the $p$-adic differential operator obtained from $L_q$ by letting $q$ tend to 1, then $L$ has a strong Frobenius structure for $p$ with period $h$. The second one deals with congruence modulo cyclotomic polynomials. We want to show that if $f(q,z)\in\mathbb{Z}[q][[z]]$ is a solution of a $q$-difference operator having a strong Frobenius structure for $p$ then $f(q,z)$ satisfies some congruences modulo the $p$-th cyclotomic polynomial.\\ 
 At present, we give our definition of a strong Frobenius structure associated with $q$-difference operators.  The main ingredients in this definition are the ultrametric rings $\mathbb{Z}_{q,p}$ and $E_{\mathbb{Z},q,p}$, where $p$ is a prime number and $q$ is a variable. In Section~\ref{section_def_sff} we construct these rings. However, here we point out a crucial fact in the construction of these rings. Let $\mathfrak{m}$ be the ideal of $\mathbb{Z}[q]$ generated by $p$ and $1-q$, since $\mathbb{Z}[q]$ is a Noetherian ring it follows that for every non-zero $x\in\mathbb{Z}[q]$ there is a unique positive integer $n$ such that $x\in\mathfrak{m}^n\setminus\mathfrak{m}^{n+1}$, so that we set $|x|_{\mathfrak{m}}=\frac{1}{p^n}$. Then $|\cdot|_{\mathfrak{m}}$ is an ultrametric absolute value  on $\mathbb{Z}[q]$. This fact is proved in Section~\ref{section_def_sff}. The  ring $E_{\mathbb{Z},q,p}$ is a ring of functions and is endowed with the automorphisms $\sigma_q$ and $\delta_q$, where $\sigma_q(f(q,z))=f(q,qz)$ and $\delta_q(f(q,z))=\frac{f(q,qz)-f(q,z)}{q-1}$. It is also equipped with a Frobenius endomorphism $\textbf{F}_q$ such that $\textbf{F}_q(q)=q^p$. Furthermore, this ring verifies the following properties: the ring $\mathbb{Z}[q][z]$ is contained in $E_{\mathbb{Z},q,p}$ and there is a continuous homomorphism $\textbf{ev}:E_{\mathbb{Z},q,p}\rightarrow E_p\cap\mathbb{Z}_p[[z]]$ such that for every $P(q,z)\in\mathbb{Z}[q][z]$, $\textbf{ev}(P(q,z))=P(1,z)$. In our context, the homomorphism $\textbf{ev}$ represents the specialization of the variable $q$ at 1. Now consider the $q$-difference operator of order $n$
\begin{equation}\label{eq: q_diff_ope}
L_{\delta_q}:=\delta^n_q+b_{1}(z)\delta^{n-1}_q+\cdots+b_{n-1}(z)\delta_q+b_{n}(z)\in E_{\mathbb{Z},q,p}[\delta_q].
\end{equation}
We set
 \begin{equation}\label{eq:ev}
 \textbf{ev}(L_{\delta_q})=\delta^n+\textbf{ev}(b_1(z))\delta^{n-1}+\cdots+\textbf{ev}(b_{n-1}(z))\delta+\textbf{ev}(b_n(z)).
 \end{equation}
We say that ${L}_{\delta_q}$ is a $q$-analog of the differential operator $\textbf{ev}(L_{\delta_q})$.  For example, from Section~\ref{sec_qhyp1} we conclude that $\mathcal{H}_{\delta_q}(\underline{\alpha},\underline{\beta})\in E_{\mathbb{Z},q,p}[\delta_q]$ and $\textbf{ev}(\mathcal{H}_{\delta_q}(\underline{\alpha},\underline{\beta}))=\mathcal{H}(\underline{\alpha},\underline{\beta})$.

We associate to $L_{\delta_q}$ the following matrix \[B_q=\begin{pmatrix}
1 & (q-1) & 0 & \cdots & 0 \\
0 & 1 & (q-1) & \cdots & 0\\
\vdots & \vdots & \vdots & \cdots & \vdots\\
0 & 0 & 0 & \cdots & (q-1)\\
-(q-1)b_n(z) & -(q-1)b_{n-1}(z) & -(q-1)b_{n-2}(z) & \cdots & -(q-1)b_1+1\\ 
\end{pmatrix}.\]
Let $B_q^{\textbf{F}_q^h}\in\mathcal{M}_n(E_{\mathbb{Z},q,p})$ be the matrix obtained by applying $h$-times the automorphism $\textbf{F}_q$ to each entry of $B_q$. The matrix $B_q$ is called \emph{the matrix associated with $L_{\delta_q}$}. Here is our definition of \emph{strong Frobenius structure for a $q$-difference operator} with coefficients in $E_{\mathbb{Z},q,p}$.
\begin{defi}
	Let $L_{\delta_q}$ be as in \eqref{eq: q_diff_ope}  and $B_q$ be the matrix associated with $L_{\delta_q}$. We say that $L_{\delta_q}$ \emph{has a strong Frobenius structure} with period $h$ if there  are an integer $h>0$ and $H_q\in\rm{ GL}_n(E_{\mathbb{Z},q,p})$ such that
	\begin{equation}\label{expressionsff 13}
	\sigma_q(H_q)B^{\textbf{F}_q^h}_q=B_qH_q.
	\end{equation}
	The matrix $H_q$ is called a \emph{Frobenius transition matrix} for the operator $L_{\delta_q}$. 
\end{defi}
Equality~\eqref{expressionsff 13} is understood to mean that the $\sigma_q$-modules defined by the matrices $B^{\textbf{F}_q^h}_q$ and $B_q$ are isomorphic as $\sigma_q$-modules over $E_{\mathbb{Z},q,p}$.

Through this definition we want answer the following two questions.  The first one deals with  the \emph{confluence} property. Let $L_{\delta_q}\in E_{\mathbb{Z},q,p}[\delta_q]$ be as in \eqref{eq: q_diff_ope} and suppose that $L_{\delta_q}$ has a strong Frobenius structure with period $h$. Let $H_q$ be a Frobenius transition matrix for this operator, let $B$ be the companion matrix of $ \textbf{ev}(L_{\delta_q})$ and let $\textbf{ev}(H_q)$ be the matrix obtained by applying  $\textbf{ev}$ to each entry of $H_q$. Do we have the equality $\delta(\textbf{ev}(H_q))=B\textbf{ev}(H_q)-p^h\textbf{ev}(H_q)B(z^{p^h})$?\\
Since  $H_q\in\rm{ GL}_n(E_{\mathbb{Z},q,p})$ and $\textbf{ev}$ is a homomorphism of rings, the equality  $\delta(\textbf{ev}(H_q))=B\textbf{ev}(H_q)-p^h\textbf{ev}(H_q)B(z^{p^h})$ would imply that $ \textbf{ev}(L_{\delta_q})$  has a strong Frobenius structure with period $h$ and $\textbf{ev}(H_q)$ is a Frobenius transition matrix for the operator $\textbf{ev}(L_{\delta_q})$. For this reason, in case where this question has an affirmative answer we can recover the strong Frobenius structure of the differential operator $ \textbf{ev}(L_{\delta_q})$ from the strong Frobenius structure of the $q$-difference operator $L_{\delta_q}$ by applying $ \textbf{ev}$ to each entry of $H_q$. The second one deals with congruence modulo cyclotomic polynomial. In \cite{cyclotomique} some congruences involving the $p$-th cyclotomic polynomial are obtained for certain $q$-series in $\mathbb{Z}[q][[z]]$. Recall that the $p$-th cyclotomic polynomial is the polynomial $\phi_p(q)=1+q+\cdots+q^{p-1}$.  In \cite{cyclotomique}, the authors proved that there are many power series $f(q,z)$ in $1+z\mathbb{Z}[q][[z]]$ such that for every prime number $p$, there is a polynomial $A_p(q,z)$ in $\mathbb{Z}[q][z]$ whose degree in $z$ is less than $p$ such that $f(q,z)\equiv A_p(q,z)f(1,z^p)\bmod\phi_p(q)\mathbb{Z}[q][[z]]$. But  $q^p\equiv1\bmod\phi_p(q)$ and therefore,  \begin{equation}\label{eq: phi p lucas}
f(q,z)\equiv A_p(q,z)f(q^p,z^p)\bmod\phi_p(q)\mathbb{Z}[q][[z]].
\end{equation}
Hence, by substituting $q=1$ into \eqref{eq: phi p lucas}, we deduce that for all prime numbers $p$, 
\begin{equation}\label{eq:plucas}
f(1,z)\equiv A_p(1,z)f(1,z^p)\bmod p\mathbb{Z}[[z]].
\end{equation}

%Let $g(z)=\sum_{n\geq0}a(n)$ be in $\mathbb{Z}[[z]]$ and let $p$ be a prime number.  We let  $$g_{\mid p}(z)=\sum_{n\geq0}(a(n)\bmod p)z^n\in\mathbb{F}_p[[z]]$$
%denote the reduction of $g(z)$ modulo $p$.

The relevance of the equations of type \eqref{eq:plucas} is supported from \cite{ABD2019}, where the authors gave  an algebraic independence criterion for power series that satisfy equations of type  \eqref{eq:plucas} for almost every prime number $p$.  Further, in \cite{cyclotomique} the authors gave an algebraic independence criterion for $q$-series that satisfy equations of type \eqref{eq: phi p lucas} for almost every prime number $p$. The criterion given in \cite{cyclotomique} is a phenomenon propagation of algebraic independence the $ \textbf{ev}(f_1(q,z)),\ldots,\textbf{ev}(f_r(q,z))$ to their $q$-analogs $f_1(q,z),\ldots, f_r(q,z)$. It turns out that  many of the power series $f(q,z)\in 1+z\mathbb{Z}[q][[z]]$ considered in \cite{cyclotomique} verify that $f(1,z)$ is a solution of a differential operator having a strong Frobenius structure for $p$. Therefore, according to \cite[Theorem 2.1]{vmsff},  if $n$ is the order of this differential operator and $h$ is the period of the strong Frobenius structure then, there are $a_{0}(z),\ldots, a_{n^2}(z)\in\mathbb{Z}[z]$ such that their reductions modulo $p$ are not all zero and
\begin{equation}\label{eq:thm2_1}
\sum_{i=0}^na_i(z)f(1,z^{p^{ih}})\equiv0\bmod p\mathbb{Z}[[z]].
\end{equation}
Recall that we can recover Equality~\eqref{eq:plucas} by substituting $q=1$ into \eqref{eq: phi p lucas}. A natural question in this direction is to determine whether there exist  $d_0(q,z),\ldots, d_{n^2}(q,z)\in\mathbb{Z}[q][z]$ such that 
\begin{equation}\label{eq_mudulo_cyclotomic_polynomial}
\sum_{i=0}^nd_i(q,z)f(q^{p^{ih}},z^{p^{ih}})=0\bmod\phi_p(q)\mathbb{Z}[q][[z]]
\end{equation}
and additionally that by substituting $q=1$ into \eqref{eq_mudulo_cyclotomic_polynomial} we can recover Equality~\eqref{eq:thm2_1}. \\
More precisely, let $f(q,z)\in\mathbb{Z}[q][[z]]$ be a solution of a $q$-difference operator of order $n$ having a strong Frobenius structure for $p$ with period $h$, then, by analogy with Theorem~2.1 of \cite{vmsff}, we ask the following question.\\
Are there $d_0(q,z),\ldots,d_n(q,z)\in\vartheta_{Frac(\mathbb{Z}_{q,p})}[[z]]$ such that $$\sum_{i=0}^nd_i(q,z)f(q^{p^{ih}},z^{p^{ih}})=0\bmod\phi_p(q)\mathbb{Z}[q][[z]]$$ and  by substituting $q=1$ into it we can recover an equality of the form \eqref{eq:thm2_1}?

\subsection{Main results}
Our main results, Theorem~\ref{confluence3} and Theorem~\ref{alg3}, put us in a position to answer affirmatively  to the two questions asked in the previous section.
\begin{theorem}\label{confluence3}
	Let $L_{\delta_q}=\delta^n_q+b_1(z)\delta^{n-1}_q+\cdots+b_{n-1}(z)\delta_q+b_n(z)$ be in $E_{\mathbb{Z},q,p}[\delta_q]$. If $L_{\delta_q}$ has a strong Frobenius structure with period $h$, then the differential operator $\textbf{ev}(L_{\delta_q})\in E_p[\delta]$ has a strong Frobenius structure with period $h$. Moreover, if $H_q$ is a  Frobenius transition matrix for the operator $L_{\delta_q}$, then $\textbf{ev}(H_q)$ is a Frobenius transition matrix for the operator $\textbf{ev}(L_{\delta_q})$.
\end{theorem}

 Theorem~\ref{confluence3} answers the question regarding the confluence property and in Section~\ref{sub_recuperando}, by using Theorem~\ref{alg3}, we will give an affirmative answer to  the question regarding the congruences modulo cyclotomic polynomials.

\begin{theorem}\label{alg3}
	Let  $L_q$ be in $E_{\mathbb{Z},q,p}[\sigma_q]$ of order $n$ and $f(q,z)$ be in $\mathbb{Z}_{q,p}[[z]]$ a solution of $L_q$. Suppose that $L_q$ has a strong Frobenius structure with period $h$. Then there exist $d_0(q,z),\ldots,d_{n^2}(q,z)\in\vartheta_{Frac(\mathbb{Z}_{q,p})}[[z]]$  and a non-zero constant  $c\in\mathbb{Z}_{q,p}$ such that the polynomial  $d_0(q,z)Y_0+d_1(q,z)Y_1+\cdots+d_{n^2}(q,z)Y_{n^2}\bmod\phi_p(q)\vartheta_{Frac(\mathbb{Z}_{q,p})}$ is not zero, $cd_0(z),\ldots, cd_{n^2}(z)\in E_{\mathbb{Z},q,p}$, and
	\begin{equation}\label{genrale3}
	\sum_{i=0}^{n^2}d_i(q,z)\textbf{F}^{ih}_q(f(q,z))\equiv0\bmod\phi_p(q)\vartheta_{Frac(\mathbb{Z}_{q,p})}[[z]].
	\end{equation}	
\end{theorem}

\subsection{Structure of the article}
In Section~\ref{section_def_sff}, we construct the ultrametric rings $\mathbb{Z}_{q,p}$ and $E_{\mathbb{Z},q,p}$. In Section~\ref{sec_proof}, we prove Theorem~\ref{confluence3} and Theorem~\ref{alg3}. In Section~\ref{sec_qhyp1}, we show that some $q$-hypergoemetric operators of order 1 have a strong Frobenius structure for infinitely many prime numbers $p$. In \cite{anredivzio}, André and Di Vizio  gave also a definition of strong Frobenius structure associated with $q$-difference operators. In Section~\ref{sec_andre_divizio}, we point out why this definition does not seem to be sufficient to answer the questions asked in this introduction. 

\section{ A strong Frobenius structure associated with $q$-difference operators}\label{section_def_sff}

We begin by remembering the definition of a Frobenius endomorphism. Let $A$ be an ultrametric ring of characteristic zero, let $\vartheta_A$ be the set of  $x\in A$ with $|x|\leq1$ and let $\mathcal{I}$ be the set of  $x\in A$ with $|x|<1$. Since the norm is ultremetric, the set $\vartheta_A$ is a local ring whose maximal ideal is  $\mathcal{I}$. \emph{The residue ring of $A$} is the ring $k:=\vartheta_A/\mathcal{I}$. We suppose that $k$ has prime characteristic $p>0$, that is, we suppose that the prime number $p$ belongs to $\mathcal{I}$. For $|x|\leq1$ we denote by $\overline{x}$ the residual class of $x$ in $k$. 
\begin{defi}
	Let $A$ be an ultrametric ring and $k$ its residue ring. We assume that the characteristic of $k$ is $p>0$. We say that $\tau:A\rightarrow A$ is a Frobenius endomorphism if:
	\begin{enumerate}
		\item The endomorphism $\tau$ is continuous.
		\item For all $x\in\vartheta_A$, $|\tau(x)-x^p|<1$. That is, the endomorphism $\overline{\tau}:k\rightarrow k$ given by $\overline{\tau}(\overline{x})=\overline{\tau(x)}$ is the Frobenius endomorphism of $k$.
	\end{enumerate}
\end{defi}
Consider the ring $\mathbb{Z}[q]$, where $q$ is a variable and let  $\textbf{ev}:\mathbb{Z}[q]\rightarrow\mathbb{Z}$ be the homomorphism given by $\textbf{ev}(P(q))=P(1)$. 

From now on, we will fix the prime number $p>0$ and the ideal $\mathfrak{m}=(p,1-q)\mathbb{Z}[q]$.

\begin{rema}\label{noyau3}
	Let $\mathbb{F}_p$ be the field with $p$ elements. The ideal $\mathfrak{m}$ is a maximal ideal of $\mathbb{Z}[q]$ because $\mathfrak{m}$ is the kernel of the surjective homomorphism $\pi\circ\textbf{ev}:\mathbb{Z}[q]\rightarrow\mathbb{F}_p$, where $\pi:\mathbb{Z}\rightarrow\mathbb{F}_p$ is the homomorphism given by $\pi(n)=n\bmod p$.
\end{rema}
Since the ring $\mathbb{Z}[q]$ is Noetherian and $\mathfrak{m}\neq (1)$, it follows from the corollary 10.18 of  \cite{atiyah} that $\bigcap_{n\geq0}\mathfrak{m}^n=\{0\}$. Thus, if $x\in\mathbb{Z}[q]$ is not zero, there is a unique positive integer $n$ such that $x\in\mathfrak{m}^{n}$ and $x\notin\mathfrak{m}^{n+1}$. We set $$| x|_{\mathfrak{m}}=\frac{1}{p^n}.$$
According to the following proposition, we see that $|\cdot|_{\mathfrak{m}}$ is an ultrametric absolute value  on $\mathbb{Z}[q]$.
 
\begin{prop}\label{prop_m_adique3}
	Let $x,y$ be in $\mathbb{Z}[q]$, then:
	\begin{enumerate}
		\item $|x|_{\mathfrak{m}}=0$ if and only if $x=0$.
		\item $|x+y|_{\mathfrak{m}}\leq\max\{|x|_{\mathfrak{m}},| y|_{\mathfrak{m}}\}$
		\item $|xy|_{\mathfrak{m}}=|x|_{\mathfrak{m}}| y|_{\mathfrak{m}}$.
	\end{enumerate}
\end{prop}

We prove  Proposition~\ref{prop_m_adique3} at the end of this section. So, the ring $\mathbb{Z}[q]$ is equipped with an ultrametric absolute value which we call $\mathfrak{m}$-adic norm. In this case, it is clear that: $\vartheta_{\mathbb{Z}[q]}=\mathbb{Z}[q]$ and the ideal maximal of $\vartheta_{\mathbb{Z}[q]}$ is $\mathfrak{m}\mathbb{Z}[q]$. Thus, according to Remark~\ref{noyau3}, the residue ring of $\mathbb{Z}[q]$ is $\mathbb{F}_p$.

\begin{prop}\label{e3}
	Let $\mathbb{Z}[q]$ be endowed with the $\mathfrak{m}$-adic norm and let $\mathbb{Z}$ be endowed with the $p$-adic norm. Then the homomorphism $\textbf{ev}:\mathbb{Z}[q]\rightarrow\mathbb{Z}$ is continuous, and for all $P(q)\in\mathbb{Z}[q]$, $|\textbf{ev}(P(q))|_p\leq|P(q)|_{\mathfrak{m}}$.
\end{prop}

\begin{proof}
	Let $P(q)$ be in $\mathbb{Z}[q]$. Suppose that $|P(q)|_{\mathfrak{m}}=1/p^n$, then $P(q)\in\mathfrak{m}^n\setminus\mathfrak{m}^{n+1}$. Since $\mathfrak{m}^n$ is generated by the family $\{p^{n-i}(1-q)^i\}_{0\leq i\leq n}$, $P(q)=\sum_{i=0}^n a_i(q)p^{n-i}(1-q)^i$ where $a_0(q),\ldots, a_n(q)$ belong to $\mathbb{Z}[q]$. Hence, $\textbf{ev}(P(q))=P(1)=a_0(1)p^n$ and for this reason, $|P(1)|_p\leq 1/p^n$. Therefore, for all $P(q)\in\mathbb{Z}[q]$,
	\begin{equation}\label{econtinu}
	|\textbf{ev}(P(q))|_{p}\leq|P(q)|_{\mathfrak{m}}.
	\end{equation}
	Since $\textbf{ev}$ is a homomorphism of rings, this last inequality implies that $\textbf{ev}$ is continuous.
\end{proof}

Now, let us consider the endomorphism  $\textbf{F}_q:\mathbb{Z}[q]\rightarrow\mathbb{Z}[q]$ given by $\textbf{F}_q(P(q))=P(q^p)$.

\begin{prop}\label{frob3}
	The endomorphism $\textbf{F}_q:\mathbb{Z}[q]\rightarrow\mathbb{Z}[q]$ is a Frobenius endomorphism, and for all $P(q)\in\mathbb{Z}[q]$, $|\textbf{F}_q(P(q))|_{\mathfrak{m}}\leq|P(q)|_{\mathfrak{m}}.$
\end{prop}

\begin{proof}	
	Let us show that $\textbf{F}_q$ is continuous. Note that, for all integers $n>0$, we have $\textbf{F}_q(\mathfrak{m}^n\mathbb{Z}[q])\subset\mathfrak{m}^n\mathbb{Z}[q]$ because $\textbf{F}_q(\mathfrak{m}\mathbb{Z}[q])\subset\mathfrak{m}\mathbb{Z}[q]$ and $\textbf{F}_q$ is an endomorhism of rings. Then,  $$|\textbf{F}_q(P(q))|_{\mathfrak{m}}=|P(q^{p})|_{\mathfrak{m}}\leq|P(q)|_{\mathfrak{m}}.$$ 
	Since $\textbf{F}_q$ is an endomorphism of rings, from this last inequality it follows that $\textbf{F}_q$ is continuous.
	
	Now, we show that for all $x\in\mathbb{Z}[q]$,  $|\textbf{F}_q(x)-x^p|_{\mathfrak{m}}<1$. Let $x=P(q)$ be in $\mathbb{Z}[q]$, then $$\textbf{F}_q(x)-x^p=P(q^{p})-P(q)^p.$$
	As $\pi_p\circ\textbf{ev}(P(q^{p})-P(q)^p)=0$ then, according to Remark~\ref{noyau3}, $P(q^{p})-P(q)^p\in\mathfrak{m}$. Therefore, $$|\textbf{F}_q(x)-x^p|_{\mathfrak{m}}<1.$$
	Then, we conclude that $ \textbf{F}_q$ is a Frobenius endomorphism.
\end{proof}

We are going to explain our motivation for considering the Frobenius endomorphism $\textbf{F}_q:\mathbb{Z}[q]\rightarrow\mathbb{Z}[q]$ given by $\textbf{F}_q(q)=q^p$. We recall that the endomorphism $\textbf{F}:\mathbb{Z}[[z]]\rightarrow\mathbb{Z}[[z]]$ given by $\textbf{F}(z)=z^p$ is the unique Frobenius endomorphism of $\mathbb{Z}[[z]]$\footnote{Here the ring $\mathbb{Z}[[z]]$ is equipped with the $p$-adic Gauss norm, $ \left|\sum_{n\geq0}a_nz^n\right|_{p,\mathcal{G}}=\sup\{|a_n|_p\}_{n\geq0}$.} such that $z\mapsto z^p$. We set $\delta=z\frac{d}{dz}$, then for all integers $h>0$, $\textbf{F}$ is the unique endomorphism of $\mathbb{Z}[[z]]$ such that
\begin{equation}\label{eq_egal_frob}
\delta\circ\textbf{F}^h=p^h\textbf{F}^h\circ\delta.
\end{equation} 
We consider the following homomorphisms: $$\textbf{ev}:\mathbb{Z}[q][[z]]\rightarrow\mathbb{Z}[[z]]\quad\textup{,}\quad \textbf{ev}\left(\sum_{n\geq0}P_n(q)z^n\right)=\sum_{n\geq0}P_n(1)z^n$$ and $$\delta_q:\mathbb{Z}[q][[z]]\rightarrow\mathbb{Z}[q][[z]] \quad\textup{,}\quad\delta_q\left(\sum_{n\geq0}P_n(q)z^n\right)=\sum_{n\geq0}(1+q+\cdots+q^{n-1})P_n(q)z^n.$$ We say that $f(q,z)\in\mathbb{Z}[q][[z]]$ is a $q$-analog of $g(z)\in\mathbb{Z}[[z]]$ if $\textbf{ev}(f(q,z))=g(z)$. In particular, the polynomial $[p^h]_q=1+q+\cdots+q^{{p^h}-1}$ is a $q$-analog of $p^h$. Since,  for $f(q,z)\in\mathbb{Z}[q][[z]]$, $\textbf{ev}(\delta_q(f(q,z))=\delta(\textbf{ev}(f(q,z))$, it follows that  %$$\textbf{ev}\left(\delta_q\left(\sum_{n\geq0}P_n(q)z^n\right)\right)=\delta\left(\textbf{ev}\left(\sum_{n\geq0}P_n(q)z^n\right)\right)=\delta\left(\sum_{n\geq0}P_n(1)z^n\right),$$ it follows that for all $f(q,z)\in\mathbb{Z}[q][[z]]$, 
 $\delta_q(f(q,z))$ is a $q$-analog of $\delta(f(1,z))$. 
 
 \smallskip
 
 It turns out that $\textbf{F}_q:\mathbb{Z}[q][[z]]\rightarrow\mathbb{Z}[q][[z]]$ given by $\textbf{F}_q(\sum_{n\geq0}P_n(q)z^n)=\sum_{n\geq0}P_n(q^p)z^{np}$ is the unique endomorphism of $\mathbb{Z}[q][[z]]$  such that, for every integer $h>0$, we have 
\begin{equation}\label{eq_q_analogue_egal_frob}
\textbf{ev}\circ\delta_q\circ\textbf{F}_q^h=\delta\circ\textbf{F}^h\circ\textbf{ev}.
\end{equation} 
And additionally, $\textbf{ev}\circ[p^h]_q\textbf{F}_q^h\circ\delta_q=p^h\textbf{F}^h\circ\delta\circ\textbf{ev}$, and $\textbf{ev}\circ[p^h]_q\textbf{F}_q^h\circ\delta_q=p^h\textbf{F}^h\circ\delta\circ\textbf{ev}$. These last two equalities imply that \eqref{eq_q_analogue_egal_frob} is a $q$-analog of \eqref{eq_egal_frob}.

The field $\mathbb{Q}(q)$ is the field of fractions of $\mathbb{Z}[q]$. The $\mathfrak{m}$-adic norm extends to $\mathbb{Q}(q)$ in a unique way. It extends as follows: let $x$ be in $\mathbb{Q}(q)$, then there are $x_1,x_2\in\mathbb{Z}[q]$ such that $x=\frac{x_1}{x_2}$. We set $|x|_{\mathfrak{m}}=\frac{| x_1|_{\mathfrak{m}}}{|x_2|_{\mathfrak{m}}}$. The value of $|x|_{\mathfrak{m}}$ does not depend on the writing of $x$. Thus,  $\mathbb{Q}(q)$ is endowed with the ultrametric norm $|\cdot|_{\mathfrak{m}}$ which we will call $\mathfrak{m}$-adic norm. According to Proposition~\ref{frob3}, we know that $\textbf{F}_q:\mathbb{Z}[q]\rightarrow\mathbb{Z}[q]$ is a Frobenius endomorphism. A natural question is to determine whether $\textbf{F}_q$ extends to a Frobenius endomorphism of $\mathbb{Q}(q)$. Unfortunately, that is not the case. In fact,  we are going to see that if that $\tau$ is a Frobenius endomorphism of $\mathbb{Q}(q)$ such that $\tau=\textbf{F}_q$ on $\mathbb{Z}[q]$, then $\tau$ is not a Frobenius endomorphism. Indeed, since $\tau=\textbf{F}_q$ on $\mathbb{Z}[q]$ and $\tau$ is a homomorphism of rings, we get that $\tau\left(\frac{P(q)}{R(q)}\right)=\frac{P(q^{p})}{R(q^{p})}$. On the one hand, as $1-q$ and $p$ belong to $\mathfrak{m}\setminus\mathfrak{m}^2$ then the $\mathfrak{m}$-adic norm of $x=\frac{1-q}{p}$ is equal to 1. On the other hand, the $\mathfrak{m}$-adic norm of $1-q^{p}$ is equal to $1/p^2$ because $1-q^p=(1-q)[p]_q\in\mathfrak{m}^2\setminus\mathfrak{m}^3$. Therefore, the $\mathfrak{m}$-adic norm of $\tau(x)=\frac{1-q^{p}}{p}$ is equal to $1/p$, and since the norm is ultrametric, we deduce that the $\mathfrak{m}$-adic norm of $\tau(x)-x^p$ is equal to 1. Thus, $\tau$ is not a Frobenius endomorphism of $\mathbb{Q}(q)$. Consequently, the field $\mathbb{Q}(q)$ is not equipped with a Frobenius endormorphism such that $q\mapsto q^{p}$. However, for the subring $\mathbb{Z}[q]_{\mathfrak{m}}$ of $\mathbb{Q}(q)$ we can extend the endomorphisme $\textbf{F}_q$ to a Frobenius endomorphism of $\mathbb{Z}[q]_{\mathfrak{m}}$. The ring $\mathbb{Z}[q]_{\mathfrak{m}}$ is the  localization of $\mathbb{Z}[q]$ at ideal $\mathfrak{m}$, that is, the ring $\mathbb{Z}[q]_{\mathfrak{m}}$ is the set of rational functions $\frac{P(q)}{Q(q)}$ such that $P(q)$, $Q(q)\in\mathbb{Z}[q]$ and $Q(q)\notin\mathfrak{m}$.

\begin{rema}\label{rq: bn def3}
	\begin{enumerate}[label=(\alph*)]
		\item Let $Q(q)$ be in $\mathbb{Z}[q]$. It follows from the definition of the $\mathfrak{m}$-adic norm that $Q(q)\notin\mathfrak{m}$ if and only if $|Q(q)|_{\mathfrak{m}}=1$. Moreover, Remark~\ref{noyau3} implies that $Q(q)\notin\mathfrak{m}$ if and only if $Q(1)\notin(p)$. Therefore, $|Q(q)|_{\mathfrak{m}}=1$ if and only if $|Q(1)|_p=1$.
		
		\item It follows from (a) that all elements in $\mathbb{Z}[q]_{\mathfrak{m}}$ has a norm less than or equal to 1. Consequently, $\vartheta_{\mathbb{Z}[q]_{\mathfrak{m}}}=\mathbb{Z}[q]_{\mathfrak{m}}$. Moreover, the maximal ideal of $\mathbb{Z}[q]_{\mathfrak{m}}$ is  $\mathfrak{m}\mathbb{Z}[q]_{\mathfrak{m}}$. Thus, the residue ring of $\mathbb{Z}[q]_{\mathfrak{m}}$  is $\mathbb{F}_p$.
	\end{enumerate}
\end{rema}
This last remark implies that the homomorphisms $\textbf{ev}:\mathbb{Z}[q]_{\mathfrak{m}}\rightarrow\mathbb{Z}_{p}$ and $\textbf{F}_q:\mathbb{Z}[q]_{\mathfrak{m}}\rightarrow\mathbb{Z}[q]_{\mathfrak{m}}$ given by $\textbf{ev}\left(\frac{P(q)}{Q(q)}\right)=\frac{P(1)}{Q(1)}$ and $\textbf{F}_q\left(\frac{P(q)}{Q(q)}\right)=\frac{P(q^p)}{Q(q^p)}$ are well defined. 
\begin{prop}\label{prop: ev con3}
	The homomorphism  $\textbf{ev}:\mathbb{Z}[q]_{\mathfrak{m}}\rightarrow\mathbb{Z}_{p}$ given by $\textbf{ev}\left(\frac{P(q)}{Q(q)}\right)=\frac{P(1)}{Q(1)}$ is continuous, and for all $x\in\mathbb{Z}[q]_{\mathfrak{m}}$, $|\textbf{ev}(x)|_p\leq|x|_{\mathfrak{m}}$.
\end{prop}

\begin{proof}
Let $x=\frac{P(q)}{Q(q)}$ be in $\mathbb{Z}[q]_{\mathfrak{m}}$, where $P(q),Q(q)\in\mathbb{Z}[q]$ and $Q(q)\notin\mathfrak{m}$. From Proposition~\ref{e3} and Remark~\ref{rq: bn def3} we obtain $$|\textbf{ev}(x)|_p=\frac{|\textbf{ev}(P(q))|_p}{|\textbf{ev}(Q(q))|_p}=|\textbf{ev}(P(q))|_p\leq|P(q)|_m=\frac{|P(q)|_{\mathfrak{m}}}{|Q(q)|_{\mathfrak{m}}}=|x|_{\mathfrak{m}}.$$
Since $\textbf{ev}:\mathbb{Z}[q]_{\mathfrak{m}}\rightarrow\mathbb{Z}_p$ is a homomorphism of rings, it follows from the last inequality that $\textbf{ev}$ is continuous.
\end{proof}
In the following proposition we show that $\mathbb{Z}[q]_{\mathfrak{m}}$ has a Frobenius endomorphism such that $q\mapsto q^p$.
\begin{prop}\label{prop: frob Z[q]3}
	The endomorphism $\textbf{F}_q:\mathbb{Z}[q]_{\mathfrak{m}}\rightarrow\mathbb{Z}[q]_{\mathfrak{m}}$ given by $\textbf{F}_q\left(\frac{P(q)}{Q(q)}\right)=\frac{P(q^p)}{Q(q^p)}$ is a Frobenius endomorphism, and for all $x\in\mathbb{Z}[q]_{\mathfrak{m}}$, $|\textbf{F}_q(x)|_{\mathfrak{m}}\leq|x|_{\mathfrak{m}}$.
\end{prop}

\begin{proof}
	 Let $x=\frac{P(q)}{Q(q)}$ be in $\mathbb{Z}[q]_{\mathfrak{m}}$, where $P(q),Q(q)\in\mathbb{Z}[q]$ and $Q(q)\notin\mathfrak{m}$. Proposition~\ref{frob3} and Remark~\ref{rq: bn def3} imply that  $$|\textbf{F}_q(x)|_{\mathfrak{m}}=\frac{|P(q^{p})|_{\mathfrak{m}}}{|Q(q^{p})|_{\mathfrak{m}}}=|P(q^{p})|_{\mathfrak{m}}\leq|P(q)|_{\mathfrak{m}}=\frac{|P(q)|_{\mathfrak{m}}}{|Q(q)|_{\mathfrak{m}}}=|x|_{\mathfrak{m}}.$$
	From this last inequality we conclude that $\textbf{F}_q$ is continuous because $\textbf{F}_q$ is an endomorphism of rings. Now, we are going to show that for all $x\in\mathbb{Z}[q]_{\mathfrak{m}}$, $|\textbf{F}_q(x)-x^p|_{\mathfrak{m}}<1$. Let $x=\frac{P(q)}{Q(q)}$ be in $\mathbb{Z}[q]_{\mathfrak{m}}$, where $Q(q)\notin\mathfrak{m}$. Then we have $$\textbf{F}_q(x)-x^p=\frac{P(q^{p})Q(q)^p-P(q)^pQ(q^{p})}{Q(q)^pQ(q^{p})}.$$
	As $Q(q)\notin\mathfrak{m}$,  Remark~\ref{rq: bn def3} implies $Q(q^{p})\notin\mathfrak{m}$. Since $\mathfrak{m}$ is a maximal ideal of $\mathbb{Z}[q]$, we obtain $Q(q)^pQ(q^{p})\notin\mathfrak{m}$. Therefore, $|Q(q)^pQ(q^{p})|_{\mathfrak{m}}=1$. But, $$\pi_p\circ\textbf{ev}(P(q^{p})Q(q)^p-P(q)^pQ(q^{p}))=0.$$ Then, according to Remark~\ref{noyau3}, $P(q^{p})Q(q)^p-P(q)^pQ(q^{p})\in\mathfrak{m}$, and for this reason we have $|P(q^{p})Q(q)^p-P(q)^pQ(q^{p})|_{\mathfrak{m}}<1$. So that, $$|\textbf{F}_q(x)-x^p|_{\mathfrak{m}}<1.$$
	Therefore, $ \textbf{F}_q$ is a Frobenius endormorphism.
\end{proof}

\begin{defi}
	The ring $\mathbb{Z}_{q,p}$ is the completion of $\mathbb{Z}[q]_{\mathfrak{m}}$ with respect to the $\mathfrak{m}$-adic norm.
\end{defi}

\begin{rema}\label{rq: prop de Z_q,p3} The ring $\mathbb{Z}_{q,p}$ has the following properties: \begin{enumerate}[label=(\arabic*)]
		\item Every element of $\mathbb{Z}_{q,p}$ has a norm less than or equal to 1 because $\mathbb{Z}_{q,p}$ is the completion of  $\mathbb{Z}[q]_{\mathfrak{m}}$, the $\mathfrak{m}$-adic norm is ultrametric and, according to Remark~\ref{rq: bn def3}, every element of $\mathbb{Z}[q]_{\mathfrak{m}}$ has a norm less than or equal to 1.
		
		\item The ring $\mathbb{Z}_{q,p}$ is a local ring whose residue ring is $\mathbb{F}_p$. Indeed, the ring $\mathbb{Z}_{q,p}$ is local because it is complete with respect to the $\mathfrak{m}$-adic norm. Now, we know from Remark~\ref{rq: bn def3} that the residue ring of $\mathbb{Z}[q]_{\mathfrak{m}}$ is $\mathbb{F}_p$. As $\mathbb{Z}_{q,p}$ is the completion of $\mathbb{Z}[q]_{\mathfrak{m}}$ with respect to the $\mathfrak{m}$-adic norm, then the residue ring of $\mathbb{Z}_{q,p}$ and the residue ring of $\mathbb{Z}[q]_{\mathfrak{m}}$ are the same. Therefore, the residue ring of $\mathbb{Z}_{q,p}$ is $\mathbb{F}_p$.
	\end{enumerate}
\end{rema}

\begin{rema}\label{remarque: extension de ev y F_q}
	It follows from Proposition~\ref{prop: ev con3} that $\textbf{ev}:\mathbb{Z}[q]_{\mathfrak{m}}\rightarrow\mathbb{Z}_{p}$ is Lipschitz continuous. \footnote{Let $(X,d_X)$ and $(Y,d_Y)$ be two metric spaces and let $f:X\rightarrow Y$ be a function. The function f is Lipschitz continuous if there is a real $K\geq0$ such that for all  $x_1$ and $x_2\in X$ we have $d_Y(f(x_1),f(x_2))\leq Kd_X(x_1,x_2)$.  Let $\widetilde{X}$ and $\widetilde{Y}$ be the completion of $X$ and $Y$ respectively.  If the function $f$ is Lipschitz continuous then $f$  extends in a unique way to a continuous function $\widetilde{f}:\widetilde{X}\rightarrow\widetilde{Y}$ given by $\widetilde{f}(x)=\lim\limits_{n\rightarrow\infty}f(x_n)$ where $x=\lim\limits_{n\rightarrow\infty}x_n$.} Then, it extends in a unique way to a continuous homomorphism $\textbf{ev}:\mathbb{Z}_{q,p}\rightarrow\mathbb{Z}_p$. Likewise, it follows from Proposition~\ref{prop: frob Z[q]3} that $\textbf{F}_q:\mathbb{Z}[q]_{\mathfrak{m}}\rightarrow\mathbb{Z}[q]_{\mathfrak{m}}$ is Lipschitz continuous. Then, it extends in a unique way to a continuous endomorphism $\textbf{F}_q:\mathbb{Z}_{q,p}\rightarrow\mathbb{Z}_{q,p}$.
\end{rema}

\begin{prop}\label{prop: frob Z_{q,p} 3}
	The endomorphism $\textbf{F}_q:\mathbb{Z}_{q,p}\rightarrow\mathbb{Z}_{q,p}$ is the unique Frobenius endomorphism of $\mathbb{Z}_{q,p}$ such that $\textbf{F}_q(q)=q^p$. Moreover, for every $a\in\mathbb{Z}_{q,p}$, $|\textbf{F}_q(a)|_{\mathfrak{m}}\leq|a|_{\mathfrak{m}}$ and $\textbf{ev}\circ\textbf{F}_q=\textbf{ev}$.
\end{prop}
We call the endomorphism $\textbf{F}_q:\mathbb{Z}_{q,p}\rightarrow\mathbb{Z}_{q,p}$, \emph{the Frobenius endomorphism of} $\mathbb{Z}_{q,p}$.
\begin{proof}
	From Remark~\ref{remarque: extension de ev y F_q} we know that $ \textbf{F}_q$ is continuous. First, we are going to show that  for every $a\in\mathbb{Z}_{q,p}$, $|\textbf{F}_q(a)|_{\mathfrak{m}}|\leq |a|_{\mathfrak{m}}$. Let $a$ be in $\mathbb{Z}_{q,p}$ non-zero. Then, there is a Cauchy sequence $\{a_n\}_{n\geq0}$ in $\mathbb{Z}[q]_{\mathfrak{m}}$ such that $a$ is the limit of this sequence. As the $\mathfrak{m}$-adic norm is ultrametric and $a$ is non zero, then there exists an integer $N>0$ such that for all  $n\geq N$, $|a|_{\mathfrak{m}}=|a_n|_{\mathfrak{m}}$. If $|\textbf{F}_q(a)|_{\mathfrak{m}}=0$ then $|\textbf{F}_q(a)|_{\mathfrak{m}}|\leq |a|_{\mathfrak{m}}$. Now, if $|\textbf{F}_q(a)|_{\mathfrak{m}}\neq0$, as $\textbf{F}_q(a)$ is the limit of the sequence $\{\textbf{F}_q(a_n)\}_{n\geq0}$ and the $\mathfrak{m}$-adic norm is ultrametric, then there is an integer $M>0$ such that for all $n\geq M$, $|\textbf{F}_q(a)|_{\mathfrak{m}}=|\textbf{F}_q(a_n)|_{\mathfrak{m}}$. From Proposition~\ref{prop: frob Z[q]3}, we know that for all $n\geq0$, $|\textbf{F}_q(a_n)|_{\mathfrak{m}}\leq|a_n|_{\mathfrak{m}}$ since $a_n\in\mathbb{Z}[q]_{\mathfrak{m}}$. For this reason for all $n\geq\max\{N,M\}$,
	\begin{equation}\label{ingelité}
	|\textbf{F}_q(a)|_{\mathfrak{m}}=|\textbf{F}_q(a_n)|_{\mathfrak{m}}\leq|a_n|_{\mathfrak{m}}=|a|_{\mathfrak{m}}.
	\end{equation}
	Second, we are going to show that for every $a\in\mathbb{Z}_{q,p}$, $|\textbf{F}_q(a)-a^p|_{\mathfrak{m}}<1$. Let $a$ be in $\mathbb{Z}_{q,p}$, then there is $x\in\mathbb{Z}[q]_{\mathfrak{m}}$ such that $|a-x|_{\mathfrak{m}}<~1$. Then, $|a^p-x^p|_{\mathfrak{m}}<1$ since $\overline{a}=\overline{x}$ and so, $\overline{a^p}=\overline{x^p}$. Now, from \eqref{ingelité} follows that, $|\textbf{F}_q(a)-\textbf{F}_q(x)|_{\mathfrak{m}}\leq|a-x|_{\mathfrak{m}}<1$. But, $|\textbf{F}_q(x)-x^p|<1$ because $x\in\mathbb{Z}[q]_{\mathfrak{m}}$ and, according to Proposition~\ref{prop: frob Z[q]3}, $\textbf{F}_q:\mathbb{Z}[q]_{\mathfrak{m}}\rightarrow\mathbb{Z}[q]_{\mathfrak{m}}$ is a Frobenius endomorphism. Thus, 
	\begin{align*}
	|\textbf{F}_q(a)-a^p|_{\mathfrak{m}}  = & |\textbf{F}_q(a)-\textbf{F}_q(x)+\textbf{F}_q(x)-x^p+x^p-a^p|_{\mathfrak{m}} \\
	\leq & \max\{|\textbf{F}_q(a)-\textbf{F}_q(x)|_{\mathfrak{m}},|\textbf{F}_q(x)-x^p|_{\mathfrak{m}}, |x^p-a^p|_{\mathfrak{m}}\} & \\ <&1.
	\end{align*}
	For this reason $ \textbf{F}_q$ is a Frobenius endomorphism of $\mathbb{Z}_{q,p}$  such that $q\mapsto q^{p}$. Now, we are going to show that if $\tau$ is a Frobenius endomorphism of $\mathbb{Z}_{q,p}$ such that $\tau(q)=q^{p}$, then $\tau=\textbf{F}_q$. In fact, the restriction of $\tau$ on $\mathbb{Z}[q]_{\mathfrak{m}}$ is given by $\tau\left(\frac{P(q)}{R(q)}\right)=\frac{P(q^p)}{R(q^p)}=\textbf{F}_q\left(\frac{P(q)}{R(q)}\right)$. We know that if $x\in\mathbb{Z}_{q,p}$, then $x=\lim\limits_{n\rightarrow\infty}x_n$, where $x_n\in\mathbb{Z}[q]_{\mathfrak{m}}$ for all $n\geq0$. Therefore, by continuity we have $\tau(x)=\lim\limits_{n\to\infty}\tau(x_n)$. But, $\tau(x_n)=\textbf{F}_{q}(x_n)$ since $x_n\in\mathbb{Z}[q]_{\mathfrak{m}}$ and by continuity we have $\lim\limits_{n\to\infty}\textbf{F}_{q}(x_n)=\textbf{F}_{q}(x).$ For this reason, $\tau(x)=\textbf{F}_q(x)$. Finally, we are going to show that $\textbf{ev}\circ\textbf{F}_q=\textbf{ev}$. Let $\frac{P(q)}{Q(q)}$ be in $\mathbb{Z}[q]_{\mathfrak{m}}$ such that $Q(q)\notin\mathfrak{m}$. Then we have, $$\textbf{ev}\circ\textbf{F}_q\left(\frac{P(q)}{Q(q)}\right)=\textbf{ev}\left(\frac{P(q^p)}{Q(q^p)}\right)=\frac{P(1)}{Q(1)}=\textbf{ev}\left(\frac{P(q)}{Q(q)}\right).$$
	As $\mathbb{Z}_{q,p}$ is the completion of $\mathbb{Z}[q]_{\mathfrak{m}}$ with respect to the $\mathfrak{m}$-adic norm and the functions $ \textbf{F}_q$ and $ \textbf{ev}$ are continuous, then  it follows from this last equality that $\textbf{ev}\circ\textbf{F}_q(x)=\textbf{ev}(x)$ for all $x\in\mathbb{Z}_{q,p}$.
\end{proof}
The ring $\mathbb{Z}_{q,p}[[z]]$ is equipped with the Gauss norm. This norm is defined as follows: $|\sum_{n\geq0}a(n)z^n|_{\mathcal{G}}=sup\{|a(n)|_{\mathfrak{m}}\}_{n\geq0}$. The ring $\mathbb{Z}_{q,p}[[z]]$ is complete for this norm because $\mathbb{Z}_{q,p}$ is complete and its residue ring is the ring $\mathbb{F}_p[[z]]$ since the residue field of $\mathbb{Z}_{q,p}$ is $\mathbb{F}_p$. Now, we are going to construct a subring of $\mathbb{Z}_{q,p}[[z]]$ whose residue ring is contained in $\mathbb{F}_p(z)$. 
\begin{defi}
	Let $S$ be the set of $P(z)\in\mathbb{Z}_{q,p}[z]$ with $|P(0)|_{\mathfrak{m}}=1$\footnote{The set $S$ is a multiplicative set, that is, if $P(z)$ and $Q(z)$ belong to $S$,  then $P(z)Q(z)$ belongs to $S$.}. Let $\mathbb{Z}_{q,p}[z]_{S}$ be the localization of $\mathbb{Z}_{q,p}[z]$ at $S$. The ring $E_{\mathbb{Z},q,p}$ is the completion of $\mathbb{Z}_{q,p}[z]_{S}$ with respect to the Gauss norm\footnote{Note that $\mathbb{Z}_{q,p}[z]_{S}$ is contained in $\mathbb{Z}_{q,p}[[z]]$. For this reason it makes sense to complete $\mathbb{Z}_{q,p}[z]_{S}$ with respect to the Gauss norm.}.
\end{defi}
Let $\textbf{F}:\mathbb{Z}_p[[z]]\rightarrow\mathbb{Z}_p[[z]]$ be the endomorphism given by $\textbf{F}(\sum_{n\geq0}a_nz^n)=\sum_{n\geq0}a_nz^{np}$. The endomorphism $\textbf{F}$ is the unique Frobenius endomorphism of $\mathbb{Z}_p[[z]]$ such that $z\mapsto z^p$. In our next proposition, we state some properties of the ring $E_{\mathbb{Z},q,p}$.

\begin{prop}\label{prop: proprietes de E_q,p3}
	\begin{enumerate}[label=(\arabic*)]
		\item The ring $E_{\mathbb{Z},q,p}$ is contained in $\mathbb{Z}_{q,p}[[z]]$ and its residue ring is contained in $\mathbb{F}_p(z)$.
		\item The ring $E_{\mathbb{Z},q,p}$ is equipped with a continuous endomorphism $\sigma_q$ given by $$\sigma_q\left(\sum_{n\geq0}a_nz^n\right)=\sum_{n\geq0}a_nq^nz^n.$$ It is also equipped with a continuous endomorphism $\delta_q$ given by $$\delta_q(f(z))=\frac{\sigma_q(f)-f}{q-1}.$$
		\item The endomorphism $\textbf{F}_q:E_{\mathbb{Z},q,p}\rightarrow E_{\mathbb{Z},q,p}$ given by $$\textbf{F}_q\left(\sum_{n\geq0}a_nz^n\right)=\sum_{n\geq0}\textbf{F}_q(a_n)z^{np}$$ is a Frobenius endomorphism.
		\item The homomorphism $\textbf{ev}:E_{\mathbb{Z},q,p}\rightarrow E_p\cap\mathbb{Z}_p[[z]]$ given by $$\textbf{ev}\left(\sum_{n\geq0}a_nz^n\right)=\sum_{n\geq0}\textbf{ev}(a_n)z^n$$ is continuous.
		\item For each integer $h>0$, we have $\textbf{ev}\circ\textbf{F}_q^h=\textbf{F}^h\circ\textbf{ev}$. %et, pour tout entier $h\geq0$ et pour tout $\sum_{n\geq0}a_nz^n\in\mathbb{Z}_{q,p}[[z]]$, $$\textbf{F}_q^h\left(\sum_{n\geq0}a_nz^n\right)\equiv\textbf{F}^h\left(\textbf{ev}\left(\sum_{n\geq0}a_nz^n\right)\right)\mod[p]_q\mathbb{Z}_{q,p}[[z]].$$ 
	\end{enumerate}
\end{prop}

We will show Proposition~\ref{prop: proprietes de E_q,p3} at the end of this section.\\
Now, we consider the $q$-difference operator the order $n$
\begin{equation}\label{L_q3}
L_q:=\sigma^n_q+a_{1}(z)\sigma^{n-1}_q+\cdots+a_{n-1}(z)\sigma_q+a_{n}(z)\in E_{\mathbb{Z},q,p}[\sigma_q].
\end{equation}
The companion matrix of $L_q$ is the following matrix \[A_q=\begin{pmatrix}
0 & 1 & 0 & \cdots & 0 \\
0 & 0 & 1 & \cdots & 0\\
\vdots & \vdots & \vdots & \cdots & \vdots\\
0 & 0 & 0 & \cdots & 1\\
-a_n & -a_{n-1} & -a_{n-2} & \cdots & -a_1\\ 
\end{pmatrix}.\]
Let $A_q^{\textbf{F}^h_q}\in\mathcal{M}_n(E_{\mathbb{Z},q,p})$ be the matrix obtained by applying  $h$-times the endomorphism $\textbf{F}_q$ to each entry of $A_q$.

\begin{defi}[\textit{Strong Frobenius structure}]\label{defsffq3}
	Let $L_q$ be in $E_{\mathbb{Z},q,p}[\sigma_q]$ defined as in \eqref{L_q3} and let $A_q$ be the companion matrix of $L_q$. We say that $L_q$ has a strong Frobenius structure with period $h$ if there exists $H_q\in\rm{ GL}_n(E_{\mathbb{Z},q,p})$ such that
	\begin{equation}\label{expressionsff3}
	\sigma_q(H_q)A^{\textbf{F}^h_q}_q=A_qH_q.
	\end{equation}
\end{defi}

Suppose that $L_q$ in $E_{\mathbb{Z},q,p}[\sigma_q]$ has a strong Frobenius structure with period $h$. We say that the matrix $H_q\in\rm{ GL}_n(E_{\mathbb{Z},q,p})$ is a \emph{Frobenius transition matrix} for the operator $L_q$ if $H_q$ satisfies Equality~\eqref{expressionsff3}. In general, we say that the $q$-difference system $\sigma_qX=A_qX$ with $A_q\in\mathcal{ M} _n(E_{\mathbb{Z},q,p})$ has a strong Frobenius structure with period $h$ if there exists a matrix $H_q\in\rm{GL}_n(E_{\mathbb{Z},q,p})$ such that $\sigma_q(H_q)A^{\textbf{F}^h_q}_q=A_qH_q.$\\
Applying induction on $i>0$, we have $\sigma^i_q=\sum_{j=0}^i\binom{i}{j}(q-1)^j\delta^j_q$, where $\delta_q=\frac{\sigma_q-1}{q-1}$. 
Let $L_q$ be in $E_{\mathbb{Z},q,p}[\sigma_q]$ defined as in \eqref{L_q3}. Therefore, if for each $i\in\{1,\ldots,n\}$,  $\sigma_q^i$ is substituted by $\sum_{j=0}^i\binom{i}{j}(q-1)^j\delta^j_q$ into $\frac{1}{(q-1)^n}L_q$ we get the following operator
\begin{equation}\label{delta3}
L_{\delta_q}=\delta^n_q+b_1(z)\delta^{n-1}_q+\cdots+b_{n-1}(z)\delta_q+b_n(z),
\end{equation}
where for each $i\in\{1,\ldots,n\}$, $$b_i(z)=\frac{1}{(q-1)^i}\sum\limits_{k=n-i}^n\binom{k}{n-i}a_{n-k}(z).$$
Now, we consider the system $\delta_qX=C_qX$, where $C_q$ is the companion matrix of $L_{\delta_q}$. We want to define the notion of strong Frobenius structure for this system. Previously, we have defined the notion of strong Frobenius structure for a system of the form $\sigma_qX=A_qX$. As  $\delta_q=\frac{\sigma_q-1}{q-1}$, from $\delta_qX=C_qX$ we obtain $\sigma_qX=((q-1)C_q+Id)X$, where $Id$ is the identity matrix of rang $n$.  We set $B_q=(q-1)C_q+Id$. Then, we say that the system $\delta_qX=C_qX$ has a \emph{strong Frobenius structure} if $B_q\in\mathcal{ M}_n(E_{\mathbb{Z},q,p})$ and if the system $\sigma_qX=B_qX$ has a strong Frobenius structure. The matrix $B_q$ is called \emph{the matrix associated with $L_{\delta_q}$.}

\begin{defi}
	Let $L_{\delta_q}$ be in $E_{\mathbb{Z},q,p}[\delta_q]$ defined as in \eqref{delta3} and let $C_q$ be the companion matrix of $L_{\delta_q}$. The operator $L_{\delta_q}$ has a strong Frobenius structure if the system $\delta_qX=C_qX$ has a strong Frobenius structure.
\end{defi}
We finish this section by showing Proposition~\ref{prop_m_adique3} and Proposition~\ref{prop: proprietes de E_q,p3}. In order to prove Proposition~\ref{prop_m_adique3} we need the following auxiliary lemma.

\begin{lemm}\label{lem_auxiliaure}
	Let $n$ be a positive integer. If $x\in\mathfrak{m}^{n}\cap\mathfrak{m}^{n+1}$ and $x=\sum\limits_{i=0}^{n}a_{i}(q)p^{n-i}(1-q)^i$, then for all $i\in\{0,1,\ldots,n\}$, $a_i(q)\in\mathfrak{m}$.
\end{lemm}

	\begin{proof}
		For all $n\geq1$, the ideal $\mathfrak{m}^n$ is generated by the family $\{p^{n-i}(1-q)^i\}_{0\leq i\leq n}$ because $\mathfrak{m}=(p,1-q)$. We recall that the homomorphism $\textbf{ev}:\mathbb{Z}[q]\rightarrow\mathbb{Z}$ is given by $\textbf{ev}(P(q))=P(1)$. In order to prove our lemma, we are going to apply induction on $n\geq1$. 
		
When $n=1$, according to our hypotheses, $x\in\mathfrak{m}\cap\mathfrak{m}^2$ and $x=a_0(q)p+a_1(q)(1-q)$, where $a_0(q),a_1(q)\in\mathbb{Z}[q]$. We are going to show that $a_0(q),a_1(q)\in\mathfrak{m}$. As $x\in\mathfrak{m}^2$ we have $x=b_0(q)p^2+b_1(q)p(1-q)+b_2(q)(1-q)^2$, where $b_0(q), b_1(q)$, and $b_2(q)$ belong to $\mathbb{Z}[q]$. So that, 
		\begin{equation}\label{egalite}
		a_0(q)p+a_1(q)(1-q)=b_0(q)p^2+b_1(q)p(1-q)+b_2(q)(1-q)^2.
		\end{equation}
		Substituting $q=1$ into this last equality we obtain $a_0(1)=b_0(1)p$. Thus, $\pi_p\circ \textbf{ev}(a_0(q))=0$. According to Remark~\ref{noyau3}, the ideal $\mathfrak{m}$ is the kernel of $\pi_p\circ\textbf{ev}$, then $a_0(q)\in\mathfrak{m}$. For this reason, there are $t_0(q),t_1(q)\in\mathbb{Z}[q]$ such that $a_0(q)=t_0(q)p+t_1(q)(1-q)$ and consequently, we have $a_0(q)p=t_0(q)p^2+t_1(q)p(1-q)$. It follows from \eqref{egalite} that, $$a_1(q)(1-q)=p^2(b_0(q)-t_0(q))+p(1-q)(b_1(q)-t_1(q))+b_2(q)(1-q)^2.$$
		Then, $\textbf{ev}(p^2(b_0(q)-t_0(q))=p^2(b_0(1)-t_0(1))=0.$
		Therefore, 1 is a zero of the polynomial $b_0(q)-t_0(q)$ and so, there is $f_0\in\mathbb{Z}[q]$ such that $b_0(q)-t_0(q)=f_0(q)(1-q)$. Thus, $$a_1(q)=p^2f_0(q)+p(b_1(q)-t_1(q))+b_2(q)(1-q)\in\mathfrak{m}.$$
For this reason, $a_0(q), a_1(q)\in\mathfrak{m}$. \\
Now, we suppose that our claim is true for $n$ and we are going to show that it is also true for $n+1$. Indeed, we take $x\in\mathfrak{m}^{n+1}\cap\mathfrak{m}^{n+2}$ and we assume that $x=\sum\limits_{i=0}^{n+1}a_i(q)p^{n+1-i}(1-q)^i$, where $a_0(z),\ldots,a_{n+1}(q)\in\mathbb{Z}[q]$. We want to show that, for all $i\in\{0,\ldots,n+1\}$, $a_i(q)\in\mathfrak{m}$. As $x\in\mathfrak{m}^{n+2}$ then, $x=\sum\limits_{j=0}^{n+2}b_j(q)p^{n+2-j}(1-q)^j$, where $b_0(z),\ldots,b_{n+1}(q)$ belong to $\mathbb{Z}[q]$. So that,
		\begin{equation}\label{egalitee}
		\sum\limits_{i=0}^{n+1}a_i(q)p^{n+1-i}(1-q)^i=\sum\limits_{j=0}^{n+2}b_j(q)p^{n+2-j}(1-q)^j.
		\end{equation} 
		Substituting $q=1$ into this last equality we obtain $a_0(1)=b_0(1)p$. Hence, $\pi_p\circ\textbf{ev}(a_0(q))=0$. According to Remark~\ref{noyau3}, $\mathfrak{m}$ is the kernel of $\pi_p\circ\textbf{ev}$, then $a_0(q)\in\mathfrak{m}$. Therefore, there are $t_0(q),t_1(q)\in\mathbb{Z}[q]$ such that $a_0(q)=t_0(q)p+t_1(q)(1-q)$ and consequently, we have
		\begin{equation}\label{eq_trans}
		a_0(q)p^{n+1}=t_0(q)p^{n+2}+t_1(q)p^{n+1}(1-q).
		\end{equation}
		  Now, it follows from \eqref{egalitee} that
		$$
		(1-q)\sum_{i=0}^{n} a_{i+1}(q)p^{n-i}(1-q)^i  =
		\left(\sum_{j=0}^{n+2}b_j(q)p^{n+2-j}(1-q)^j\right)-a_0(q)p^{n+1}.$$
		Substituting \eqref{eq_trans} into this last equality, we obtain
		\begin{align*}
		&(1-q)\sum_{i=0}^{n} a_{i+1}(q)p^{n-i}(1-q)^i  \\
		& =  \left(\sum_{j=0}^{n+2}b_j(q)p^{n+2-j}(1-q)^j\right)-t_0(q)p^{n+2}-t_1(q)p^{n+1}(1-q)\\
		& 
		= p^{n+2}(b_0(q)-t_0(q))+p^{n+1}(1-q)(b_1(q)-t_1(q))+\sum_{j=2}^{n+2}b_j(q)p^{n+2-j}(1-q)^j.	
		\end{align*}
		
		For this reason, $\textbf{ev}(p^{n+2}(b_0(q)-t_0(q)))=p^{n+2}(b_0(1)-t_0(1))=0$ and consequently, there is $f_0(q)\in\mathbb{Z}[q]$ such that $b_0(q)-t_0(q)=f_0(q)(1-q)$. Then,
		\begin{equation}\label{eq_final}
		\sum_{i=0}^{n}a_{i+1}(q)p^{n-i}(1-q)^i=p^{n+1}(pf_0(q)+b_1(q)-t_1(q))+\sum_{j=1}^{n+1}b_{j+1}(q)p^{n+1-j}(1-q)^j.
		\end{equation} 
		Now put $y=\sum\limits_{i=0}^{n}a_{i+1}(q)p^{n-i}(1-q)^i$; then $y\in\mathfrak{m}^{n}$.  Furthermore, from Equality~\eqref{eq_final} we obtain $y\in\mathfrak{m}^{n+1}$. So that, $y\in\mathfrak{m}^n\cap\mathfrak{m}^{n+1}$. Therefore, by applying the induction hypothesis to $y$, we obtain $a_1(q),\ldots, a_{n+1}(q)\in\mathfrak{m}$. Therefore, $a_0(q),	 a_1(q),\ldots, a_{n+1}(q)\in\mathfrak{m}$.
	\end{proof}

\begin{proof}[Proof of Proposition~\ref{prop_m_adique3}]
	
	\smallskip

	\emph{Proof of (i).} 
	
	 According to Corollary 10.18 of \cite{atiyah}, we have $\bigcap_{n\geq0}\mathfrak{m}^n=\{0\}$ since $\mathbb{Z}[q]$ is a Noetherian ring  and $\mathfrak{m}\neq(1)$. Thus, $|x|_{\mathfrak{m}}=0$ if and only if, $x=0$.
	
	\smallskip
	
	\emph{Proof of (ii).}

	Suppose that $|x|_{\mathfrak{m}}=1/p^n$ and that $|y|_{\mathfrak{m}}=1/p^m$, where $n$ are $m$ are positive integers such that $x\in\mathfrak{m}^n\setminus\mathfrak{m}^{n+1}$ and $y\in\mathfrak{m}^m\setminus\mathfrak{m}^{m+1}$. If  $n\leq m$, then $\mathfrak{m}^{m}\subset\mathfrak{m}^{n}$ and for this reason $y\in\mathfrak{m}^n$. Therefore, $x+y\in\mathfrak{m}^n$. Then, $|x+y|_{\mathfrak{m}}\leq 1/p^n=\max\{|x|_{\mathfrak{m}},|y|_{\mathfrak{m}}\}.$ In a similar fashion, if $m\leq n$ we obtain $|x+y|_{\mathfrak{m}}\leq 1/p^m=\max\{|x|_{\mathfrak{m}},|y|_{\mathfrak{m}}\}.$
	
	\smallskip
	
	\emph{Proof of (iii).}

	We are going to show that $|xy|_{\mathfrak{m}}=|x|_{\mathfrak{m}}| y|_{\mathfrak{m}}$. Suppose that $|x|_{\mathfrak{m}}=1/p^n$ and that $|y|_{\mathfrak{m}}=1/p^m$, where $n$ and $m$ are non negative integers such that $x\in\mathfrak{m}^n\setminus\mathfrak{m}^{n+1}$ and $y\in~\mathfrak{m}^m\setminus\mathfrak{m}^{m+1}$. To show that $|xy|_{\mathfrak{m}}=|x|_{\mathfrak{m}}| y|_{\mathfrak{m}}$ it is sufficient to prove that $xy\in\mathfrak{m}^{n+m}\setminus\mathfrak{m}^{n+m+1}$. First, if $m=0=n$ then $x\notin\mathfrak{m}$ and $y\notin\mathfrak{m}$. But, $\mathfrak{m}$ is a maximal ideal of $\mathbb{Z}[q]$ and consequently, $xy\notin\mathfrak{m}$. Now, we assume that $m>0$ or $n>0$. As $x\in\mathfrak{m}^n$, we write $x=\sum\limits_{i=0}^{n}a_{i}(q)p^{n-i}(1-q)^i$ and since $x\notin\mathfrak{m}^{n+1}$, there is $i\in\{0,\ldots,n\}$ such that $a_i(q)\notin\mathfrak{m}$. Let $i_0$ be in $\{0,\ldots,n\}$ such that $a_{i_0}(q)\notin\mathfrak{m}$, and for all $i\in\{0,\ldots, i_0-1\}$, $a_i(q)\in\mathfrak{m}$.\\
	In a similar way, as $y\in\mathfrak{m}^m$ we write $y=\sum\limits_{j=0}^{m}b_{j}(q)p^{m-j}(1-q)^j$ and since $y\notin\mathfrak{m}^{m+1}$, there is $j\in\{0,\ldots,m\}$ such that $b_j(q)\notin\mathfrak{m}$. Let $j_0$ be in $\{0,\ldots,m\}$ such that $b_{j_0}(q)\notin\mathfrak{m}$, and for all $j\in\{1,\ldots,j_0-1\}$, $b_j(q)\in\mathfrak{m}$. Note that $$xy=\sum_{k=0}^{n+m}\left(\sum_{i+j=k}a_i(q)b_j(q)\right)p^{n+m-k}(1-q)^k\in\mathfrak{m}^{n+m}.$$ 
	To see that $xy\notin\mathfrak{m}^{n+m+1}$, assume for contradiction that $xy\in\mathfrak{m}^{n+m+1}$. It is clear that $xy\in\mathfrak{m}^{n+m}\cap\mathfrak{m}^{n+m+1}$. But,  $n+m>0$ since $n>0$ or $m>0$. Then, it follows from Lemma~\ref{lem_auxiliaure} that $\sum_{i+j=k}a_i(q)b_j(q)\in\mathfrak{m}$ for $k\in\{0,\ldots,n+m\}$. In particular,
	\begin{multline*}
	a_0(q)b_{i_0+j_0}(q)+\cdots+a_{i_{0}-1}(q)b_{j_{0}+1}(q)+a_{i_0}(q)b_{j_0}(q)\\
	+a_{i_{0}+1}(q)b_{j_0-1}(q)+\cdots+a_{i_0+j_0}(q)b_{0}(q)\in\mathfrak{m}.
	\end{multline*}
	As $a_i(q)\in\mathfrak{m}$ for $i<i_0$ then $a_0(q)b_{i_0+j_0}(q)+\cdots+a_{i_{0}-1}(q)b_{j_{0}+1}(q)\in\mathfrak{m}$. In a similar way, since  $b_j(q)\in\mathfrak{m}$ for $j<j_0$ then $a_{i_{0}+1}(q)b_{j_0-1}(q)+\cdots+a_{i_0+j_0}(q)b_{0}(q)\in\mathfrak{m}$. Therefore, $a_{i_0}(q)b_{j_0}(q)\in\mathfrak{m}$. According to Remark~\ref{noyau3}, the ideal $\mathfrak{m}$ is maximal and thus, $a_{i_0}(q)\in\mathfrak{m}$ or $b_{j_0}(q)\in\mathfrak{m}$. This is a contradiction of the fact that $a_{i_0}(q)$ and $b_{j_0}(q)$ do not belong to $\mathfrak{m}$.
	Therefore, $xy\notin\mathfrak{m}^{n+m+1}$. Then we have $|xy|_{\mathfrak{m}}=1/p^{n+m}=|x|_{\mathfrak{m}}|y|_{\mathfrak{m}}.$
\end{proof}

\begin{proof}[Proof of Proposition~\ref{prop: proprietes de E_q,p3}]

\smallskip
	
We recall that $S=\{P(z)\in\mathbb{Z}_{q,p}[z]:|P(0)|_{\mathfrak{m}}=1\}$. 
	
\smallskip
	
	\emph{Proof of (1).} 
	
	Let $x$ be an element of $E_{\mathbb{Z},q,p}$. By construction of the ring $E_{\mathbb{Z},q,p}$, there is a Cauchy sequence $\{x_n\}_{n\geq0}$ in $\mathbb{Z}_{q,p}[z]_S$ such that  $x=\lim_{n\to\infty}x_n$. For each integer $n>0$, we write $x_n=P_n(z)/Q_n(z)$ with $P_n(z)$, $Q_n(z)\in\mathbb{Z}_{q,p}[z]$ and $Q_n(z)\in S$. By definition of the set $S$, $|Q_n(0)|_{\mathfrak{m}}=1$. Since $\mathbb{Z}_{q,p}$ is a local ring (see Remark~\ref{rq: prop de Z_q,p3}) we obtain $1/Q_n(0)\in\mathbb{Z}_{q,p}$ and so, $1/Q_n(z)\in\mathbb{Z}_{q,p}[[z]]$. Thus, for each integer $n>0$, $x_n=P_n(z)/Q_n(z)\in\mathbb{Z}_{q,p}[[z]]$. As $\mathbb{Z}_{q,p}[[z]]$ is complete with respect to the Gauss norm, then $x\in \mathbb{Z}_{q,p}[[z]]$. Therefore, $E_{\mathbb{Z},q,p}\subset\mathbb{Z}_{q,p}[[z]]$.\\
	Now, we show that the residue ring of $E_{\mathbb{Z},q,p}$ is contained in $\mathbb{F}_p(z)$. Recall that for $a\in\mathbb{Z}_{q,p}$ we denote by $\overline{a}$  the reduction of $a$ modulo the maximal ideal of $\mathbb{Z}_{q,p}$. Since 	$E_{\mathbb{Z},q,p}$ is the completion of $\mathbb{Z}_{q,p}[z]_{S}$ with respect to the Gauss norm, the residue ring of $E_{\mathbb{Z},q,p}$ and the residue ring of $\mathbb{Z}_{q,p}[z]_{S}$ are the same. For this reason, it is sufficient to show that the residue ring of $\mathbb{Z}_{q,p}[z]_{S}$ is contained in $\mathbb{F}_p(z)$. Let us consider $\xi:\mathbb{Z}_{q,p}[z]_{S}\rightarrow\mathbb{F}_p(z)\cap\mathbb{F}_p[[z]]$ given by $$\xi\left(\frac{\sum_{i=0}^na_iz^i}{\sum_{j=0}^mb_jz^j}\right)=\frac{\sum_{i=0}^n\overline{a_i}z^i}{\sum_{j=0}^m\overline{b_j}z^j}.$$
	According to Remark~\ref{rq: prop de Z_q,p3}, the residue ring of $\mathbb{Z}_{q,p}$ is $\mathbb{F}_p$, and again by Remark~\ref{rq: prop de Z_q,p3}, if $a\in\mathbb{Z}_{q,p}$ then $|a|_{\mathfrak{m}}\leq1$ and for this reason $\overline{a}\in\mathbb{F}_p$. By definition, if $\sum_{j=0}^mb_jz^j\in S$ then $|b_0|_{\mathfrak{m}}=1$ and so, $\overline{b_0}\neq0$. Therefore, $1/\sum_{j=0}^m\overline{b_j}z^j\in\mathbb{F}_p[[z]]$.  From the above arguments we get that $\xi$ is well defined. Moreover, $\xi$ is a homomorphism of rings. The maximal ideal of $\mathbb{Z}_{q,p}[z]_{S}$ is the set of elements of $\mathbb{Z}_{q,p}[z]_{S}$ with norm less than 1. Now, $\xi\left(x_1(z)/x_2(z)\right)=0$ if and only if, $\left|x_1(z)/x_2(z)\right|_{\mathcal{G}}<1$, and this is a equivalent to saying that $x_1(z)/x_2(z)$ belongs to maximal ideal of $\mathbb{Z}_{q,p}[z]_{S}$. So, if $A$ is the residue ring of $\mathbb{Z}_{q,p}[z]_{S}$ we have the injective homomorphism $\overline{\xi}:A\rightarrow\mathbb{F}_p(z)\cap\mathbb{F}_p[[z]]$ given by $\overline{\xi}\left(\overline{x_1(z)/x_2(z)}\right)=\xi\left(x_1(z)/x_2(z)\right)$. Therefore, the residue ring of $E_{\mathbb{Z},q,p}$ is contained in $\mathbb{F}_p(z)\cap\mathbb{F}_p[[z]]$.
	
	\medskip
	
	\emph{Proof of (2).} 
	
	Consider $\sigma_q:\mathbb{Z}_{q,p}[[z]]\rightarrow\mathbb{Z}_{q,p}[[z]]$ given by $\sigma_q(\sum_{n\geq0}a_nz^n)=\sum_{n\geq0}a_nq^nz^n$. Firstly,  $\sigma_q$ is a homomorphism of rings and it is continuous for the Gauss norm because $|q|_{\mathfrak{m}}=1$. Now, we are going to show that $\sigma_q(E_{\mathbb{Z},q,p})\subset E_{\mathbb{Z},q,p}$. In oder to show this, it is sufficient to prove that $\sigma_q(\mathbb{Z}_{q,p}[z]_S)\subset\mathbb{Z}_{q,p}[z]_S$ because $\sigma_q$ is a continuous homomorphism for the Gauss norm and $E_{\mathbb{Z},q,p}$ is the completion of $\mathbb{Z}_{q,p}[z]_S$ with respect to the Gauss norm. Indeed, it is clear that if $Q(z)\in S$ then $Q(qz)\in S$, and thus if $\frac{P(z)}{Q(z)}\in\mathbb{Z}_{q,p}[z]_{S}$ then $\sigma_q\left(\frac{P(z)}{Q(z)}\right)=\frac{P(qz)}{Q(qz)}\in\mathbb{Z}_{q,p}[z]_S$. Therefore, $\sigma_q(\mathbb{Z}_{q,p}[z]_S)\subset\mathbb{Z}_{q,p}[z]_S$. Consequently, the homomorphism $\sigma_q:E_{\mathbb{Z},q,p}\rightarrow E_{\mathbb{Z},q,p}$ given by $\sigma_q(\sum_{n\geq0}a_nz^n)=\sum_{n\geq0}a_nq^nz^n$ is continuous.\\	
	Secondly, let us consider $\delta_q:\mathbb{Z}_{q,p}[[z]]\rightarrow\mathbb{Z}_{q,p}[[z]]$ given by $\delta_q(f)=\frac{\sigma_q(f)-f}{q-1}$. Show that $\delta_q$ is well defined. That is, we want to show that for each $\sum_{n\geq0} a_nz^n\in\mathbb{Z}_{q,p}[[z]]$, $\delta_q(\sum_{n\geq0} a_nz^n)\in\mathbb{Z}_{q,p}[[z]]$. In fact, we have
	$$\delta_q\left(\sum_{n\geq0}a_nz^n\right)=\sum_{n\geq0}\frac{q^n-1}{q-1}a_nz^n=\sum_{n\geq1}(1+q+\cdots+q^{n-1})a_nz^n.$$
Consequently, $\delta_q:\mathbb{Z}_{q,p}[[z]]\rightarrow\mathbb{Z}_{q,p}[[z]]$ is well defined. Moreover, the endomorphism $\delta_q$ is continuous for the Gauss norm because, for every integer $n>0$, $|1+q+\cdots+q^{n-1}|_{\mathfrak{m}}\leq1$ and, in consequence, for each $\sum_{n\geq0}a_nz^n\in\mathbb{Z}_{q,p}[[z]]$, $$\left|\delta_q\left(\sum_{n\geq0}a_nz^n\right)\right|_{\mathcal{G}}\leq\left|\sum_{n\geq0}a_nz^n\right|_{\mathcal{G}}.$$
	Since $\delta_q$ preserves addition, it follows from the last inequality that $\delta_q$ is continuous.\\
	 Now, we are going to show that $\delta_q(E_{\mathbb{Z},q,p})\subset\ E_{\mathbb{Z},q,p}$. In order to get this, it is sufficient to prove that  $\delta_q(\mathbb{Z}_{q,p}[z]_S)\subset\mathbb{Z}_{q,p}[z]_S$ because $\delta_q$ is continuous for the Gauss norm and $E_{\mathbb{Z},q,p}$ is the completion of $\mathbb{Z}_{q,p}[z]_S$ with respect to the Gauss norm. Indeed, if $x(z)=\frac{\sum_{i=0}^na_iz^i}{\sum_{j=0}^mb_jz^j}\in\mathbb{Z}_{q,p}[z]_S$, then
	\begin{align*}
	\delta_q(x(z))=\frac{x(qz)-x(z)}{q-1}&=\frac{1}{q-1}\frac{\sum_{i=0}^na_iq^iz^i\sum_{j=0}^mb_jz^j-\sum_{i=0}^na_iz^i\sum_{j=0}^mb_jq^jz^j}{\sum_{i=0}^mb_jz^j\sum_{j=0}^mb_jq^jz^j}\\
	&=\frac{1}{q-1}\frac{\sum_{k=0}^{n+m}\left(\sum_{i+j=k}a_ib_j(q^i-q^j)\right)z^k}{\sum_{i=0}^mb_jz^j\sum_{j=0}^mb_jq^jz^j}\\
	&=\sum_{k=0}^{n+m}\left(\sum_{i+j=k}a_ib_j\frac{(q^i-q^j)}{q-1}\right)z^k\cdot\frac{1}{\sum_{i=0}^mb_jz^j\sum_{j=0}^mb_jq^jz^j}.
	\end{align*}
	On the one hand, $\sum_{k=0}^{n+m}\left(\sum_{i+j=k}a_ib_j\frac{(q^i-q^j)}{q-1}\right)z^k$ belongs to $\mathbb{Z}_{q,p}[z]$ because, for every $(i,j)~\in\mathbb{Z}^2_{\geq0}$, $\frac{(q^i-q^j)}{q-1}$ belongs to $\mathbb{Z}[q]$. On the other hand, since $S$ is a multiplicative set and the polynomials $\sum_{i=0}^nb_jz^j, \sum_{i=0}^nb_jq^jz^j$ belong to  S, it follows that $\sum_{i=0}^nb_jz^j\sum_{j=0}^mb_jq^jz^j\in S$. Then, $\delta_q(x(z))\in\mathbb{Z}_{q,p}[z]_S$. Thus, $\delta_q(\mathbb{Z}_{q,p}[z]_S)\subset\mathbb{Z}_{q,p}[z]_S$. \\
	Finally, since $E_{\mathbb{Z},q,p}$ is the completion of $\mathbb{Z}_{q,p}[z]_S$ with respect to the Gauss norm and $\delta_q:\mathbb{Z}_{q,p}[z]_S\rightarrow\mathbb{Z}_{q,p}[z]_S$ is continuous, we deduce that the endomorphism  $\delta_q:E_{\mathbb{Z},q,p}\rightarrow E_{\mathbb{Z},q,p}$ given by $\delta_q(f)=\frac{\sigma_q(f)-f}{q-1}$ is continuous.
	
	\medskip
	
	\emph{Proof of (3).} 
	
	Let $\textbf{F}_q$ be the endomorphism of $\mathbb{Z}_{q,p}[[z]]$ given by $\textbf{F}_q(\sum_{n\geq0}a_nz^n)=\sum_{n\geq0}\textbf{F}_q(a_n)z^{np}$, where $\textbf{F}_q:\mathbb{Z}_{q,p}\rightarrow\mathbb{Z}_{q,p}$ is the Frobenius endomorphism of $\mathbb{Z}_{q,p}$.\\
	We are going to show that $\textbf{F}_q:\mathbb{Z}_{q,p}[[z]]\rightarrow\mathbb{Z}_{q,p}[[z]]$ is a Frobenius endomorphism. From Proposition~\ref{prop: frob Z_{q,p} 3}, we know that for each $a\in\mathbb{Z}_{q,p}$, $|\textbf{F}_q(a)|_{\mathfrak{m}}\leq|a|_{\mathfrak{m}}$. Then, for every $\sum_{n\geq0}a_nz^n\in\mathbb{Z}_{q,p}[[z]]$, we have $$\left|\textbf{F}_q\left(\sum_{n\geq0}a_nz^n\right)\right|_{\mathcal{G}}=\sup\left\{|\textbf{F}_q(a_n)|_{\mathfrak{m}}\right\}_{n\geq0}\leq\sup\{|a_n|_{\mathfrak{m}}\}_{n\geq0}=\left|\sum_{n\geq0}a_nz^n\right|_{\mathcal{G}}.$$
	Since $\textbf{F}_q:\mathbb{Z}_{q,p}[[z]]\rightarrow\mathbb{Z}_{q,p}[[z]]$ is an endomorphism of rings, this last inequality implies that $\textbf{F}_q:\mathbb{Z}_{q,p}[[z]]\rightarrow\mathbb{Z}_{q,p}[[z]]$ is continuous.\\
	Now, we are going to prove that, for every $\sum_{n\geq0}a_nz^n\in\mathbb{Z}_{q,p}[[z]]$, $$\left|\textbf{F}_q\left(\sum_{n\geq0}a_nz^n\right)-\left(\sum_{n\geq0}a_nz^n\right)^p\right|_{\mathcal{G}}<1.$$ Let $\mathfrak{m}_{\mathbb{Z}_{q,p}}$ be the maximal ideal of $\mathbb{Z}_{q,p}$. As $\textbf{F}_q:\mathbb{Z}_{q,p}\rightarrow\mathbb{Z}_{q,p}$ is the Frobenius endomorphism of $\mathbb{Z}_{q,p}$ and, by Remark~\ref{rq: prop de Z_q,p3},  $\mathbb{Z}_{q,p}/\mathfrak{m}_{\mathbb{Z}_{q,p}}=\mathbb{F}_p$,  then we have
	\begin{align*}
	\textbf{F}_q\left(\sum_{n\geq0}a_nz^n\right)&=\sum_{n\geq0}\textbf{F}_q(a_n)z^{np}\\
	&\equiv\sum_{n\geq0}a_n^pz^{np}\mod\mathfrak{m}_{\mathbb{Z}_{q,p}}\mathbb{Z}_{q,p}[[z]]\\
	&\equiv\left(\sum_{n\geq0}a_nz^{n}\right)^p\mod\mathfrak{m}_{\mathbb{Z}_{q,p}}\mathbb{Z}_{q,p}[[z]].
	\end{align*}
	Since the $\mathfrak{m}$-adic norm is ultrametric, we easily check that for each $a\in\mathbb{Z}_{q,p}$, $a\in\mathfrak{m}_{\mathbb{Z}_{q,p}}$ if and only if $|a|_{\mathfrak{m}}\leq1/p$. Consequently,  $$\left|\textbf{F}_q\left(\sum_{n\geq0}a_nz^n\right)-\left(\sum_{n\geq0}a_nz^n\right)^p\right|_{\mathcal{G}}\leq1/p<1.$$
	Now, we show that $\textbf{F}_q(E_{\mathbb{Z},q,p})\subset\ E_{\mathbb{Z},q,p}$. To show this, it is sufficient to prove that $\textbf{F}_q(\mathbb{Z}_{q,p}[z]_S)\subset\mathbb{Z}_{q,p}[z]_S$ because $\textbf{F}_q$ is continuous for the Gauss norm and $E_{\mathbb{Z},q,p}$ is the completion of $\mathbb{Z}_{q,p}[z]_S$ with respect to the Gauss norm. Indeed, let $P(z)/Q(z)$ be in $\mathbb{Z}_{q,p}[z]_S$, where $P(z)\in\mathbb{Z}_{q,p}[z]$ and $Q(z)\in S$. By definition of the set $S$, $|Q_n(0)|_{\mathfrak{m}}=1$ and since $\mathbb{Z}_{q,p}$ is a local complete ring, $1/Q_n(0)\in\mathbb{Z}_{q,p}$ and so, $1/Q_n(z)\in\mathbb{Z}_{q,p}[[z]]$. As $\textbf{F}_q:\mathbb{Z}_{q,p}[[z]]\rightarrow\mathbb{Z}_{q,p}[[z]]$ is an endomorphism of rings then $\textbf{F}_q\left(\frac{P(z)}{Q(z)}\right)=\frac{\textbf{F}_q(P(z))}{\textbf{F}_q(Q(z))}$. It is clear that $\textbf{F}_q(P(z))\in\mathbb{Z}_{q,p}[z]$. Now, we are going to see that $\textbf{F}_q(Q(z))\in S$. We write $Q(z)=\sum_{j=0}^mb_jz^j$. We have $|b_0|_{\mathfrak{m}}=1$ since $Q(z)\in S$. By definition, $\textbf{F}_q(Q(z))=\sum_{j=0}^m\textbf{F}_q(b_n)z^{np}$.  Since $\textbf{F}_q:\mathbb{Z}_{q,p}\rightarrow\mathbb{Z}_{q,p}$ is the Frobenius endomorphism of $\mathbb{Z}_{q,p}$, $|\textbf{F}_q(b_0)-b_0^p|_{\mathfrak{m}}<1$ and as the $\mathfrak{m}$-adic norm is ultrametric, and $|b_0^p|_{\mathfrak{m}}=1$, then we have $$|\textbf{F}_q(b_0)|_{\mathfrak{m}}=|\textbf{F}_q(b_0)-b_0^p+b_0^p|_{\mathfrak{m}}=\max\{|\textbf{F}_q(b_0)-b_0^p|_{\mathfrak{m}},|b_0^p|_{\mathfrak{m}}\}=1.$$
	Therefore, $\textbf{F}_q(Q(z))\in S$ and $\textbf{F}_q\left(\frac{P(z)}{Q(z)}\right)\in\mathbb{Z}_{q,p}[z]_S$. For this reason, we have $\textbf{F}_q(\mathbb{Z}_{q,p}[z]_S)\subset\mathbb{Z}_{q,p}[z]_S$. Finally, we conclude that the endomorphism $\textbf{F}_q:E_{\mathbb{Z},q,p}\rightarrow E_{\mathbb{Z},q,p}$ given by $$\textbf{F}_q\left(\sum_{n\geq0}a_nz^n\right)=\sum_{n\geq0}\textbf{F}_q(a_n)z^{np}$$ is a Frobenius endomorphism of $E_{\mathbb{Z},q,p}$.
	
	\medskip
	
	\emph{Proof of (4).} 
	
	We consider $\textbf{ev}:\mathbb{Z}_{q,p}[[z]]\rightarrow\mathbb{Z}_p[[z]]$ given by $\textbf{ev}(\sum_{n\geq0}a_nz^n)=\sum_{n\geq0}\textbf{ev}(a_n)z^n$, where $\textbf{ev}:\mathbb{Z}_{q,p}\rightarrow\mathbb{Z}_p$ is the continuous homomorphism constructed in Remark~\ref{remarque: extension de ev y F_q}. For this reason, the homomorphism  $\textbf{ev}:\mathbb{Z}_{q,p}[[z]]\rightarrow\mathbb{Z}_p[[z]]$ is continuous.  Now, we are going to show that $\textbf{ev}(\mathbb{Z}_{q,p}[z]_S)\subset E_p\cap\mathbb{Z}_p[[z]]$. Let $P(z)/Q(z)$ be in $\mathbb{Z}_{q,p}[z]_S$, where $P(z)\in\mathbb{Z}_{q,p}[z]$ and $Q(z)\in S$. We write $Q(z)=\sum_{j=0}^mb_jz^j$. Since $Q(z)\in S$, $|b_0|_{\mathfrak{m}}=1$. Recall that $\mathbb{Z}_{q,p}$ is the completion of $\mathbb{Z}[q]_{\mathfrak{m}}$ with respect to the $\mathfrak{m}$-adic norm. Then $b_0=\lim_{n\to\infty}x_n$, where $x_n\in\mathbb{Z}[q]_{\mathfrak{m}}$. In particular, for $\epsilon>0$ sufficiently small, there is an integer $N>0$ such that, for all $n\geq N$, $|b_0-x_n|_{\mathfrak{m}}<\epsilon$. As the norm is ultrmetric and $|b_0|_{\mathfrak{m}}=1$, then for all $n\geq N$, $|x_n|_{\mathfrak{m}}=1$. We write $x_n=R_n(q)/T_n(q)\in\mathbb{Z}[q]_{\mathfrak{m}}$, where $R_n(q)$, $T_n(q)\in\mathbb{Z}[q]$ and $T_n(q)\notin\mathfrak{m}$. Since, for all $n\geq N$, $|x_n|_{\mathfrak{m}}=1$ and $|T_n(q)|_{\mathfrak{m}}=1$, we have $|R_n(q)|_{\mathfrak{m}}=1$. Therefore, for all $n\geq N$, $R_n(q)$ and $T_n(q)$ do not belong to $\mathfrak{m}$. Then, it follows from Remark~\ref{noyau3} that, for all $n\geq N$, $\textbf{ev}(R_n(q))$ and $\textbf{ev}(T_n(q))$ do not belong to $p\mathbb{Z}$. For this reason, for all $n\geq N$, $|\textbf{ev}(x_n)|_p=1$. By construction of $\textbf{ev}:\mathbb{Z}_{q,p}\rightarrow\mathbb{Z}_p$ we have $\textbf{ev}(b_0)=\lim_{n\to\infty}\textbf{ev}(x_n)$. Since the norm is utrametric and, for all $n\geq N$, $|\textbf{ev}(x_n)|_p=1$, then $|\textbf{ev}(b_0)|_{p}=1$. Thus, $1/\textbf{ev}(Q(z))\in\mathbb{Z}_p[[z]]$. Finally, it is clear that $\textbf{ev}(P(z))\in\mathbb{Z}_p[z]$. Thus, $\textbf{ev}(P(z))/\textbf{ev}(Q(z))\in E_p$ because $\mathbb{Q}_p(z)\subset E_p$ and $\textbf{ev}(P(z))/\textbf{ev}(Q(z))\in\mathbb{Q}_p(z)$. Then, we have $$\textbf{ev}\left(\frac{P(z)}{Q(z)}\right)=\frac{\textbf{ev}(P(z))}{\textbf{ev}(Q(z))}\in E_p\cap\mathbb{Z}_p[[z]].$$
	Therefore, the homomorphism $\textbf{ev}:\mathbb{Z}_{q,p}[z]_S\rightarrow E_p\cap\mathbb{Z}_p[[z]]$ is well defined. Moreover, it is continuous since $\textbf{ev}:\mathbb{Z}_{q,p}\rightarrow\mathbb{Z}_p$ is continuous. Then, by continuity, $\textbf{ev}:\mathbb{Z}_{q,p}[z]_S\rightarrow E_p\cap\mathbb{Z}_p[[z]]$ extends to the continuous homomorphism $\textbf{ev}:E_{\mathbb{Z},q,p}\rightarrow E_p\cap\mathbb{Z}_p[[z]]$.
	
	\medskip
	
	\emph{Poroof of (5).} 
	
	Consider the continuous endomorphism  $\textbf{F}_q:\mathbb{Z}_{q,p}\rightarrow\mathbb{Z}_{q,p}$ and the continuous homomorphism $\textbf{ev}:\mathbb{Z}_{q,p}\rightarrow\mathbb{Z}_p$. By Proposition~\ref{prop: frob Z_{q,p} 3},  for each $a\in\mathbb{Z}_{q,p}$, we have
	\begin{equation}\label{eq: ev com F= ev}
	\textbf{ev}(\textbf{F}_q(a))=\textbf{ev}(a).
	\end{equation}
	Consider the continuous endomorphism $\textbf{F}_q:\mathbb{Z}_{q,p}[[z]]\rightarrow\mathbb{Z}_{q,p}[[z]]$ given by $\textbf{F}_q(\sum_{n\geq0}a_nz^n)=\sum_{n\geq0}\textbf{F}_q(a_n)z^{np}$ and the continuous  homomorphism $\textbf{ev}:E_{\mathbb{Z}, q,p}\rightarrow E_p\cap\mathbb{Z}_p[[z]]$ given by $\textbf{ev}(\sum_{n\geq0}a_nz^n)=\sum_{n\geq0}\textbf{ev}(a_n)z^n$. For each integer $h>0$, we want to see that for these last two homomorphisms we have $\textbf{ev}\circ\textbf{F}_q^h=\textbf{F}^h\circ\textbf{ev}$, where $\textbf{F}:\mathbb{Z}_p[[z]]\rightarrow \mathbb{Z}_p[[z]]$ verifies $\textbf{F}(z)=z^p$. Let $\sum_{n\geq0}a_nz^n$ be in $\mathbb{Z}_{q,p}$[[z]]. Then, we have $$\textbf{ev}\left(\textbf{F}_q^h\left(\sum_{n\geq0}a_nz^n\right)\right)=\sum_{n\geq0}\textbf{ev}(\textbf{F}^h_q(a_n))z^{np^{h}}$$ and $$\textbf{F}^h\left(\textbf{ev}\left(\sum_{n\geq0}a_nz^n\right)\right)=\sum_{n\geq0}\textbf{ev}(a_n)z^{np^h}.$$ So, $\textbf{ev}\left(\textbf{F}_q^h\left(\sum_{n\geq0}a_nz^n\right)\right)=\textbf{F}^h\left(\textbf{ev}\left(\sum_{n\geq0}a_nz^n\right)\right)$ if only if, for every integer $n\geq0$, $\textbf{ev}(\textbf{F}_q^h(a_n))=\textbf{ev}(a_n)$. But from \eqref{eq: ev com F= ev} we get that, for every integer $n\geq0$, $\textbf{ev}(\textbf{F}_q^h(a_n))=\textbf{ev}(a_n)$ and consequently, $\textbf{ev}\left(\textbf{F}_q^h\left(\sum_{n\geq0}a_nz^n\right)\right)=\textbf{F}^h\left(\textbf{ev}\left(\sum_{n\geq0}a_nz^n\right)\right)$. 
	
\end{proof}

\section{Proof of  Theorem~\ref{confluence3} and Theorem~\ref{alg3}}\label{sec_proof}

In this section we prove Theorem~\ref{confluence3} and Theorem~\ref{alg3} .  We start by proving Theorem~\ref{confluence3}.

\begin{proof}[Proof of Theorem~\ref{confluence3}]
	 We are going to show that the differential operator $$\textbf{ev}(L_{\delta_q}):=\delta^n+\textbf{ev}(b_1(z))\delta^{n-1}+\cdots+\textbf{ev}(b_1(z))\delta+\textbf{ev}(b_n(z))\in E_p[\delta]$$
	has a strong Frobenius structure with period $h$. By definition, the matrix associated  with $L_{\delta_q}$ is
	\[B_q=\begin{pmatrix}
	1 & (q-1) & \cdots & 0 \\
	0 & 1 & \ddots & 0\\
	\vdots & \vdots & \ddots & \vdots\\
	0 & 0 & \cdots & (q-1)\\
	-(q-1)b_n(z) & -(q-1)b_{n-1}(z) & \cdots & -(q-1)b_1(z)+1\end{pmatrix}.\]
	Then,  
	\[
	\frac{B_q-I_n}{q-1}=\begin{pmatrix}
	0 & 1 & \cdots & 0 \\
	0 & 0 & \ddots & 0\\
	\vdots & \vdots & \ddots & \vdots\\
	0 & 0 & \cdots & 1\\
	-b_n(z) & -b_{n-1}(z) & \cdots & -b_1(z)\end{pmatrix}.\]
	So that, $B:=\textbf{ev}\left(\frac{B_q-Id_n}{q-1}\right)$ is the companion matrix of the differential operator  $\textbf{ev}(L_{\delta_q})$.  Note that $B$ is a matrix with coefficients in $E_p\cap\mathbb{Z}_p[[z]]$.
	By hypotheses,  there is $G_q\in\rm{ GL}_n(E_{\mathbb{Z,}q,p})$ such that
	\begin{equation}\label{depart}
	\sigma_q(G_q)B^{\textbf{F}^h_q}_{q}=B_qG_q.
	\end{equation}	 
	 It follows from Equality~\eqref{depart} that
	\begin{align*}
	\sigma_q(G_q)B^{\textbf{F}^h_q}_{q}=B_qG_q&\iff\sigma_q(G_q)B^{\textbf{F}^h_q}_q-I_n=B_qG_q-I_n\\
	&\iff\sigma_q(G_q)B^{\textbf{F}^h_q}_q-\sigma_q(G_q)+\sigma_q(G_q)-I_n=B_qG_q-G_q+G_q-I_n\\
	&\iff\sigma_q(G_q)[B^{\textbf{F}^h_q}_q-I_n]+\sigma_q(G_q)-Id_n=[B_q-Id_n]G_q+G_q-I_n\\
	&\iff\sigma_q(G_q)\left(\frac{B^{\textbf{F}^h_q}_q-I_n}{q-1}\right)+\frac{\sigma_q(G_q)-I_n}{q-1}=\left(\frac{B_q-I_n}{q-1}\right)G_q+\frac{G_q-I_n}{q-1}\\
	&\iff\frac{\sigma_q(G_q)-I_n}{q-1}-\frac{G_q-I_n}{q-1}=\left(\frac{B_q-I_n}{q-1}\right)G_q-\sigma_q(G_q)\left(\frac{B^{\textbf{F}^h_q}_q-I_n}{q-1}\right)\\
	\end{align*}
From this last equality we obtain
\begin{equation}\label{d3}
\frac{\sigma_q(G_q)-G_q}{q-1}=\left(\frac{B_q-I_n}{q-1}\right)G_q-\sigma_q(G_q)\left(\frac{B^{\textbf{F}^h_q}_q-I_n}{q-1}\right).
\end{equation}	
	For  $\sum_{i\geq0}a_iz^i\in\mathbb{Z}_{q,p}[[z]]$ we have $$\textbf{ev}\left(\sigma_q\left(\sum_{i\geq0}a_iz^i\right)\right)=\textbf{ev}\left(\sum_{i\geq0}a_iq^iz^i\right)=\sum_{i\geq0}\textbf{ev}(a_iq^i)z^i=\sum_{i\geq0}\textbf{ev}(a_i)\textbf{ev}(q^i)z^i.$$
But $\textbf{ev}(q)=1$ and for this reason 
	$$\textbf{ev}\left(\sigma_q\left(\sum_{i\geq0}a_iz^i\right)\right)=\sum_{i\geq0}\textbf{ev}(a_i)z^i=\textbf{ev}\left(\sum_{i\geq0}a_iz^i\right).$$  
	 Therefore, $\textbf{ev}(\sigma_q(G_q))=\textbf{ev}(G_q)$. Now, if $\sum_{i\geq0}a_iz^i$ belongs to $\mathbb{Z}_{q,p}[[z]]$ , then
	\begin{align*}
	\frac{\sigma_q(\sum_{i\geq0}a_iz^i)-\sum_{i\geq0}a_iz^i}{q-1}=&\frac{\sum_{i\geq0}a_iq^iz^i-\sum_{i\geq0}a_iz^i}{q-1}\\=&\frac{\sum_{i\geq1}(q^i-1)a_iz^i}{q-1}\\
	=&\sum_{i\geq1}(1+q\cdots+q^{i-1})a_iz^i\in\mathbb{Z}_{q,p}[[z]].
	\end{align*}
	As for all $i\geq1$, $\textbf{ev}(1+q\cdots+q^{i-1})=i$ then, $$\textbf{ev}\left(\frac{\sigma_q(\sum_{i\geq0}a_iz^i)-\sum_{i\geq0}a_iz^i}{q-1}\right)=\sum_{i\geq1}i\textbf{ev}(a_i)z^i=\delta\textbf{ev}\left(\sum_{i\geq0}a_iz^i\right),$$
	where $\delta=z\frac{d}{dz}$. For this reason, $\textbf{ev}\left(\frac{\sigma_q(G_q)-G_q}{q-1}\right)=\delta\textbf{ev}(G_q).$
	Now,  from the definition of the endomorphism $\textbf{F}_q$ we have
	\[\frac{B^{\textbf{F}^h_q}_q-I_n}{q-1}=\begin{pmatrix}
	0 & \frac{(q^{p^h}-1)}{q-1} & \cdots & 0 \\
	0 & 0 & \ddots & 0\\
	\vdots & \vdots & \ddots & \vdots\\
	0 & 0 & \cdots & \frac{(q^{p^h}-1)}{q-1}\\
	-\frac{(q^{p^h}-1)}{q-1}\textbf{F}_q(b_n(z)) & -\frac{(q^{p^h}-1)}{q-1}\textbf{F}_q(b_{n-1}(z)) & \cdots & -\frac{(q^{p^h}-1)}{q-1}\textbf{F}_q(b_1(z))\end{pmatrix}.\]
	We know that $\textbf{ev}\left(\frac{q^{p^h}-1}{q-1}\right)=p^h$.  For every $i\in\{1,\ldots,n\}$,  put $b_i(z)=\sum_{j\geq0}b_{j,i}z^{j}$.  So, $\textbf{F}^h_q(\sum_{j\geq0}b_{j,i}z^{j})=\sum_{j\geq0}\textbf{F}_q^h(b_{j,i})z^{jp^h}$.  Then,  according to the point 5 of Proposition~\ref{prop: proprietes de E_q,p3}, $$\textbf{ev}\left(\textbf{F}_q^h\left(\sum_{j\geq0}b_{j,i}z^{j}\right)\right)=\textbf{ev}\left(\sum_{j\geq0}b_{j,i}z^{jp^h}\right).$$ Thus,
	$$\textbf{ev}\left(\frac{B_q^{\textbf{F}_q^h}-I_n}{q-1}\right)=p^h\textbf{ev}\left(\frac{B_q-I_n}{q-1}\right)(z^{p^h})=p^hB(z^{p^h})=p^hB^{\textbf{F}^h},$$
	where $\textbf{F}:\mathbb{Z}_p[[z]]\rightarrow\mathbb{Z}_p[[z]]$  is the Frobenius endomorphism of $\mathbb{Z}_p[[z]]$. Then, by applying  $\textbf{ev}$ to  \eqref{d3}, we obtain
	\begin{equation}\label{ff}
	\delta(\textbf{ev}(G_q))=B\textbf{ev}(G_q)-\textbf{ev}(G_q)p^hB^{\textbf{F}^h}.
	\end{equation} 
	As $G_q\in\rm{ GL}_n(E_{\mathbb{Z},q,p})$, then $\textbf{ev}(G_q)\in\rm{ GL}_n(E_p)$. For this reason, we get from Equality~\eqref{ff}  that the differential operator $\textbf{ev}(L_{\delta_q})$ has a strong Frobenius structure with period  $h$.
\end{proof}

In order to prove Theorem~\ref{alg3} we need some auxiliary propositions.  

\subsection{Auxiliary propositions}

 Let $L_q$ be in $E_{\mathbb{Z},q,p}[\sigma_q]$  of order $n$ and let $A_q$ be the companion matrix of $L_q$. We set $$N(A_q)=\{\vec{y}\in\mathbb{Z}_{q,p}[[z]]^n\quad\textup{such that}\quad\sigma_q\vec{y}=A_q\vec{y}\}.$$ 
 The set $N(A_q)$ is naturally a $\mathbb{Z}_{q,p}$-module. Recall that $\textbf{F}_q:\mathbb{Z}_{q,p}[[z]]\rightarrow\mathbb{Z}_{q,p}[[z]]$ is  the endomorphism given by $\textbf{F}_q(\sum_{i\geq0}a_iz^i)=\sum_{i\geq0}\textbf{F}_q(a_i)z^{ip}$.
In the following proposition we prove that the set $N(A_q)$ comes equipped with a  Frobenius action. 
\begin{prop}\label{actionfrob3}
	Let $L_q$ be in $E_{\mathbb{Z},q,p}[\sigma_q]$ of order $n$ and let $A_q$ be the companion matrix of $L_q$.  Suppose that $L_q$ has a strong Frobenius structure with period $h$ and let $H_q$ be a Frobenius transition matrix for $L_q$. Then the application $\psi:N(A_q)\rightarrow N(A_q)$ given by $$\psi((f_1,\ldots,f_n)^t)=H_q(\textbf{F}_q^h(f_1),\ldots,\textbf{F}_q^h(f_n))^t$$ 
	is a $\textbf{F}_q^h$-semilinear endomorphism.
\end{prop}

\begin{proof}
As $L_q$ has a strong Frobenius structure and $H_q\in\rm{ GL}_{n}(E_{\mathbb{Z},q,p})$ is a Frobenius transition matrix for $L_q$, we have
	\begin{equation}\label{sff3}
	\sigma_q(H_q)A^{\textbf{F}^h_q}_q=A_qH_q.
	\end{equation}
	Let us consider $\psi:N(A_q)\rightarrow N(A_q)$ given by $\psi((f_1,\ldots,f_n)^t)=H_q((\textbf{F}_q^h(f_1),\ldots,\textbf{F}_q^h(f_n))^t$.  We are going to show that $\psi$ is well defined. That is, we prove that if $(f_1,\cdots,f_n)^{t}\in N(A_q)$ then $H_q(\textbf{F}_q^h(f_1),\cdots,\textbf{F}_q^h(f_n))^{t}\in N(A_q)$. Indeed, since $(f_1,\cdots,f_n)^{t}$ is a solution of $\sigma_qy=A_qy$,  we get \begin{equation}\label{systeme3}
	\begin{pmatrix}
	\sigma_qf_{1}\\
	\sigma_qf_2\\
	\vdots\\
	\sigma_qf_{n}\\
	\end{pmatrix}
	=
	A_q
	\begin{pmatrix}
	f_1\\
	f_2\\
	\vdots\\
	f_n\\
	\end{pmatrix}.
	\end{equation}
	Note that the restriction of $\textbf{F}_q:\mathbb{Z}_{q,p}[[z]]\rightarrow\mathbb{Z}_{q,p}[[z]]$ on $E_{\mathbb{Z},q,p}$ is the  Frobenius endomorphism of $E_{\mathbb{Z},q,p}$  constructed in the point 3 of Proposition ~\ref{prop: proprietes de E_q,p3}. Hence, by applying $\textbf{F}_q^h$ to Equality~\eqref{systeme3} we have \begin{equation}\label{systeme13}
	\begin{pmatrix}
	\textbf{F}_q^h(\sigma_qf_{1})\\
	\vdots\\
	\textbf{F}_q^h(\sigma_qf_{n})\\
	\end{pmatrix}
	=
	A^{\textbf{F}^h_q}_q
	\begin{pmatrix}
	\textbf{F}_q^h(f_1)\\
	\vdots\\
	\textbf{F}_q^h(f_n)\\
	\end{pmatrix}.
	\end{equation}
	Now put $f_j=\sum_{i\geq0}a_{i,j}z^i$; then  $\textbf{F}_q^h(f_j)=\sum_{i\geq0}\textbf{F}_q^h(a_{i,j})z^{ip^h}$ and so we have
	\begin{align*}
	\textbf{F}_q^h(\sigma_qf_j)=\textbf{F}_q^h\left(\sum_{i\geq0}a_{i,j}q^iz^i\right)=&\sum_{i\geq0}\textbf{F}_q^h(a_{i,j})\textbf{F}_q^h(q^i)z^{ip^h}\\=&\sum_{i\geq0}\textbf{F}_q^h(a_{i,j})q^{ip^h}z^{ip^h}\\=&\sigma_q(\textbf{F}_q^h(f_j)).
	\end{align*}
	Then, from Equality~\eqref{systeme13} we obtain
	\begin{equation*}
	\begin{pmatrix}
	\sigma_q(\textbf{F}_q^h(f_{1}))\\
	\vdots\\
	\sigma_q(\textbf{F}_q^h(f_n))\\
	\end{pmatrix}
	=
	A_q^{\textbf{F}^h_q}
	\begin{pmatrix}
	\textbf{F}_q^h(f_1)\\
	\vdots\\
	\textbf{F}_q^h(f_n)\\
	\end{pmatrix}.
	\end{equation*}
	Consequently, $(\textbf{F}_q^h(f_1),\ldots,\textbf{F}_q^h(f_n))^t$ is a solution of the system $\sigma_qX=A_q^{\textbf{F}^h_q}X.$
	Then, it follows from~\eqref{sff3} that
	\begin{equation*}
	H_q
	\begin{pmatrix}
	\textbf{F}_q^h(f_1)\\
	\vdots\\
	\textbf{F}_q^h(f_n)\\
	\end{pmatrix}
	\end{equation*}
	is a solution of the system $\sigma_qX=A_qX.$ As $H_q\in\rm{ GL}_n(E_{\mathbb{Z},q,p})$ and,  according to the point 1 of Proposition~\ref{prop: proprietes de E_q,p3},  $E_{\mathbb{Z},q,p}\subset\mathbb{Z}_{q,p}[[z]]$, then this solution belongs to  $N(A_q)$. Therefore, $\psi$ is well defined.  Furthemore, from the definition of $\psi$ we easily see that for all $c\in \mathbb{Z}_{q,p}$ and for all $(f_1,\ldots, f_n)^t\in N(A_q)$, $\psi(c\cdot(f_1,\ldots, f_n)^t)=\textbf{F}_q^h(c)\cdot\psi((f_1,\ldots, f_n)^t)$.
	
Finally, it is clear that for all $(f_1,\ldots, f_n)^t, (g_1,\ldots, g_n)^t\in N(A_q)$, $$\psi((f_1,\ldots, f_n)^t + (g_1,\ldots, g_n)^t)=\psi((f_1,\ldots, f_n)^t)+ \psi( (g_1,\ldots, g_n)^t).$$
\end{proof}
For each $q$-difference operator $L_q\in Frac(\mathbb{Z}_{q,p})[[z]][\sigma_q]$, we set  $$Ker(L_q, Frac(\mathbb{Z}_{q,p})[[z]])=\{y\in Frac(\mathbb{Z}_{q,p})[[z]]\quad\text{such that}\quad L_qy=0\}.$$  

\begin{prop}\label{dimensionker3}
	Let $L_q$ be in $Frac(\mathbb{Z}_{q,p})[[z]][\sigma_q]$ a $q$-difference operator of order $n$.  The set $Ker(L_q, Frac(\mathbb{Z}_{q,p})[[z]])$ is a vector space over the field $Frac(\mathbb{Z}_{q,p})$  and its  dimension is less than or equal to $n$.
\end{prop}
This proposition is a particular case  from Proposition 1.2.1 of \cite{lucia}  however, for the sake of completeness  we prove it. 
\begin{proof} 
	It is clear that if $y_1$ and $y_2$ belong to $Ker(L_q,Frac(\mathbb{Z}_{q,p})[[z]])$ then $y_1+y_2$ belongs to $Ker(L_q,Frac(\mathbb{Z}_{q,p})[[z]])$.  Moreover, for all $c\in Frac(\mathbb{Z}_{q,p})$, we have $ \sigma_q(c)=c$. Then, it follows that $Ker(L_q,Frac(\mathbb{Z}_{q,p})[[z]])$ is a vector space over the field $Frac(\mathbb{Z}_{q,p})$. Now, consider $C=\{f\in Frac(\mathbb{Z}_{q,p})((z))\text{ : }\sigma_q(f)=f\}$. Note that $C$ is a ring. As $q$ is not  solution of any polynomial over $\mathbb{Z}$, we have $C=Frac(\mathbb{Z}_{q,p})$. For this reason, $Ker(L_q, Frac(\mathbb{Z}_{q,p})[[z]])$ is a vector space over $C$.  We are going to prove that  $dim_CKer(L_q, Frac(\mathbb{Z}_{q,p})[[z]])\leq n$. Let $f_1,\ldots, f_{n+1}$ be in $Ker(L_q, Frac(\mathbb{Z}_{q,p})[[z]])$. We write $$L_q=\sigma^n_q+a_1(z)\sigma^{n-1}_q+\cdots+a_{n-1}(z)\sigma_q+a_n(z)\in Frac(\mathbb{Z}_{q,p})[[z]][\sigma_q].$$
	Then, for every $i\in\{1,\ldots,n+1\}$, we have $$\sigma^{n}_q(f_i)+a_1(z)(\sigma^{n-1}_qf_i)+\cdots+a_{n-1}(z)(\sigma_qf)+a_n(z)f_i=0.$$
	
	Therefore, the vectors \[y_1=\begin{pmatrix}
	f_1\\
	\sigma_qf_1\\
	\vdots\\
	\sigma^n_qf_{1}
	\end{pmatrix},\, y_2=\begin{pmatrix}
	f_2\\
	\sigma_qf_2\\
	\vdots\\
	\sigma^n_qf_2
	\end{pmatrix},\ldots,\,y_{n+1}=\begin{pmatrix}
	f_{n+1}\\
	\sigma_qf_{n+1}\\
	\vdots\\
	\sigma^n_qf_{n+1}
	\end{pmatrix}\] 
	are linearly dependent over $Frac(\mathbb{Z}_{q,p})((z))$ since $Frac(\mathbb{Z}_{q,p})[[z]]\subset Frac(\mathbb{Z}_{q,p})((z))$.   Let $r$ be the rang of the matrix whose columns are the vectors $y_1,\ldots, y_{n+1}$. Then, we have $r\leq n$. So that, there are $v_1,\ldots,v_r\in\{y_1,\ldots,y_{n+1}\}$ linearly independent over $Frac(\mathbb{Z}_{q,p})((z))$ such that $y_1,\ldots, y_{n+1}$ belong to the vector space generated  by $v_1,\ldots, v_r$. Without loss of generality we can assume that $v_1=y_1,\ldots, y_r=v_r$. Then, for each $j\in\{r+1,\ldots,n+1\}$, there are $c_{1,j},\ldots, c_{r,j}\in Frac(\mathbb{Z}_{q,p})((z))$ such that $y_{j}=c_{1,j}y_1+\cdots+c_{r,j}y_{r}$. Consequently, for all $j\in\{r+1,\ldots,n+1\}$, we have
	\begin{equation}\label{ma}
	\begin{pmatrix}
	f_1 & \cdots& f_r\\
	\sigma_qf_1 & \cdots& \sigma_qf_r\\
	\vdots & \cdots & \vdots\\
	\sigma^n_qf_1 & \cdots & \sigma^n_qf_r
	\end{pmatrix}\begin{pmatrix}
	c_{1,j}\\
	c_{2,j}\\
	\vdots\\
	c_{r,j}
	\end{pmatrix}=\begin{pmatrix}
	f_{j}\\
	\sigma_qf_{j}\\
	\vdots\\
	\sigma^n_qf_{j}
	\end{pmatrix}.
	\end{equation}
	We are going to show that, for all $j\in\{r+1,\ldots,n+1\}$, $c_{1,j},\ldots, c_{r,j}$ belong to $C$. Indeed, by applying $\sigma_q$ to \eqref{ma} we obtain \begin{equation}\label{sigma}\begin{pmatrix}
	\sigma_qf_1 & \cdots& \sigma_qf_r\\
	\sigma^2_qf_1 & \cdots& \sigma^2_qf_r\\
	\vdots & \cdots & \vdots\\
	\sigma^{n+1}_qf_1 & \cdots & \sigma^{n+1}_qf_r
	\end{pmatrix}\begin{pmatrix}
	\sigma_qc_{1,j}\\
	\sigma_qc_{2,j}\\
	\vdots\\
	\sigma_qc_{r,j}
	\end{pmatrix}=\begin{pmatrix}
	\sigma_qf_{j}\\
	\sigma^2_qf_{j}\\
	\vdots\\
	\sigma^{n+1}_qf_{j}
	\end{pmatrix}.\end{equation}
	It follows from \eqref{ma} and \eqref{sigma} that \[\begin{pmatrix}
	\sigma_qf_1 & \cdots& \sigma_qf_r\\
	\sigma^2_qf_1 & \cdots& \sigma^2_qf_r\\
	\vdots & \cdots & \vdots\\
	\sigma^{n}_qf_1 & \cdots & \sigma^{n}_qf_r
	\end{pmatrix}\begin{pmatrix}
	\sigma_qc_{1,j}-c_{1,j}\\
	\sigma_qc_{2,j}-c_{2,j}\\
	\vdots\\
	\sigma_qc_{r,j}-c_{r,j}
	\end{pmatrix}=\begin{pmatrix}
	0\\
	0\\
	\vdots\\
	0
	\end{pmatrix}.\] 
	But the rang of the matrix 
	\[\begin{pmatrix}
	\sigma_qf_1 & \cdots& \sigma_qf_r\\
	\sigma^2_qf_1 & \cdots& \sigma^2_qf_r\\
	\vdots & \cdots & \vdots\\
	\sigma^{n}_qf_1 & \cdots & \sigma^{n}_qf_r
	\end{pmatrix}\]
 is $r$ since the vectors $v_1,\ldots,v_r$ are linearly independent over $Frac(\mathbb{Z}_{q,p})((z))$. Therefore, for all $j\in\{r+1,\ldots,n+1\}$ and for all $i\in\{1,\ldots,r\}$, $\sigma_qc_{i,j}=c_{i,j}$.  Then $c_{i,j}$ belongs to $C$ for $j\in\{r+1,\ldots,n+1\}$ and $i\in\{1,\ldots,r\}$. Thus, the vectors $y_1,\ldots, y_{n+1}$ are linearly dependent over $C$ and for this reason, $f_1,\ldots, f_{n+1}$ are linearly dependent over $C$.  Consequently,  $dim_{Frac(\mathbb{Z}_{q,p})}Ker(L_q, Frac(\mathbb{Z}_{q,p})[[z]])\leq n$ since $C=Frac(\mathbb{Z}_{q,p})$
\end{proof}

We can now prove Theorem~\ref{alg3}.

\begin{proof}[of Theorem~\ref{alg3}]
According to Proposition~\ref{dimensionker3}, the set $Ker(L_q,Frac(\mathbb{Z}_{q,p})[[z]])$ is a vector space over $Frac(\mathbb{Z}_{q,p})$ and $dim_{Frac(\mathbb{Z}_{q,p})}Ker(L_q, Frac(\mathbb{Z}_{q,p})[[z]])\leq n$.  Let $A_q$ be the companion matrix of $L_q$ and we set $$\widehat{N(A_q)}=\{\vec{y}\in Frac(\mathbb{Z}_{q,p})[[z]]^n\quad\text{such that}\quad\sigma_q\vec{y}=A_q\vec{y} \}.$$ 

The set $\widehat{N(A_q)}$ is naturally  a $Frac(\mathbb{Z}_{q,p})$-vector space. Recall that $N(A_q)$ is the set of solutions of the systme $\sigma_qX=A_qX$ that belong to $\mathbb{Z}_{q,p}[[z]]^n$. It is clear that $N(A_q)\subset\widehat{N(A_q)}$. %We provide the set  $\widehat{N(A_q)}$ of the operations  "+" and "$\cdot$"  wich are defined as follows: let  $\vec{v}_1=(y_1,\ldots, y_n)^t$, $\vec{v}_2=(w_1,\ldots,w_n)^t$ be in $\widehat{N(A_q)}$ and $c$ be in $Frac(\mathbb{Z}_{q,p})$.  We set $\vec{v}_1+\vec{v}_2:=(y_1+w_1,\ldots,y_n+w_n)^t$ and $c\cdot\vec{v}_1:=(cy_1,\ldots,cy_n)^t$.  

The sets $Ker(L_q,Frac(\mathbb{Z}_{q,p})[[z]])$ and $\widehat{N(A_q)}$ are isomorphic as $Frac(\mathbb{Z}_{q,p})$-vector spaces since the homomorphism $\widehat{N(A_q)}\rightarrow Ker(L_q,Frac(\mathbb{Z}_{q,p})[[z]])$ given by $(f_1,\ldots, f_n)^t\mapsto f_1$ is an isomorphism. Thus, $dim_{Frac(\mathbb{Z}_{q,p})}\widehat{N(A_q)}\leq n$.  

Let $H_q$ be a Frobenius transition matrix for the operator $L_q$. By Proposition~\ref{actionfrob3}, the application $\psi:N(A_q)\rightarrow N(A_q)$ given by $\psi((f_1,\ldots, f_n)^t)=H_q((\textbf{F}_q^h(f_1)),\ldots,\textbf{F}_q^h(f_n))^t$ is a $ \textbf{F}_q^h$-semilinear endomorphism. Now put $\textbf{f}=(f,\sigma_q(f),\ldots,\sigma^{n-1}_qf)^t$. By hypothesis $f$ is a solution of $L_q$ and $f\in\mathbb{Z}_{q,p}[[z]]$. Then $\textbf{f}\in N(A_q)$ and therefore, for all integers $i>0$, $\psi^i(\textbf{f})\in N(A_q)$. Since $N(A_q)\subset\widehat{N(A_q)}$ and $dim_{Frac(\mathbb{Z}_{q,p})}\widehat{N(A_q)}\leq n$, there are $r_1,\ldots, r_n\in Frac(\mathbb{Z}_{q,p})$ not all zero, such that $$\psi^n(\textbf{f})+r_1\psi^{n-1}(\textbf{f})+\cdots+r_{n-1}\psi(\textbf{f})+r_n\textbf{f}=0.$$
Then, there exists $s_0\in\mathbb{Z}_{q,p}$ non-zero and there are $s_1,\ldots, s_n\in\mathbb{Z}_{q,p}$ not all zero, such that \begin{equation}\label{combf}
s_0\psi^n(\textbf{f})+s_1\psi^{n-1}(\textbf{f})+\cdots+s_{n-1}\psi(\textbf{f})+s_n\textbf{f}=0.
\end{equation}
Let $Z$ be the $Frac(E_{\mathbb{Z},q,p})$-vector space generated by the elements of the following set $\{\textbf{F}_q^{jh}(\sigma^i_q(f)):j\in\mathbb{Z}_{\geq0}, i\in\{0,1,\ldots,n-1\}\}$. It follows from \eqref{combf} that $Z$, as $Frac(E_{\mathbb{Z},q,p})$-vector space, is generated by $\{\textbf{F}_q^{rh}(\sigma^i_q(f)): r,i\in\{0,1,\ldots,n-1\}\}$ and therefore, we have $dim_{Frac(E_{\mathbb{Z},q,p})}Z\leq n^2$.\\ 
Since $f, \textbf{F}_q^h(f),\ldots, \textbf{F}_q^{n^2h}(f)\in Z$, there exist $b_0(q,z),b_1(q,z),\ldots, b_{n^2}(q,z)\in E_{\mathbb{Z},q,p}$ not all zero, such that $$b_0(q,z)f+b_1(q,z)\textbf{F}_q^h(f)+\cdots+b_{n^2}(q,z)\textbf{F}_q^{n^2h}(f)=0.$$
Let $i_0$ be in $\{0,\ldots, n^2\}$ such that $|b_{i_0}(q,z)|_{\mathcal{G}}=\max\{|b_0(q,z)|_{\mathcal{G}},\ldots,|b_{n^2}(q,z)|_{\mathcal{G}}\}$. As $b_{i_0}(q,z)\in E_{\mathbb{Z},q,p}$ is not zero and the norm is ultrametric, then there is $P(z)/Q(z)\in\mathbb{Z}_{q,p}[z]_S$ such that $|b_{i_0}(q,z)|_{\mathcal{G}}=|P(z)/Q(z)|_{\mathcal{G}}$.  But $|P(z)/Q(z)|_{\mathcal{G}}=|P(z)|_{\mathcal{G}}$ because $Q(z)\in S$.  Let $c$ be in $\mathbb{Z}_{q,p}$ the coefficient of $P(z)$ such that $|P(z)|_{\mathcal{G}}=|c|_{\mathfrak{m}}$. Therefore, $|c|_{\mathfrak{m}}=\max\{|b_0(q,z)|_{\mathcal{G}},\ldots,|b_{n^2}(q,z)|_{\mathcal{G}}\}$. Note that $c$ is not zero since $b_0(q,z),\ldots, b_{n^2}(q,z)$ are not all zero.\\
For each $i\in\{0,1,\ldots,n^2\}$,  we set $d_i(q,z)=\frac{b_i(q,z)}{c}$. So, $1=\max\{|d_0(q,z)|_{\mathcal{G}},\ldots, |d_{n^2}(q,z)|_{\mathcal{G}}\}$ and
\begin{equation}\label{falg}
d_0(q,z)f+d_1(q,z)\textbf{F}_q^h(f)+\cdots+d_{n^2}(q,z)\textbf{F}_q^{n^2h}(f)=0.
\end{equation}
From point 1 of Proposition~\ref{prop: proprietes de E_q,p3},  we know that $E_{\mathbb{Z},q,p}\subset\mathbb{Z}_{q,p}[[z]]$. Hence, $b_0(q,z),b_1(q,z),\ldots$,  $b_{n^2}(q,z)\in\mathbb{Z}_{q,p}[[z]]$ and $d_0(q,z),d_1(z),\ldots,d_{n^2}(q,z)\in Frac(\mathbb{Z}_{q,p})[[z]]$.  

As, for all $i\in\{0,1,\ldots,n^2\}$, $|d_i(q,z)|_{\mathcal{G}}\leq1$, then for all $i\in\{0,1,\ldots,n^2\},$ $d_i(q,z)\in\vartheta_{Frac(\mathbb{Z}_{q,p})}[[z]]$. Thus, we get from \eqref{falg} that $$d_0(q,z)f+d_1(q,z)\textbf{F}_q^h(f)+\cdots+d_{n^2}(q,z)\textbf{F}_q^{n^2h}(f)\equiv0\bmod\phi_p(q)\vartheta_{Frac(\mathbb{Z}_{q,p})}[[z]].$$ 
 Finally, the polynomial $d_0(q,z)Y_0+\cdots+d_{n^2}(q,z)Y_{n^2}\bmod\phi_p(q)\vartheta_{Frac(\mathbb{Z}_{q,p})}$ is not zero because $1=\max\{|d_1(q,z)|_{\mathcal{G}},\ldots, |d_{n^2}(q,z)|_{\mathcal{G}}\}$.
\end{proof}

\begin{rema}\label{frac3}
	Let $k$ be the residue ring of $\vartheta_{Frac(\mathbb{Z}_{q,p})}$.  Note that $k$ is a field because $Frac(\mathbb{Z}_{q,p})$ is a field. Let $\mathfrak{n}$ be the maximal ideal of $\vartheta_{Frac(\mathbb{Z}_{q,p})}$. We have just shown that under the hypotheses of Theorem~\ref{alg3}, there are $d_0(z),\ldots,d_{n^2}(z)\in\vartheta_{Frac(\mathbb{Z}_{q,p})}[[z]]$ and there is a non-zero constant $c\in\mathbb{Z}_{q,p}$  such that  the polynomial $d_0(q,z)Y_0+d_1(z)Y_1+\cdots+d_{n^2}(q,z)Y_{n^2}\mod[p]_q\vartheta_{Frac(\mathbb{Z}_{q,p})}$ is not zero, $cd_0(q,z),\ldots, cd_{n^2}(q,z)\in E_{\mathbb{Z},q,p}$, and
	$$
	d_0(q,z)f(z)+d_1(q,z)\textbf{F}^h_q(f)+\cdots+d_{n^2}(q,z)\textbf{F}^{n^2h}_q(f)\equiv0\bmod\phi_p(q)\vartheta_{Frac(\mathbb{Z}_{q,p})}[[z]].$$  For every $i\in\{0,1,\ldots,n^2\}$,  let $d_i|_{\mathfrak{n}}(q,z)$ be the power series over $k$ obtained after reducing each coefficient of $d_i(q,z)$ modulo the ideal $\mathfrak{n}$.  In this remark we are going to show that for all $i\in\{0,1,\ldots, n^2\}$,  $d_i|_{\mathfrak{n}}(q,z)\in k(z)$. Indeed, as $b_i(q,z)\in E_{\mathbb{Z},q,p}$ then $b_i(q,z)=\lim_{n\to\infty}x_n(z)$, where $x_n(z)\in\mathbb{Z}_{q,p}[z]_S$ for all $n\geq0$. So, $d_i(q,z)=\lim_{n\to\infty}\frac{x_n(z)}{c}$ and therefore, there exists $N$ such that $|d_i(q,z)-\frac{x_N(z)}{c}|_{\mathcal{G}}<1$.  Since the norm is ultrametric and $|d_i(q,z)|_{\mathcal{G}}\leq1$, we have $\left|\frac{x_N(z)}{c}\right|_{\mathcal{G}}\leq1$.  We write $$x_N(z)=\frac{\sum_{i=0}^la_{N,i}z^i}{\sum_{j=0}^tb_{N,j}z^j},$$ where $\sum_{i=0}^la_{N,i}z^i\in\mathbb{Z}_{q,p}[z]$ and $\sum_{j=0}^tb_{N,j}z^j\in S$. Let $i_0$ be in $\{0,\ldots, l\}$ such that $|a_{N,i_0}|_{\mathfrak{m}}=\max\{|a_{N,0}|_\mathfrak{m},\ldots,|a_{N,l}|_\mathfrak{m}\}$. 
	
	As $\sum_{j=0}^tb_{N,j}z^j\in S$ then  $1=|b_{N,0}|_{\mathfrak{m}}=\max\{|b_{N,0}|_\mathfrak{m},\ldots,|b_{N,t}|_\mathfrak{m}\}$ and therefore, $$x_N(z)=\frac{a_{N,i_0}}{b_{N,0}}\frac{P(z)}{Q(z)},\quad\textup{where}\quad P(z)=\sum_{i=0}^l\frac{a_{N,i}}{a_{N,i_0}}z^i,\,\ Q(z)=\sum_{j=0}^t\frac{b_{N,j}}{b_{N,0}}z^j.$$
	Note that $|P(z)|_{\mathcal{G}}=1$ and $|Q(z)|_{\mathcal{G}}=1$. Thus, $P(z),Q(z)\in\vartheta_{Frac(\mathbb{Z}_{q,p})}[z]$. 
	Since $|b_{N,0}|_{\mathfrak{m}}=1$, $\frac{1}{b_{N,0}}\in\mathbb{Z}_{q,p}$ because $\mathbb{Z}_{q,p}$ is a complete local ring and for this reason, $\frac{a_{N,i_0}}{b_{N,0}}\in~\mathbb{Z}_{q,p}$. 
	As $\frac{X_{N}(z)}{c}=\frac{a_{N,i_0}}{cb_{N,0}}\frac{P(z)}{Q(z)}$ and $\left|P(z)/Q(z)\right|_{\mathcal{G}}=1$ and, moreover $\left|x_N(z)/c\right|_{\mathcal{G}}\leq1$, then $\left|a_{N,i_0}/cb_{N,0}\right|_{\mathfrak{m}}\leq1$. Whence, $a_{N,i_0}/cb_{N,0}\in~\vartheta_{Frac(\mathbb{Z}_{q,p})}$.  Therefore, we can reduce each coefficient of $\frac{X_N(z)}{c}$ modulo the maximal ideal $\mathfrak{n}$.  We let $\frac{X_N(z)}{c}\big|_{\mathfrak{n}}$ denote the power series obtained after reducing each coefficient of $\frac{X_N(z)}{c}$  modulo the maximal ideal $\mathfrak{n}$; then $$\frac{X_N(z)}{c}\bigg|_{\mathfrak{n}}=\left(\frac{a_{N,i_0}}{cb_{N,0}}\mod\mathfrak{n}\right)\frac{\sum_{i=0}^l(\frac{a_{N,i}}{a_{N,i_0}}\mod\mathfrak{n})z^i}{\sum_{j=0}^t(\frac{b_{N,j}}{b_{N,0}}\mod\mathfrak{n})z^j}\in k(z).$$ Since  $|d_i(q,z)-\frac{x_N(z)}{c}|_{\mathcal{G}}<1$,  $d_{i\mid\mathfrak{n}}(q,z)=\frac{X_N(z)}{c}\big|_{\mathfrak{n}}$ and so, $d_i|_{\mathfrak{n}}(q,z)\in k(z)$. 
	
	\smallskip
	
	Finally, $d_0|_{\mathfrak{n}}(q,z),\ldots,d_{n^2}|_{\mathfrak{n}}(q,z)$ are not all zero because $1=\max\{|d_i(q,z)|_{\mathcal{G}}\}_{0\leq i\leq n^2}$.

\end{rema}

\subsection{Recovering congruences modulo $p$ from congruences modulo $\phi_p(q)$}\label{sub_recuperando}
In this section we are going to explain how to use Theorem~\ref{alg3} for recovering congruences modulo $p$ from congruences modulo cyclotomic polynomial $\phi_p(q)$.\\
Let $f(q,z)$ be in $\mathbb{Z}[q][[z]]$. Then, $\textbf{F}_q(f(q,z))=f(q^p,z^p)$ and $\textbf{ev}(f(q,z))=f(1,z)$. Suppose that $f(q,z)\in\mathbb{Z}[q][[z]]$ is a solution of a $q$-difference operator $L_q$ of order $n$ having a strong Frobenius structure with period $h$.  Therefore, by Theorem~\ref{alg3}, there are $d_0(q,z),\ldots,d_{n^2}(q,z)\in\vartheta_{Frac(\mathbb{Z}_{q,p})}[[z]]$ and there is a non-zero constant $c\in\mathbb{Z}_{q,p}$ such that the polynomial  $d_0(q,z)Y_0+d_1(z)Y_1+\cdots+d_{n^2}(q,z)Y_{n^2}\mod\phi_p(q)\vartheta_{Frac(\mathbb{Z}_{q,p})}$ is not zero, $cd_0(q,z),\ldots, cd_{n^2}(q,z)\in E_{\mathbb{Z},q,p}$, and
\begin{equation}\label{eq: querida}
\sum_{i=0}^{n^2}d_i(q,z)f(q^{p^{ih}},z^{p^{ih}})\equiv0\bmod\phi_p(q)\vartheta_{Frac(\mathbb{Z}_{q,p})}[[z]].
\end{equation}
Let $\mathfrak{n}$ be the maximal ideal of $\vartheta_{Frac(\mathbb{Z}_{q,p})}$.  Note that $\mathfrak{m}\subset\mathfrak{n}$ because $\mathfrak{m}\subset\mathbb{Z}[q]\subset\vartheta_{Frac(\mathbb{Z}_{q,p})}$  and each element of $\mathfrak{m}$ has a norm less than 1. Hence,  $\phi_p(q)\in\mathfrak{n}$, and for all $i\geq0$, $q^{p^{ih}}\equiv1\bmod\mathfrak{n}$. Then, from Equality~\eqref{eq: querida}, we get that
\begin{equation}\label{eq: red mod n}
\sum_{i=0}^{n^2}d_{i}(1,z)f(1,z^{p^{ih}})\equiv0\bmod\mathfrak{n}\vartheta_{Frac(\mathbb{Z}_{q,p})}[[z]].
\end{equation}
According to Remark~\ref{frac3}, the reduction of $d_{0}(q,z),\ldots,d_{n^2}(q,z)$ modulo the maximal ideal $\mathfrak{n}$ are rational functions not all zero over the field  $\vartheta_{Frac(\mathbb{Z}_{q,p})}/\mathfrak{n}$. Note that for $i\geq0$, $f(1,z^{p^{ih}})\mod\mathfrak{n}\in\mathbb{F}_p[[z]]$  because the residue field of $\mathbb{Z}[q]$ is $\mathbb{F}_p$ and $f(1,z^{p^ih})\in\mathbb{Z}[q][[z]]$. Finally,  as $\vartheta_{Frac(\mathbb{Z}_{q,p})}/\mathfrak{n}$ is an extension of $\mathbb{F}_p$, then it follows  from Lemme 4.2 of \cite{vmsff} and from Equation~\eqref{eq: red mod n}  that there exist $a_0(z),\ldots,a_{n^2}(z)\in\mathbb{Z}[z]$  such that their reduction modulo $p$ are not all zero and  $$\sum_{i=0}^{n^2}a_i(z)f(1,z^{p^{ih}})\equiv0\bmod p\mathbb{Z}[[z]].$$
\section{ $q$-hypergeometric operators of order 1}\label{sec_qhyp1}
 Let us consider the following $q$-hypergeometric operator of order 1 $$\mathcal{H}_{q,\alpha}:=\sigma_q-\frac{1-z}{1-q^{\alpha}z}$$
 where $\alpha\in\mathbb{Q}\cap(0,1]$. 
 
 The aim of this section  is to show that if $\alpha=a/b\in\mathbb{Q}\cap(0,1]$ and $p$ is congruent to 1 modulo $b$, then the $q$-difference operator $\mathcal{H}_{q,\alpha}$ has a strong Frobenius structure with period 1. We start by proving that if $\alpha\in\mathbb{Z}_p$ then $\mathcal{H}_{q,\alpha}$ is a $q$-difference operator with coefficients in $E_{\mathbb{Z},q,p}$. For this purpose  we need somme preliminary results.

\begin{lemm}\label{norme}
	For every  positive integer $j$, $$1-q^{p^j}=(1-q)\phi_p(q)\cdots\phi_p(q^{p^{j-1}}).$$ Moreover, for all positive integers $j$, $|1-q^{p^j}|_{\mathfrak{m}}\leq1/p^{j+1}$ and $\left|\frac{1-q^{p^j}}{1-q}\right|_{\mathfrak{m}}\leq1/p^j$.
\end{lemm}

\begin{proof}
	Firstly, we are going to prove that for every positive integer  $j$, $$1-q^{p^j}=(1-q)\phi_p(q)\cdots\phi_p(q^{p^{j-1}}).$$ To this end we proceed by induction on $j$.  Remember that $\phi_p(q)=1+q+\cdots+q^{p-1}$. 
	
\smallskip
	
It is clear that, $1-q^p=(1-q)\phi_p(q)$. Whence our claim is true for $j=1$.
	
	Suppose that for a positive integer $j$, we have $1-q^{p^j}=(1-q)\phi_p(q)\cdots\phi_p(q^{p^{j-1}})$. Note that $$1-q^{p^{j+1}}=(1-q^{p^j})(1+q^{p^j}+\cdots+q^{p^{j+1}-p^j})=(1-q^{p^j})\phi_p(q^{p^j}).$$
	Then, the induction hypothesis  implies that $$1-q^{p^{j+1}}=(1-q)\phi_p(q)\cdots\phi_p(q^{p^{j-1}})\phi_p(q^{p^{j}}).$$
	Therefore, by induction we conclude that for every positive integer $j$, \begin{equation}\label{producto}
	1-q^{p^{j}}=(1-q)\phi_p(q)\cdots\phi_p(q^{p^{j-1}}).
	\end{equation}
	Secondly, we prove that for every positive integer $j$, $|1-q^{p^j}|_{\mathfrak{m}}\leq1/p^{j+1}$.  As $\phi_p(1)=p$ then  Remark~\ref{noyau3} implies that for every integer $j$,  $\phi_p(q^{p^j})\in\mathfrak{m}$. Besides, recall that $1-q\in\mathfrak{m}$. Thus, from Equality~\eqref{producto} follows that, for every positive intger $j$, $1-q^{p^j}\in\mathfrak{m}^{j+1}$.  For this reason, for every positive integer $j$, $|1-q^{p^j}|_{\mathfrak{m}}\leq1/p^{j+1}$. \\
	Finally, from~\eqref{producto} we see that for every positive integer $j$, $$\frac{1-q^{p^j}}{1-q}=\phi_p(q)\cdots\phi_p(q^{p^{j-1}}).$$   So, for every positive integer $j$, $\frac{1-q^{p^j}}{1-q}\in\mathfrak{m}^j$. Therefore, $\left|\frac{1-q^{p^j}}{1-q}\right|_{\mathfrak{m}}\leq1/p^j$.
\end{proof}

\begin{lemm}\label{beta}
	Let $p$ be a prime number, let $\beta$ be in $\mathbb{Z}_{p}$ and let $t$ be a positive integer such that $p$ does not divide  $t$.
	\begin{enumerate}
		\item  The following elements belong to $\mathbb{Z}_{q,p}$: $q^{\beta}$, $1-q^{\beta}$ and $\frac{1-q^{\beta}}{1-q^t}$. 
		
		\item  For every non-negative integer $h$,  $\textbf{F}_q^h(q^{\beta})=q^{\beta p^h}$ and $\textbf{F}_q^h\left(\frac{1-q^{\beta}}{1-q^t}\right)=\frac{1-q^{\beta p^{h}}}{1-q^{tp^{h}}}$. In particular, $q^{\beta p^h}$ and $\frac{1-q^{\beta p^{h}}}{1-q^{tp^{h}}}$ belong to $\mathbb{Z}_{q,p}$.
		
		\item For every non-negative integer $h$, $\textbf{ev}(q^{\beta p^h})=1$ and $\textbf{ev}\left(\frac{1-q^{\beta p^{h}}}{1-q^{tp^{h}}}\right)=\frac{\beta}{t}$.
	\end{enumerate} 
\end{lemm}

\begin{proof}
	As $\beta\in\mathbb{Z}_p$ then $\beta=\sum\limits_{n\geq0}s_np^n$, where $s_n\in\{0,1,\ldots,p-1\}$ for all $n\geq0$. 
	
	\smallskip
	
	(i). For every $j\geq0$,  put $\beta_j=\sum\limits_{n=0}^js_np^n$; then
	\begin{equation}\label{qb}
	q^{\beta}=q^{\beta_0}+q^{\beta_0}(q^{s_1p}-1)+q^{\beta_1}(q^{s_2p^2}-1)+\cdots+q^{\beta_{j-1}}(q^{s_{j}p^{j}}-1)+\cdots.
	\end{equation}
	
	For every integer $j\geq1$, there exists $r_j(q)\in\mathbb{Z}[q]$ such that $$q^{s_{j}p^{j}}-1=r_j(q)(q^{p^j}-1).$$  
	According to Lemma~\ref{norme}, the norm of  $q^{p^j}-1$ is less than or equal to $1/p^{j+1}$. Therefore, we have $|q^{s_{j}p^{j}}-1|_{\mathfrak{m}}\leq1/p^{j+1}$. But $|q^{\beta_{j-1}}(q^{s_{j}p^{j}}-1)|_{\mathfrak{m}}=|q^{s_{j}p^{j}}-1|_{\mathfrak{m}}$ and thus we obtain $|q^{\beta_{j-1}}(q^{s_{j}p^{j}}-1)|_{\mathfrak{m}}\leq1/p^{j+1}$.  Since the $\mathfrak{m}$-adic norm is ultrametric and we have $\lim\limits_{j\rightarrow\infty}\left|q^{\beta_{j-1}}(q^{s_{j}p^{j}}-1)\right|_{ \mathfrak{m}}=0$, it follows from~\eqref{qb} that $q^{\beta}$ is an element of $\mathbb{Z}_{q,p}$.  Since $q^{\beta}\in\mathbb{Z}_{q,p}$, we have $1-q^{\beta}\in\mathbb{Z}_{q,p}$.\\
	 Now, let $t$ be a positive integer such that $p$ does not divide $t$. By Remark~\ref{noyau3}, the polynomial $h_t(q)=1+q+\cdots+q^{t-1}$ does not belong to $\mathfrak{m}$. Then, $\frac{1}{h_t(q)}\in\mathbb{Z}[q]_{\mathfrak{m}}$ and $\left|\frac{1}{h_t(q)}\right|_{\mathfrak{m}}=1$. \\
	 Since $1-q^t=(1-q)h_t(q)$, it follows from~\eqref{qb} that
	\begin{equation}\label{betat}
	\frac{1-q^{\beta}}{1-q^t}=\frac{1}{h_t(q)}\left[\frac{1-q^{\beta_0}}{1-q}-q^{\beta_0}\frac{q^{s_1p}-1}{1-q}-\cdots-q^{\beta_{j-1}}\frac{q^{s_jp^j}-1}{1-q}-\cdots\right].
	\end{equation}
	Since for every integer $n>0$, $(1-q^n)/(1-q)=1+q+\cdots+q^{n-1}\in\mathbb{Z}[q]_{\mathfrak{m}}$, we get that for every integer  $j>0$, $q^{\beta_{j-1}}\frac{q^{s_jp^j}-1}{1-q}\in\mathbb{Z}[q]_{\mathfrak{m}}.$ \\
	Following Lemma~\ref{norme},  we have $\left|\frac{q^{p^j}-1}{1-q}\right|_{\mathfrak{m}}\leq1/p^j$ and as $q^{s_jp^j}-1=r_j(q)(q^{p^j}-1)$, then  $\left|q^{\beta_{j-1}}\frac{q^{s_jp^j}-1}{1-q}\right|_{\mathfrak{m}}\leq1/p^j$. Since the $\mathfrak{m}$-adic norm is ultrametric and $\mathbb{Z}_{q,p}$ is the completion of $\mathbb{Z}[q]_{ \mathfrak{m}}$ for this norm, and since moreover $\lim_{j\rightarrow\infty}\left|q^{\beta_{j-1}}\frac{q^{s_jp^j}-1}{1-q}\right|_{\mathfrak{m}}=0$, it follows from~\eqref{betat} that $\frac{1-q^{\beta}}{1-q^t}\in \mathbb{Z}_{q,p}$. 
	
	 \smallskip
	
	(ii). Now, we prove that $\textbf{F}_q(q^{\beta})=q^{\beta p^h}$.  From the definition of the endomorphism  $\textbf{F}_q$ and from Equality~\eqref{qb}, we have
	\begin{multline*}
	\textbf{F}_q^h(q^{\beta})=\\\textbf{F}_q^h(q^{\beta_0})+\textbf{F}_q^h(q^{\beta_0}(q^{s_1p}-1))+\textbf{F}_q^h(q^{\beta_1}(q^{s_2p^2}-1))+\cdots+\textbf{F}_q^h(q^{\beta_{j-1}}(q^{s_{j}p^{j}}-1))+\cdots
	\\=q^{p^h\beta_0}+q^{p^h\beta_0}(q^{s_1p^{h+1}}-1)+q^{p^h\beta_1}(q^{s_2p^{h+2}}-1)+\cdots+q^{p^h\beta_{j-1}}(q^{s_{j}p^{h+j}}-1)+\cdots.
	\end{multline*}	
	But,
	\begin{equation}\label{hbeta}
	q^{p^h\beta_0}+q^{p^h\beta_0}(q^{s_1p^{h+1}}-1)+q^{p^h\beta_1}(q^{s_2p^{h+2}}-1)+\cdots+q^{p^h\beta_{j-1}}(q^{s_{j}p^{h+j}}-1)+\cdots=q^{\beta p^h}.
	\end{equation}
	
	Therefore, $\textbf{F}_q^h(q^{\beta})=q^{\beta p^h}$. Now, we prove that $\textbf{F}_q^h\left(\frac{1-q^{\beta}}{1-q^t}\right)=\frac{1-q^{\beta p^{h}}}{1-q^{tp^{h}}}$. Let $n$ be a positive integer, then $(1-q^n)/(1-q)=1+q+\cdots+q^{n-1}$ and $$\textbf{F}_q^h\left(\frac{1-q^n}{1-q}\right)=1+q^{p^h}+\cdots+q^{p^h(n-1)}=\frac{1-q^{p^hn}}{1-q^{p^h}}.$$
	Then, it follows from \eqref{betat} that,
	\begin{align*}
	&\textbf{F}_q^h\left(\frac{1-q^{\beta}}{1-q^t}\right)\\
	=&\frac{1}{h_t(q^{p^h})}\left[\frac{1-q^{p^h\beta_0}}{1-q^{p^h}}-q^{p^h\beta_0}\frac{q^{s_1p^{h+1}}-1}{1-q^{p^h}}-\cdots-q^{p^h\beta_{j-1}}\frac{q^{s_jp^{h+j}}-1}{1-q^{p^h}}-\cdots\right]\\
	=&\frac{1}{h_t(q^{p^h})(1-q^{p^h})}\left[1-\left(q^{p^h\beta_0}+q^{p^h\beta_0}(q^{s_1p^{h+1}}-1)+\cdots+q^{p^h\beta_{j-1}}(q^{s_{j}p^{h+j}}-1)-\cdots\right)\right]\\
	=&\frac{1-q^{\beta p^h}}{1-q^{tp^h}}.
	\end{align*}
	
	\smallskip
	
	(iii). Following Equality~\eqref{hbeta} we have, $\textbf{ev}(q^{\beta p^h})=1$. As we have just seen
	\begin{equation}\label{et}
	\frac{1-q^{\beta p^h}}{1-q^{tp^h}}=\frac{1}{h_t(q^{p^h})}\left[\frac{1-q^{p^h\beta_0}}{1-q^{p^h}}-q^{p^h\beta_0}\frac{q^{s_1p^{h+1}}-1}{1-q^{p^h}}-\cdots-q^{p^h\beta_{j-1}}\frac{q^{s_jp^{h+j}}-1}{1-q^{p^h}}-\cdots\right].
	\end{equation}
	Note that for every integer $j>0$, $$\frac{q^{s_jp^{h+j}}-1}{1-q^{p^h}}=-(1+q^{p^h}+\cdots+p^{p^h(s_jp^j-1)}).$$
	So, for every integer $j>0$, $\textbf{ev}\left(\frac{q^{s_jp^{h+j}}-1}{1-q^{p^h}}\right)=-s_jp^j$ and we know that, $\textbf{ev}\left(\frac{1}{h_t(q^{p^h})}\right)=1/t$. Therefore, according to Equality~\eqref{et}, $$\textbf{ev}\left(\frac{1-q^{p^h\beta}}{1-q^{p^ht}}\right)=\frac{1}{t}\left[\beta_0+s_1p+\cdots+s_jp^j+\cdots\right]=\frac{\beta}{t}.$$ \end{proof}

\begin{coro}
	Let $\alpha$ be in $\mathbb{Q}\cap\mathbb{Z}_p$  and $\mathcal{H}_{q,\alpha}:=\sigma_q-\frac{1-z}{1-q^{\alpha}z}$. Then $\mathcal{H}_{q,\alpha}\in E_{\mathbb{Z},q,p}[\sigma_q]$. 
\end{coro}

\begin{proof}
	As $\alpha\in\mathbb{Z}_p$, by Lemma~\ref{beta}, we have  $q^{\alpha}\in\mathbb{Z}_{q,p}$. Furthermore, the polynomial $1-q^{\alpha}z\in S$. Consequently, $\frac{1-z}{1-q^{\alpha}z}\in\mathbb{Z}_{q,p}[z]_{S}$. Whence, $\mathcal{H}_{q,\alpha}\in E_{\mathbb{Z},q,p}[\sigma_q]$. 
\end{proof}
The power series $$f_{\alpha}=\sum\limits_{n\geq0}\frac{(q^\alpha,q)_n}{(q,q)_n}z^n,\quad \frac{(q^{\alpha},q)_n}{(q,q)_n}=\frac{(1-q^{\alpha})(1-q^{\alpha+1})\cdots(1-q^{\alpha+n-1})}{(1-q)(1-q^2)\cdots(1-q^n)}$$ is a solution of $\mathcal{H}_{q,\alpha}$.

\begin{lemm}\label{f}
	Let $\alpha=a/b$ be in $\mathbb{Q}\cap(0,1]$, $p$ be a prime number congruent to 1 modulo $b$ and $\textbf{F}_q$ be the Frobenius endomorphism of $\mathbb{Z}_{q,p}$.  Then:
	
	\begin{enumerate}
		\item for every integer $n\geq0$, $\frac{(q^{\alpha},q)_n}{(q,q)_n}\in\mathbb{Z}_{q,p}$. In particular, the power series $f_{\alpha}$ belongs to $\mathbb{Z}_{q,p}[[z]]$,
		
		\item for every integer $n\geq0$, $$\textbf{F}_q\left(\frac{(q^{\alpha},q)_n}{(q,q)_n}\right)=\frac{\prod\limits_{k=0}^{n-1}1-q^{p(\alpha+k)}}{\prod\limits_{k=1}^{n}1-q^{pk}},$$
		
		\item for every integer $n\geq0$, $$\frac{(q^{\alpha},q)_n}{(q,q)_n}\in \textbf{F}_q\left(\frac{(q^{\alpha},q)_{\lfloor{n/p}\rfloor}}{(q,q)_{\lfloor{n/p}\rfloor}}\right)\mathbb{Z}_{q,p}.$$
	\end{enumerate}
	
\end{lemm}

\begin{proof}
	
	(i). Let $n$  be a non-negative integer.  We start by writing  $n$ in base $p$, that is $$n=r_0+r_1p+\cdots+ r_{l-1}p^{l-1}+r_lp^l,$$ where $r_0,\ldots, r_l$ belong to $\{0,1,\ldots,p-1\}$.  Now put $n_1=r_1+r_2p+\cdots+r_{l}p^{l-1}$; then we have $n=n_1p+r_0$.  As $p\equiv1\mod b$ and $0<\alpha\leq1$, then there is a unique $j\in\{1,\ldots,p-1\}$ such that $\alpha+j=p\alpha$. Note that,
	
	\begin{equation}\label{t0}
	\frac{(q^{\alpha},q)_n}{(q,q)_n}=\frac{\prod\limits_{k=0}^{n_1-1}\prod\limits_{i=0}^{p-1}(1-q^{\alpha+kp+i})}{\prod\limits_{k=0}^{n_1-1}\prod\limits_{i=1}^{p}(1-q^{kp+i})}\cdot
	\frac{(1-q^{\alpha+n_1p})\cdots(1-q^{\alpha+n_1p+r_0-1})}{(1-q^{n_1p+1})\cdots(1-q^{n_1p+r_0})}.
	\end{equation}
	Moreover, for every $k\in\{0,\ldots, n_1-1\}$, $(\alpha+kp)+j=p(\alpha+k)$.  Therefore,
	\begin{equation}\label{t}
	\frac{\prod\limits_{k=0}^{n_1-1}\prod\limits_{i=0}^{p-1}(1-q^{\alpha+kp+i})}{\prod\limits_{k=0}^{n_1-1}\prod\limits_{i=1}^{p}(1-q^{kp+i})}=\frac{\prod\limits_{k=0}^{n_1-1}\prod\limits_{i=0,i\neq j}^{p-1}(1-q^{\alpha+kp+i})}{\prod\limits_{k=0}^{n_1-1}\prod\limits_{i=1}^{p-1}(1-q^{kp+i})}\cdot\frac{\prod\limits_{k=0}^{n_1-1}(1-q^{p(\alpha+k)})}{\prod\limits_{k=1}^{n_1}(1-q^{pk})}.
	\end{equation}
	Then, from Equalities~\eqref{t0} and~\eqref{t}  we obtain
	\begin{small}
		\begin{equation}\label{t1}
		\frac{(q^{\alpha},q)_n}{(q,q)_n}=\frac{\prod\limits_{k=0}^{n_1-1}\prod\limits_{i=0,i\neq j}^{p-1}(1-q^{\alpha+kp+i})}{\prod\limits_{k=0}^{n_1-1}\prod\limits_{i=1}^{p-1}(1-q^{kp+i})}\cdot\frac{(1-q^{\alpha+n_1p})\cdots(1-q^{\alpha+n_1p+r_0-1})}{(1-q^{n_1p+1})\cdots(1-q^{n_1p+r_0})}\cdot\frac{\prod\limits_{k=0}^{n_1-1}(1-q^{p(\alpha+k)})}{\prod\limits_{k=1}^{n_1}(1-q^{pk})}.
		\end{equation} 
	\end{small}
	Lemme~\ref{beta} implies that, if $m$ is an integer and $t$ is a positive integer such that $p$ does not divide $t$ then $\frac{1-q^{\alpha+m}}{1-q^{t}}\in\mathbb{Z}_{q,p}$. Therefore,
	\begin{equation}\label{t2}
	\frac{\prod\limits_{k=0}^{n_1-1}\prod\limits_{i=0,i\neq j}^{p-1}(1-q^{\alpha+kp+i})}{\prod\limits_{k=0}^{n_1-1}\prod\limits_{i=1}^{p-1}(1-q^{kp+i})}\cdot
	\frac{(1-q^{\alpha+n_1p})\cdots(1-q^{\alpha+n_1p+r_0-1})}{(1-q^{n_1p+1})\cdots(1-q^{n_1p+r_0})}\in\mathbb{Z}_{q,p}.
	\end{equation} 	
	
	According to  \eqref{t1} and \eqref{t2}, in order to prove that $\frac{(q^{\alpha},q)_n}{(q,q)_n}\in\mathbb{Z}_{q,p}$  it is sufficient to see that $$\frac{\prod\limits_{k=0}^{n_1-1}1-q^{p(\alpha+k)}}{\prod\limits_{k=1}^{n_1}1-q^{pk}}\in\mathbb{Z}_{q,p}.$$
	We set $\mathcal{A}=\{0,1,\ldots, n_1-1\}$ and $\mathcal{A'}=\{1,2,\ldots,n_1\}$. Recall that $j\in\{1,\ldots,p-1\}$ is the unique integer such that $\alpha+j=p\alpha$. Now, let us consider the set $\mathfrak{C'}=\mathcal{A'}\cap p\mathbb{Z}$ and the map $\gamma:\mathfrak{C'}\rightarrow\mathcal{A}$ defined as follows:
	\begin{align*}
	\gamma(k)=&j+jp+\cdots+jp^{v_p(k)-1}+k-p^{v_p(k)},
	\end{align*}
	where $v_p(k)$ is the $p$-adic valuation of $k$. Now, we show that $\gamma(k)\in\mathcal{A}$. It is clear that $\gamma(k)\geq0$ and $p^{v_p(k)}\alpha=\alpha+j+jp+\cdots+jp^{v_p(k)-1}$ because $\alpha+j=p\alpha$. As $\alpha\in(0,1]$, then
	\begin{align*}
	\gamma(k)<&\alpha+j+jp+\cdots+jp^{v_p(k)-1}+k-p^{v_p(k)}\\=&p^{v_p(k)}\alpha+k-p^{v_p(k)}\\\leq&p^{v_p(k)}+k-p^{v_p(k)}\\=&k.
	\end{align*}
	Therefore, $\gamma(k)<k\leq n_1$.  Moreover,  it is not hard to see  that $\gamma$ is injective. \\
	Let $k$ be in $\mathfrak{C'}$ and we wet $t=k/p^{v_p(k)}$. Since $p$ does not divide $t$ and $\alpha+t-1\in\mathbb{Z}_p$,  we are in a position to apply Lemma~\ref{beta} to deduce that $\frac{1-q^{\alpha+t-1}}{1-q^t}\in\mathbb{Z}_{q,p}$ and once again by Lemma~\ref{beta}, we have  $$\textbf{F}_q^{v_p(k)+1}\left(\frac{1-q^{\alpha+t-1}}{1-q^t}\right)=\frac{1-q^{p^{v_p(k)+1}(\alpha+t-1)}}{1-q^{tp^{v_p(k)+1}}}\in\mathbb{Z}_{q,p}.$$
	Note that
	\begin{align*}
	p^{v_p(k)+1}(\alpha+t-1)=&p^{v_p(k)+1}(\alpha+k/p^{v_p(k)}-1)\\
	=&p(p^{v_p(k)}\alpha+k-p^{v_p(k)})\\
	=&p(\alpha+j+jp+\cdots+jp^{v_p(k)-1}+k-p^{v_p(k)})\\
	=&p(\alpha+\gamma(k)).
	\end{align*}
	Therefore,
	\begin{equation}\label{61}
	\textbf{F}_q^{v_p(k)+1}\left(\frac{1-q^{\alpha+t-1}}{1-q^t}\right)=\frac{1-q^{p(\alpha+\gamma(k))}}{1-q^{pk}}\in\mathbb{Z}_{q,p}.
	\end{equation}
	Since $k\in\mathfrak{C'}$ is arbitrary, we get that for all $k\in\mathfrak{C'}$,  $\frac{1-q^{p(\alpha+\gamma(k))}}{1-q^{pk}}\in\mathbb{Z}_{q,p}$. So,
	\begin{equation}\label{62}
	\prod\limits_{k\in\mathfrak{C'}}\frac{1-q^{p(\alpha+\gamma(k))}}{1-q^{pk}}\in\mathbb{Z}_{q,p}.
	\end{equation}
	Let $\mathfrak{C}$ be the image of $\gamma$.  Since $\gamma$ is injective, it follows that $\mathfrak{C}$ and $\mathfrak{C'}$ have the same cardinal. Moreover, note that $\mathcal{A}$ and $\mathcal{A'}$ have the same cardinal and therefore, there is a bijection $\widetilde{\gamma}:\mathcal{A'}\rightarrow\mathcal{A}$ such that $\widetilde{\gamma}(x)=\gamma(x)$ for all $x\in\mathfrak{C'}$. So, if $x\in\mathcal{A'}\setminus\mathfrak{C'}$ then $\widetilde{\gamma}(x)\in\mathcal{A}\setminus\mathfrak{C}$. If $k\in\mathcal{A'}\setminus\mathfrak{C'}$ then $p$ does not divide $k$ and since $\alpha+\widetilde{\gamma}(k)\in\mathbb{Z}_{p}$, from Lemma~\ref{beta}, we have $\frac{1-q^{\alpha+\widetilde{\gamma}(k)}}{1-q^{k}}\in\mathbb{Z}_{q,p}$ and once again by Lemma~\ref{beta} we have 
	\begin{equation}\label{63}
	\textbf{F}_q\left(\frac{1-q^{\alpha+\widetilde{\gamma}(k)}}{1-q^{k}}\right)=\frac{1-q^{p(\alpha+\widetilde{\gamma}(k))}}{1-q^{pk}}\in\mathbb{Z}_{q,p}.
	\end{equation}
	For this reason we have,
	\begin{equation}\label{64}
	\prod\limits_{k\in\mathcal{A'}\setminus\mathfrak{C'}}\frac{1-q^{p(\alpha+\widetilde{\gamma}(k))}}{1-q^{pk}}\in\mathbb{Z}_{q,p}.
	\end{equation}
	Finally,  as $\widetilde{\gamma}:\mathcal{A'}\rightarrow\mathcal{A}$ is a bijection such that  $\mathfrak{C}=\widetilde{\gamma}(\mathfrak{C'})=\gamma(\mathfrak{C'})$ and $\gamma:\mathfrak{C'}\rightarrow\mathcal{A}$ is injective then,
	
	\begin{align*}
	\frac{\prod\limits_{k=0}^{n_1-1}1-q^{p(\alpha+k)}}{\prod\limits_{k=1}^{n_1}1-q^{pk}}=&\frac{\prod\limits_{k\in\mathfrak{C}}(1-q^{p(\alpha+k)})\prod\limits_{k\in\mathcal{A}\setminus\mathfrak{C}}(1-q^{p(\alpha+k)})}{\prod\limits_{k\in\mathfrak{C'}}(1-q^{pk})\prod\limits_{k\in\mathcal{A'}\setminus\mathfrak{C'}}(1-q^{pk})}\\=&\prod\limits_{k\in\mathfrak{C'}}\frac{1-q^{p(\alpha+\widetilde{\gamma}(k))}}{1-q^{pk}}\prod\limits_{k\in\mathcal{A'}\setminus\mathfrak{C'}}\frac{1-q^{p(\alpha+\widetilde{\gamma}(k))}}{1-q^{pk}}\\
	=&\prod\limits_{k\in\mathfrak{C'}}\frac{1-q^{p(\alpha+\gamma(k))}}{1-q^{pk}}\prod\limits_{k\in\mathcal{A'}\setminus\mathfrak{C'}}\frac{1-q^{p(\alpha+\widetilde{\gamma}(k))}}{1-q^{pk}}.
	\end{align*}
	Therefore, it follows from \eqref{62} and from \eqref{64} that, $$\frac{\prod\limits_{k=0}^{n_1-1}1-q^{p(\alpha+k)}}{\prod\limits_{k=1}^{n_1}1-q^{pk}}\in\mathbb{Z}_{q,p}.$$
	(ii). Let $n_1$ be a non-negative integer.  According to the point (i), $\frac{(q^{\alpha},q)_{n_1}}{(q,q)_{n_1}}\in\mathbb{Z}_{q,p}$. We are going to show that $$\textbf{F}_q\left(\frac{(q^{\alpha},q)_{n_1}}{(q,q)_{n_1}}\right)=\frac{\prod\limits_{k=0}^{n_1-1}1-q^{p(\alpha+k)}}{\prod\limits_{k=1}^{n_1}1-q^{pk}}.$$
	Recall that $\textbf{F}_q(q)=q^p$. We set $\mathcal{A}=\{0,1,\ldots,n_1-1\}$, $\mathcal{A'}=\{1,2,\ldots,n_1\}$ and $\mathfrak{C'}=\mathcal{A'}\cap p\mathbb{Z}$. Remember that $\gamma:\mathfrak{C'}\rightarrow\mathcal{A}$ was defined at point (i) as follows $\gamma(k)=j+jp+\cdots+jp^{r-1}+k-p^{v_p(k)}$, where $j\in\{1,\ldots,p-1\}$ is the unique integer such that $\alpha+j=p\alpha$ and $v_p(k)$ is the $p$-adic valuation of $k$. Finally, recall that $\widetilde{\gamma}:\mathcal{A'}\rightarrow\mathcal{A}$ is a bijection such that $\widetilde{\gamma}(x)=\gamma(x)$ for all $x\in\mathfrak{C'}$.  Proceeding exactly as in the  point (i) we have	
	\begin{equation}\label{65}
	\frac{\prod\limits_{k=0}^{n_1-1}1-q^{p(\alpha+k)}}{\prod\limits_{k=1}^{n_1}1-q^{pk}}=\prod\limits_{k\in\mathfrak{C'}}\frac{1-q^{p(\alpha+\gamma(k))}}{1-q^{pk}}\prod\limits_{k\in\mathcal{A'}\setminus\mathfrak{C'}}\frac{1-q^{p(\alpha+\widetilde{\gamma}(k))}}{1-q^{pk}}.
	\end{equation}
	
	Let $k$ be in $\mathfrak{C'}$ and $t=k/p^{v_p(k)}$.  As $p$ does not divide $t$ and $\alpha+t-1\in\mathbb{Z}_p$ then, from Lemma~\ref{beta}, we obtain $\frac{1-q^{\alpha+t-1}}{1-q^t}\in\mathbb{Z}_{q,p}$ and $$ \textbf{F}_{q}^{v_p(k)}\left(\frac{1-q^{\alpha+t-1}}{1-q^t}\right)=\frac{1-q^{p^{v_p(k)}(\alpha+t-1)}}{1-q^{p^{v_p(k)}t}}.$$
	Since $p^{v_p(k)}t=k$ and $p^{v_p(k)}(\alpha+t-1)=\alpha+j+jp+\cdots+jp^{v_p(k)-1}+k-p^{v_p(k)}=\alpha+\gamma(k)$,  we have
	\begin{equation}\label{66}
	\textbf{F}_q^{v_p(k)}\left(\frac{1-q^{\alpha+t-1}}{1-q^t}\right)=\frac{1-q^{\alpha+\gamma(k)}}{1-q^{k}}.
	\end{equation}
	As $\textbf{F}_q^{v_p(k)+1}=\textbf{F}_q\circ \textbf{F}_q^{v_p(k)}$ it follows from \eqref{61} and \eqref{66} that,
	$$\frac{1-q^{p(\alpha+\gamma(k))}}{1-q^{pk}}=\textbf{F}_q\left(\frac{1-q^{\alpha+\gamma(k)}}{1-q^{k}}\right)=\textbf{F}_q^{v_p(k)+1}\left(\frac{1-q^{\alpha+t-1}}{1-q^t}\right).$$
	Therefore, 
	\begin{equation}\label{67}
	\prod\limits_{k\in\mathfrak{C'}}\frac{1-q^{p(\alpha+\gamma(k))}}{1-q^{pk}}=\prod\limits_{k\in\mathfrak{C'}}\textbf{F}_q\left(\frac{1-q^{\alpha+\gamma(k)}}{1-q^{k}}\right)=\textbf{F}_q\left(\prod\limits_{k\in\mathfrak{C'}}\frac{1-q^{\alpha+\gamma(k)}}{1-q^{k}}\right).
	\end{equation}
	Moreover,  from Equality~\eqref{63} we have, 
	\begin{equation}\label{68}
	\prod\limits_{k\in\mathcal{A'}\setminus\mathfrak{C'}}\frac{1-q^{p(\alpha+\widetilde{\gamma}(k))}}{1-q^{pk}}=\prod\limits_{k\in\mathcal{A'}\setminus\mathfrak{C'}}\textbf{F}_q\left(\frac{1-q^{(\alpha+\widetilde{\gamma}(k))}}{1-q^{k}}\right)=\textbf{F}_q\left(\prod\limits_{k\in\mathcal{A'}\setminus\mathfrak{C'}}\frac{1-q^{(\alpha+\widetilde{\gamma}(k))}}{1-q^{k}}\right).
	\end{equation}
	Equalities~\eqref{65}, \eqref{67}, and \eqref{68} imply that
	\begin{align*}
	\frac{\prod\limits_{k=0}^{n_1-1}1-q^{p(\alpha+k)}}{\prod\limits_{k=1}^{n_1}1-q^{pk}}=&\textbf{F}_q\left(\prod\limits_{k\in\mathfrak{C'}}\frac{1-q^{\alpha+\gamma(k)}}{1-q^{k}}\right)\cdot \textbf{F}_q\left(\prod\limits_{k\in\mathcal{A'}\setminus\mathfrak{C'}}\frac{1-q^{(\alpha+\widetilde{\gamma}(k))}}{1-q^{k}}\right)\\
	=&\textbf{F}_q\left(\prod\limits_{k\in\mathfrak{C'}}\frac{1-q^{\alpha+\gamma(k)}}{1-q^{k}}\prod\limits_{k\in\mathcal{A'}\setminus\mathfrak{C'}}\frac{1-q^{(\alpha+\widetilde{\gamma}(k))}}{1-q^{k}}\right).
	\end{align*}
	Finally, since $\widetilde{\gamma}:\mathcal{A'}\rightarrow\mathcal{A}$ is a bijection and $\widetilde{\gamma}(x)=\gamma(x)$ for all $x\in\mathfrak{C'}$,  $$\prod\limits_{k\in\mathfrak{C'}}\frac{1-q^{\alpha+\gamma(k)}}{1-q^{k}}\prod\limits_{k\in\mathcal{A'}\setminus\mathfrak{C'}}\frac{1-q^{(\alpha+\widetilde{\gamma}(k))}}{1-q^{k}}=\frac{(q^{\alpha},q)_{n_1}}{(q,q)_{n_1}}.$$
	Therefore, $$\frac{\prod\limits_{k=0}^{n_1-1}1-q^{p(\alpha+k)}}{\prod\limits_{k=1}^{n_1}1-q^{pk}}=\textbf{F}_q\left(\frac{(q^{\alpha},q)_{n_1}}{(q,q)_{n_1}}\right).$$
	(iii). Let $n$ be a non-negative integer. By writing  $n$ in base $p$ we obtain $n=r_0+r_1p+\cdots+ r_{l-1}p^{l-1}+r_lp^l$, where $r_0,\ldots, r_l$ belong to $\{0,1,\ldots,p-1\}$. Put now $n_1=r_1+r_2p+\cdots+r_{l}p^{l-1}$; then $n_1=\lfloor{n/p}\rfloor$. From Equation~\eqref{t1} and from (ii) we obtain 
	
	\begin{equation}\label{reduction}
	\frac{(q^{\alpha},q)_n}{(q,q)_n}=\frac{\prod\limits_{k=0}^{n_1-1}\prod\limits_{i=0,i\neq j}^{p-1}(1-q^{\alpha+kp+i})}{\prod\limits_{k=0}^{n_1-1}\prod\limits_{i=1}^{p-1}(1-q^{kp+i})}\cdot\frac{(1-q^{\alpha+n_1p})\cdots(1-q^{\alpha+n_1p+r_0-1})}{(1-q^{n_1p+1})\cdots(1-q^{n_1p+r_0})}\cdot \textbf{F}_q\left(\frac{(q^{\alpha},q)_{\lfloor{n/p}\rfloor}}{(q,q)_{\lfloor{n/p}\rfloor}}\right)
	\end{equation}
But, by Equation~\eqref{t2}	, we know that $\frac{\prod\limits_{k=0}^{n_1-1}\prod\limits_{i=0,i\neq j}^{p-1}(1-q^{\alpha+kp+i})}{\prod\limits_{k=0}^{n_1-1}\prod\limits_{i=1}^{p-1}(1-q^{kp+i})}\cdot\frac{(1-q^{\alpha+n_1p})\cdots(1-q^{\alpha+n_1p+r_0-1})}{(1-q^{n_1p+1})\cdots(1-q^{n_1p+r_0})}$ belongs to $\mathbb{Z}_{q,p}$. Therefore, $$\frac{(q^{\alpha},q)_n}{(q,q)_n}\in\textbf{F}_q\left(\frac{(q^{\alpha},q)_{\lfloor{n/p}\rfloor}}{(q,q)_{\lfloor{n/p}\rfloor}}\right)\mathbb{Z}_{q,p}.$$ 
\end{proof}
Theorem~\ref{sfF}, stated below, ensures that under certains hypotheses the operator $\mathcal{H}_{q,\alpha}$ has a strong Frobenius structure with period 1.  

\begin{theo}\label{sfF}
	If $\alpha=a/b\in\mathbb{Q}\cap(0,1]$ and $p$ is a prime number congruent to 1 modulo $b$, then the $q$-difference operator $\mathcal{H}_{q,\alpha}$ has a strong Frobenius structure for $p$ with period $1$.
\end{theo}

Our proof of Theorem~\ref{sfF} is based on Lemma~\ref{elementanalytique} which is proved at end of this section.

\begin{lemm}\label{elementanalytique}
	Let $\alpha=a/b$ be in $\mathbb{Q}\cap(0,1]$ and $p$ be a prime number congruent to 1 modulo $b$.  Let $\{B^{(i)}(k)\}_{k\in\mathbb{N},i\geq-1}$ be the set of sequences with coefficients in $\mathbb{Z}_{q,p}$ which are defined as follows: $B^{(-1)}(n)=\frac{(q^{\alpha},q)_n}{(q,q)_n}$ and for every $i\geq0$, $B^{(i)}(n)=\textbf{F}^{i+1}_q(B^{(-1)}(n))$.  For each $s\geq0$, we set $F_s=\sum\limits_{n=0}^{p^s-1}B^{(-1)}(n)z^n$. Then $$f_{\alpha}\textbf{F}_q(F_s)\equiv \textbf{F}_q(f_{\alpha})F_{s+1}\mod\mathfrak{m}^{s+1},$$
	where $f_{\alpha}=\sum\limits_{n\geq0}B^{(-1)}(n)z^n$. 
\end{lemm}
By assuming Lemme~\ref{elementanalytique}, we are in a position to prove Theorem~\ref{sfF}.
\begin{proof}[Proof of Theorem~\ref{sfF}]
	 We set $f_{\alpha}=\sum\limits_{n\geq0}\frac{(q^\alpha,q)_n}{(q,q)_n}z^n$.  Lemma~\ref{f} implies that $f_{\alpha}\in\mathbb{Z}_{q,p}[[z]]$ and we can then apply $\textbf{F}_q$ to each coefficient of $f_{\alpha}$. For this reason  we can consider the power series $\textbf{F}_q(f_{\alpha})=\sum\limits_{n\geq0}\textbf{F}_q\left(\frac{(q^\alpha,q)_n}{(q,q)_n}\right)z^{np}$.\\
We set $H=\frac{f_{\alpha}}{\textbf{F}_q(f_{\alpha})}$. The companion matrix of $\mathcal{H}_{q,\alpha}$ is $A_q=\frac{1-z}{1-q^{\alpha}z}$. As $f_{\alpha}$ is a solution of $\mathcal{H}_{q,\alpha}$ then $\textbf{F}_q(f_{\alpha})$ is a solution of $\sigma_q -A^{\textbf{F}_q}_q$, where $\textbf{F}_q$ is the Frobenius endomorphisme of $E_{\mathbb{Z},q,p}$ constructed in Proposition~\ref{prop: proprietes de E_q,p3}. So, $$\sigma_q(H)A^{\textbf{F}_q}_q=A_qH.$$
 Thus, in order to show that $\mathcal{H}_{q,\alpha}$ has a strong Frobenius structure with period 1, it suffices to prove that  $H$ is an invertible element of $E_{\mathbb{Z},q,p}$. Indeed, we are going to see that this is the case. For every $s\geq0$, we set $F_s=\sum\limits_{n=0}^{p^s-1}\frac{(q^{\alpha},q)_n}{(q,q)_n}z^n$. By Lemma~\ref{f}, we obtain $F_s\in\mathbb{Z}_{q,p}[z]$. We set $\textbf{F}_q(F_{s})=\sum\limits_{n=0}^{p^s-1}\textbf{F}_q\left(\frac{(q^{\alpha},q)_n}{(q,q)_n}\right)z^{np}$.  Note the rational function $\frac{F_{s+1}}{\textbf{F}_q(F_s)}$ belongs to $\mathbb{Z}_{q,p}[z]_{S}$ because it follows from Lemma~\ref{f} that the polynomials $F_{s+1}, \textbf{F}_q(F_s)$ are in $\mathbb{Z}_{q,p}[z]$ and it is clear that $\textbf{F}_q(F_{s})(0)=1$, then $\textbf{F}_q(F_s)\in~S$. According to Lemma~\ref{elementanalytique}, for every $s\geq0$, $$f_{\alpha}\textbf{F}_q(F_s)\equiv\textbf{F}_q(f_{\alpha})F_{s+1}\mod\mathfrak{m}^{s+1}.$$
	That is, for every $s\geq0$, $f_{\alpha}\textbf{F}_q(F_s)-\textbf{F}_q(f_{\alpha})F_{s+1}\in\mathfrak{m}^{s+1}\mathbb{Z}_{q,p}[[z]]$. Since $\textbf{F}_q(F_s)\textbf{F}_q(f_{\alpha})(0)=1$, we have $\frac{1}{\textbf{F}_q(F_s)\textbf{F}_q(f_{\alpha})}\in\mathbb{Z}_{q,p}[[z]]$. Therefore, $$\frac{f_{\alpha}}{ \textbf{F} _q(f_{\alpha})}-\frac{F_{s+1}}{\textbf{F}_q(F_s)}=\frac{f_{\alpha}\textbf{F}_q(F_s)-\textbf{F}_q(f_{\alpha})F_{s+1}}{\textbf{F}_q(F_s) \textbf{F} _q(f_{\alpha})}\in\mathfrak{m}^{s+1}\mathbb{Z}_{q,p}[[z]].$$ Then, $$\left|\frac{f_{\alpha}}{\textbf{F}_q(f_{\alpha})}-\frac{F_{s+1}}{\textbf{F}_q(F_s)}\right|_{\mathcal{G}}\leq\frac{1}{p^{s+1}}.$$
	For this reason, $\frac{f_{\alpha}}{\textbf{F}_q(f_{\alpha})}$ is the limit of the rational functions $\frac{F_{s+1}}{\textbf{F}_q(F_s)}$. Therefore, $\frac{f_{\alpha}}{\textbf{F}_q(f_{\alpha})}\in E_{\mathbb{Z},q,p}$.  Now, note that 
	$$\frac{\textbf{F}_q(f_{\alpha})}{f_{\alpha}}-\frac{\textbf{F}_q(F_s)}{F_{s+1}}=\frac{\textbf{F}_q(f_{\alpha})F_{s+1}-f_{\alpha}\textbf{F}_q(F_s)}{f_{\alpha}F_{s+1}}.$$ 
	Since $f_{\alpha}F_{s+1}(0)=1$, $\frac{1}{f_{\alpha}F_{s+1}}\in\mathbb{Z}_{q,p}[[z]]$. As we have already seen $\textbf{F}_q(f_{\alpha})F_{s+1}-f_{\alpha}\textbf{F}_q(F_s)\in\mathfrak{m}^{s+1}\mathbb{Z}_{q,p}[[z]]$, thus $$\left|\frac{\textbf{F}_q(f_{\alpha})}{f_{\alpha}}-\frac{\textbf{F}_q(F_s)}{F_{s+1}}\right|_{\mathcal{G}}\leq\frac{1}{p^{s+1}}.$$
	So, $\frac{\textbf{F}_q(f_{\alpha})}{f_{\alpha}}$ is the limit of the rational functions $\frac{\textbf{F}_q(F_s)}{F_{s+1}}\in\mathbb{Z}_{q,p}[z]_S$ because $F_{s+1}(0)=1$. Then, $\frac{\textbf{F}_q(f_{\alpha})}{f_{\alpha}}\in E_{\mathbb{Z},q,p}.$ Therefore, $H$ is an invertible element of $E_{\mathbb{Z},q,p}$.
\end{proof}

\begin{rema}
	Let $\alpha=a/b$ be in $\mathbb{Q}\cap(0,1]$ and $p$ be a prime number congruent to 1 modulo $b$.  By Theorem~\ref{sfF}, the $q$-hypergoemertic operator $\mathcal{H}_{q,\alpha}=\sigma_q-\frac{1-z}{1-q^{\alpha}z}$  has a strong Frobenius structure for $p$ with period 1 and $H=\frac{f_{\alpha}}{\textbf{F}_q(f_{\alpha})}$ a  Frobenius transition matrix. Writing $\mathcal{H}_{q,\alpha}$ in function of $\delta_q$ we get the operator $$\mathcal{H}_{\delta_q,\alpha}=\delta_q+\frac{z(1-q^{\alpha})(q-1)^{-1}}{1-q^{\alpha}z}.$$  We are going to see that $\mathcal{H}_{\delta_q,\alpha}$ has a strong Frobenius structure.  First of all, note that $\mathcal{H}_{\delta_q,\alpha}$  belongs to $E_{\mathbb{Z}_{q,p}}[\delta_q]$ because, it follows from Lemma~\ref{beta} that, $(1-q^{\alpha})(q-1)^{-1}$ and $q^{\alpha}$ belong to $\mathbb{Z}_{q,p}$ and as $1-q^{\alpha}z\in S$, then $z(1-q^{\alpha})/((q-1)(1-q^{\alpha}z))\in\mathbb{Z}_{q,p}[z]_S.$\\
	By definition $\mathcal{H}_{\delta_q,\alpha}$ has a strong Frobenius structure if 
	$$\sigma_qX=(-(q-1)\frac{z(1-q^{\alpha})}{(q-1)(1-q^{\alpha}z)}+1)X$$ 
	has a strong Frobenius structure. But 
	$$-(q-1)\frac{z(1-q^{\alpha})}{(q-1)(1-q^{\alpha}z)}+1=\frac{1-z}{1-q^{\alpha}z}.$$ 
	Therefore, from Theorem~\ref{sfF} we deduce that  the operator $\mathcal{H}_{\delta_q,\alpha}$  has a strong Frobenius structure with $H=f_{\alpha}/ \textbf{F}_q(f_{\alpha})$ a Frobenius transition matrix. Then, according to Theorem~\ref{confluence3}, the differential operator $$\textbf{ev}(\mathcal{H}_{\delta_q,\alpha}):=\delta+\frac{z\textbf{ev}((1-q^{\alpha})(q-1)^{-1})}{1-\textbf{ev}(q^{\alpha})z}$$ has a strong Frobenius structure with $\textbf{ev}(H)$ a Frobenius transition matrix. Lemma~\ref{beta} implies that, $\textbf{ev}((1-q^{\alpha})(q-1)^{-1})=-\alpha$ and $\textbf{ev}(q^{\alpha})=1$. Then, $\textbf{ev}(\mathcal{H}_{\delta_q,\alpha})=\delta-\frac{\alpha z}{1-z}$. %et de plus, $\widehat{e}(H)=\frac{\sum_{n\geq0}(\alpha)_nz^n}{\sum_{n\geq0}(\alpha)_nz^{np}}$, où $(\alpha)_0=1$ et $(\alpha_n)=(\alpha)\cdots(\alpha+n-1)$.
\end{rema}

\begin{coro}
	Let $\alpha=a/b$ be in $\mathbb{Q}\cap(0,1]$ and $p$ be a prime number congruent to 1 modulo $b$. We set $f_{\alpha}=\sum\limits_{n\geq0}\frac{(q^\alpha,q)_n}{(q,q)_n}z^n$.  Then,$$f_{\alpha}\equiv\left(\sum\limits_{n=0}^{p-1}\frac{(q^{\alpha},q)_n}{(q,q)_n}z^n\right)\textbf{F}_q(f_{\alpha})\bmod\phi_p(q)\mathbb{Z}_{q,p}[[z]].$$
\end{coro}

\begin{proof}
	We know that $f_{\alpha}$ is a solution of the operator $\mathcal{ H}_{q,\alpha}$ and Lemma~\ref{f} implies that $f_{\alpha}$ belongs to $\mathbb{Z}_{q,p}[[z]]$.  By Proposition~\ref{dimensionker3}, $dim_{Frac(\mathbb{Z}_{q,p})}(Ker(\mathcal{ H}_{q,\alpha},Frac(\mathbb{Z}_{q,p})[[z]]))\leq1$.  Since $f_{\alpha}\in Ker(\mathcal{ H}_{q,\alpha},Frac(\mathbb{Z}_{q,p})[[z]])$ is not zero, $dim_{Frac(\mathbb{Z}_{q,p})}Ker(L_q,Frac(\mathbb{Z}_{q,p})[[z]])=1$. According to Theorem~\ref{sfF}, the $q$-hypergeometric operator $\mathcal{H}_{q,\alpha}$ has a strong Frobenius structure for $p$ with period 1, then it follows from Proposition~\ref{actionfrob3} that $H\textbf{F}_q(f_{\alpha})$ is a solution of $\mathcal{H}_{q,\alpha}$.  So, $H\textbf{F}_q(f_{\alpha})$ and $f_{\alpha}$ belong to $Ker(L_q,Frac(\mathbb{Z}_{q,p})[[z]])$. \\
As $dim_{Frac(\mathbb{Z}_{q,p})}Ker(L_q,Frac(\mathbb{Z}_{q,p})[[z]])=1$, then there is $c\in Frac(\mathbb{Z}_{q,p})$ such that $cf_{\alpha}=H\textbf{F}_q(f_{\alpha})$. But, $c=1$ because $H(0)=\textbf{F}_q(f_{\alpha})(0) =f_{\alpha}(0)=1$. Thus, $f_{\alpha}=H\textbf{F}_q(f_{\alpha})$ and therefore,
	\begin{equation}\label{reductioncyclo}
	f_{\alpha}\equiv H\textbf{F}_q(f_{\alpha})\bmod\phi_p(q)\mathbb{Z}_{q,p}[[z]].
	\end{equation}
	  We are going to show that $H-F_1\in\phi_p(q)\mathbb{Z}_{q,p}[[z]]$, where $F_1=\sum\limits_{n=0}^{p-1}\frac{(q^{\alpha},q)_n}{(q,q)_n}z^n$.  To simplify the writing, we set $B^{(-1)}(n)=\frac{(q^{\alpha},q)_n}{(q,q)_n}$ and $B^{(0)}(n)=\textbf{F}_q(B^{(-1)}(n))$. So, $$H-F_1=\frac{\sum\limits_{n\geq0}\left(\sum\limits_{r=0}^{p-1}\left(B^{(-1)}(np+r)-B^{(0)}(n)B^{(-1)}(r)\right)z^{np+r}\right)}{\textbf{F}_q(f_{\alpha})}.$$
	As $\textbf{F}_q(f_{\alpha})(0)=1$, in order to prove that $H-F_1\in\phi_p(q)\mathbb{Z}_{q,p}[[z]]$ it suffices to see that $$\sum\limits_{n\geq0}\left(\sum\limits_{r=0}^{p-1}\left(B^{(-1)}(np+r)-B^{(0)}(n)B^{(-1)}(r)\right)z^{np+r}\right)\in\phi_p(q)\mathbb{Z}_{q,p}[[z]].$$
	Let $n$ be a non-negative integer and $r$ be in $\{0,1,\ldots,p-1\}$. Then, we have
	$$\left(B^{(-1)}(np+r)-B^{(0)}(n)B^{(-1)}(r)\right)=B^{(0)}(n)\left(\frac{B^{(-1)}(np+r)}{B^{(0)}(n)}-B^{(-1)}(r)\right).$$
	From Equation~\eqref{eqn:produit}, we conclude that $$\frac{B^{(-1)}(np+r)}{B^{(0)}(n)B^{(-1)}(r)}\equiv 1\bmod\phi_p(q)\mathbb{Z}_{q,p}.$$
	Then, $$\left(B^{(-1)}(np+r)-B^{(0)}(n)B^{(-1)}(r)\right)\equiv0\bmod\phi_p(q)\mathbb{Z}_{q,p}.$$
	Consequently, $$H-F_1\in\phi_p(q)\mathbb{Z}_{q,p}[[z]].$$
	For this reason $H\equiv F_1\bmod\phi_p(q)\mathbb{Z}_{q,p}[[z]]$. By definition $F_1=\sum_{n=0}^{p-1}\left(\frac{(q^{\alpha},q)_n}{(q,q)_n}\right)z^n$. Then, it follows from Equation~\eqref{reductioncyclo} that, $$f_{\alpha}\equiv\left(\sum\limits_{n=0}^{p-1}\frac{(q^{\alpha},q)_n}{(q,q)_n}z^n\right)\textbf{F}_q(f_{\alpha})\bmod\phi_p(q)\mathbb{Z}_{q,p}[[z]].$$ %Puisque $\phi_p(q)\in\mathfrak{q}$ alors $H_{\mid\mathfrak{q}}=\sum\limits_{n=0}^{p-1}\left(\frac{(q^{\alpha},q)_n}{(q,q)_n}\mod\mathfrak{q}\right)z^n$. Notons que $H_{\mid\mathfrak{q}}$ est non nul car $H_{\mid\mathfrak{q}}(0)=1$.  Alors, d'après le deuxième point du théorème ~\ref{alg3}, il existe $A_0(z)$, $A_1(z)\in\mathbb{Z}_{q,p}[z]$ non nuls tels que  $$A_0(z)f_{\alpha}+A_1(z)\Gamma_q(f_{\alpha})\equiv0\mod\mathfrak{q}\mathbb{Z}_{q,p}[[z]].$$
	%Il suit de l'équation que, $A_0(z)\mod\mathfrak{q}=1$ et $A_1(z)\mod\mathfrak{q}=-\sum\limits_{n=0}^{p-1}\left(\frac{(q^{\alpha},q)_n}{(q,q)_n}\mod\mathfrak{q}\right)z^n$. Donc, $$f_{\alpha}\equiv\left(\sum\limits_{n=0}^{p-1}\frac{(q^{\alpha},q)_n}{(q,q)_n}z^n\right)\Gamma_q(f_{\alpha})\mod\mathfrak{q}\mathbb{Z}_{q,p}[[z]].$$ 
\end{proof}

Recall that Theorem~\ref{sfF} has been proved by using Lemma~\ref{elementanalytique}. The rest of this section is devoted to prove Lemma~\ref{elementanalytique}. The proof of Lemma~\ref{elementanalytique} rests on Lemma~\ref{dwork} which is stated below. 

\begin{lemm}\label{dwork}
	Let $\{B^{(i)}(k)\}_{k\in\mathbb{N},i\geq-1}$ be a set of sequences with coefficients in $\mathbb{Z}_{q,p}$. We set $$F(z)=\sum_{k\geq0}B^{(-1)}(k)z^k,\quad G(z)=\sum_{k\geq0}B^{(0)}(k)z^k.$$ 
	
	Suppose that the following conditions are satisfied.
	\begin{enumerate}
		\item For every $n\in\mathbb{N}$ and for every integer $i\geq-1$, $\frac{B^{(i)}(n)}{B^{(i+1)}(\lfloor{n/p}\rfloor)}\in\mathbb{Z}_{q,p}$.
		\item For every $(m,n,s)\in\mathbb{N}^3$ and for every integer $i\geq-1$, $$\frac{B^{(i)}(n+mp^{s+1})}{B^{(i+1)}(\lfloor{n/p}\rfloor+mp^s)}\equiv\frac{B^{(i)}(n)}{B^{(i+1)}(\lfloor{n/p}\rfloor)}\mod\mathfrak{m}^{s+1}.$$
		\item For every $n\in\mathbb{N}$, $B^{(i)}(n)\in\mathbb{Z}_{q,p}$
		\item For every integer $i\geq-1$ $B^{(i)}(0)\in\mathbb{Z}^*_{q,p}$.
	\end{enumerate}
	
	\smallskip
	
	Then for every $(r,s)\in\mathbb{N}^2$, $$F(z)\sum_{k=rp^s}^{(r+1)p^s-1}B^{(0)}(k)z^{pk}\equiv G(z^p)\sum_{k=rp^{s+1}}^{(r+1)p^{s+1}-1}B^{(-1)}(k)z^k\bmod B^{(s)}(r)\mathfrak{m}^{s+1}$$
\end{lemm}

 Lemma~\ref{dwork} is essentially a result due to Dwork \cite[Theorem 2]{padiccycles} that we have stated in our context.   We omit the proof of this lemma since the arguments used by Dwork in his proof  remain valid in our context.

\begin{proof}[Proof of Lemma~\ref{elementanalytique}]
	
	Let $\alpha=a/b$ be in $\mathbb{Q}\cap(0,1]$ and $p$ be a prime number congruent to 1 modulo $b$.  We are going to show that the set $\{B^{(i)}(k)\}_{k\in\mathbb{N},i\geq-1}$, with $B^{(-1)}(n)=\frac{(q^{\alpha},q)_n}{(q,q)_n}$ and for every $i\geq0$, $B^{(i)}(n)=\textbf{F}^{i+1}_q(B^{(-1)}(n))$ satisfies the conditions (i)-(iv) from Lemma~\ref{dwork}. 
	
	By Lemma~\ref{f}, we see that for every integer $n\geq0$, $B^{(-1)}(n)\in\mathbb{Z}_{q,p}$ and since $\textbf{F}_q$ is an endomorphism of  $\mathbb{Z}_{q,p}$, it follows that, for every  $i,n\geq0$, $B^{(i)}(n)\in\mathbb{Z}_{q,p}$. Therefore, the condition (iii) is satisfied. By definition, for every $i\geq-1$, $B^{(i)}(0)=1$. So, the condition (iv) is also satisfied. In the remainder we  show the conditions (i) and (ii) are also satisfied. Let $j$ be the unique integer in $\{1\ldots,p-1\}$ such that $\alpha+j=p\alpha$.
	
	\smallskip
	
	\emph{Proof of (i)}.  Let $n$ be in $\mathbb{N}$ and $i$ be in $\mathbb{Z}_{\geq-1}$. We want to show that $\frac{B^{(i)}(n)}{B^{(i+1)}(\lfloor{n/p}\rfloor)}\in\mathbb{Z}_{q,p}$.  For this, we are going to show by induction on $i\geq-1$ that, for every $n\geq0$, $\frac{B^{(i)}(n)}{B^{(i+1)}(\lfloor{n/p}\rfloor))}\in\mathbb{Z}_{q,p}$.  The point 3 of Lemma~\ref{f} implies that, for every integer $n\geq0$, $\frac{B^{(-1)}(n)}{B^{(0)}(\lfloor{n/p}\rfloor)}\in\mathbb{Z}_{q,p}$.  Now, let $i$ be in $ \mathbb{Z}_{\geq-1}$ and suppose that, for every non-negative integer $n$, $\frac{B^{(i)}(n)}{B^{(i+1)}(\lfloor{n/p}\rfloor)}\in\mathbb{Z}_{q,p}$. Then, there is $\beta_{i+1}\in\mathbb{Z}_{q,p}$ such that $$B^{(i)}(n)=\beta_{i+1}B^{(i+1)}(\lfloor{n/p}\rfloor).$$
	As $\textbf{F}_q:\mathbb{Z}_{q,p}\rightarrow\mathbb{Z}_{q,p}$ is an endomorphism of rings, then 
	
	$$\textbf{F}_q(B^{(i)}(n))=\textbf{F}_q(\beta_{i+1})\textbf{F}_q(B^{i+1}(\lfloor{n/p}\rfloor)).$$
	By definition $B^{(i+1)}(n)=\textbf{F}_q(B^{(i)}(n))$ and $B^{(i+2)}(\lfloor{n/p}\rfloor)=\textbf{F}_q(B^{(i+1)}(\lfloor{n/p}\rfloor))$. Therefore, for every non-negative integer $n$, $\frac{B^{(i+1)}(n)}{B^({i+2})(\lfloor{n/p}\rfloor)}\in\mathbb{Z}_{q,p}$. 
	
So, we conclude by induction on $i\geq-1$ that, for every non-negative integer $n$, $$\frac{B^{(i)}(n)}{B^{(i+1)}(\lfloor{n/p}\rfloor)}\in\mathbb{Z}_{q,p}.$$
	
	\emph{Proof of (ii)}.  Let $(n,m,s)$ be in $\mathbb{N}^3$ and $i$ be in $\mathbb{Z}_{\geq-1}$.  We want to show that $$\frac{B^{(i)}(n+mp^{s+1})}{B^{(i+1)}(\lfloor{n/p}\rfloor+mp^s)}\equiv\frac{B^{(i)}(n)}{B^{(i+1)}(\lfloor{n/p}\rfloor)}\mod\mathfrak{m}^{s+1}.$$
	First of all, we are going to show that this last equality  is true for $(n,m,s)\in\mathbb{N}^3$ and $i=-1$.  That is, we have to prove that for every $(n,m,s)\in\mathbb{N}^3$,
	\begin{equation}\label{pasobase}
	\frac{B^{(-1)}(n+mp^{s+1})}{B^{(0)}(\lfloor{n/p}\rfloor+mp^{s})}\equiv\frac{B^{(-1)}(n)}{B^{(0)}(\lfloor{n/p}\rfloor)}\mod\mathfrak{m}^{s+1}.
	\end{equation}
	It is clear that Equation~\eqref{pasobase} is verified for $m=0$ and $(s,n)\in\mathbb{N}^2$. Then,  we are going to show that \eqref{pasobase} is true for every $(n,s)\in\mathbb{N}^2$ and for every integer $m>0$. Note that
	
	$$\frac{B^{(-1)}(n+mp^{s+1}))}{B^{(0)}(\lfloor{n/p}\rfloor+mp^{s})}-\frac{B^{(-1)}(n)}{B^{(0)}(\lfloor{n/p}\rfloor)}=\frac{B^{(-1)}(n)}{B^{(0)}(\lfloor{n/p}\rfloor)}\left[\frac{B^{(-1)}(n+mp^{s+1}))B^{(0)}(\lfloor{n/p}\rfloor)}{B^{(0)}(\lfloor{n/p}\rfloor+mp^{s}))B^{(-1)}(n)}-1\right].$$
	From (i) we know that $\frac{B^{(-1)}(n)}{B^{(0)}(\lfloor{n/p}\rfloor)}\in\mathbb{Z}_{q,p}$. Then, in order to prove\eqref{pasobase} it suffices to prove that
	\begin{equation}\label{reduction2}
	\frac{B^{(-1)}(n+mp^{s+1})B^{(0)}(\lfloor{n/p}\rfloor)}{B^{(0)}(\lfloor{n/p}\rfloor+mp^s)B^{(-1)}(n)}-1\in\mathfrak{m}^{s+1}.
	\end{equation}
	
	We are going to see that \eqref{reduction2} is satisfied for every $(n,s)\in\mathbb{N}^2$ and for every integer $m>0$. Indeed, for $(n,s)\in\mathbb{N}^2$ we have 
	$$n=a_{s}p^s+a_{s-1}p^{s-1}+\cdots+a_1p+a_0,$$ 
	where $a_s\in\mathbb{N}$ and $a_{s-1},\ldots, a_1,a_0\in\{0,\ldots, p-1\}$. 
	We set $$l=n+mp^{s+1}\quad\textup{and}\quad l_1=\lfloor{\frac{n+mp^{s+1}}{p}}\rfloor=\lfloor{n/p}\rfloor+mp^{s}.$$	
	Recall that by definition $B^{(0)}(\lfloor{n/p}\rfloor+mp^{s})= \textbf{F}_q(B^{(-1)}(\lfloor{n/p}\rfloor+mp^{s}))$. Then, it follows from \eqref{reduction} that
	\begin{equation}\label{4}
	\frac{B^{(-1)}(n+mp^{s+1})}{B^{(0)}(\lfloor{n/p}\rfloor+mp^{s})}=\frac{\prod\limits_{k=0}^{l_1-1}\prod\limits_{i=0,i\neq j}^{p-1}(1-q^{\alpha+kp+i})}{\prod\limits_{k=0}^{l_1-1}\prod\limits_{i=1}^{p-1}(1-q^{kp+i})}\cdot
	\frac{(1-q^{\alpha+l_1p})\cdots(1-q^{\alpha+l_1p+a_0-1})}{(1-q^{l_1p+1})\cdots(1-q^{l_1p+a_0})}.
	\end{equation}
	We set $n_1=\lfloor{n/p}\rfloor$. Again, remember that by definition $B^{(0)}(\lfloor{n/p}\rfloor)= \textbf{F}_q(B^{(-1)}(\lfloor{n/p}\rfloor))$. So according to \eqref{reduction} we have 
	\begin{equation}\label{5}
	\frac{B^{(-1)}(n)}{B^{(0)}(\lfloor{n/p}\rfloor))}=\frac{\prod\limits_{k=0}^{n_1-1}\prod\limits_{i=0,i\neq j}^{p-1}(1-q^{\alpha+kp+i})}{\prod\limits_{k=0}^{n_1-1}\prod\limits_{i=1}^{p-1}(1-q^{kp+i})}
	\frac{(1-q^{\alpha+n_1p})\cdots(1-q^{\alpha+n_1p+a_0-1})}{(1-q^{n_1p+1})\cdots(1-q^{n_1p+a_0})}.
	\end{equation}	
	As $m>0$ then  $n_1< l_1$. So,  we obtain from \eqref{4} and from \eqref{5} that
	\begin{align*}
	&\frac{B^{(-1)}(n+mp^{s+1})B^{(0)}(\lfloor{n/p}\rfloor)}{B^{(0)}(\lfloor{n/p}\rfloor+mp^s)B^{(-1)}(n)}=\\
	&\frac{\prod\limits_{k=n_1+1}^{l_1-1}\prod\limits_{i=0,i\neq j}^{p-1}(1-q^{\alpha+kp+i})}{\prod\limits_{k=n_1+1}^{l_1-1}\prod\limits_{i=1}^{p-1}(1-q^{kp+i})}\frac{\prod\limits_{i=a_0,i\neq j}^{p-1}(1-q^{\alpha+n_1p+i})}{\prod\limits_{i=a_0+1}^{p-1}(1-q^{n_1p+i})}\frac{(1-q^{\alpha+l_1p})\cdots(1-q^{\alpha+l_1p+a_0-1})}{(1-q^{l_1p+1})\cdots(1-q^{l_1p+a_0})}.
	\end{align*}	
	Now, we consider the following sets:
	\begin{multline*}
	\mathcal{B}':=\{kp+i:\text{ }n_1+1\leq k\leq l_1-1,\text{ }1\leq i\leq p-1\}\\\cup\{n_1p+i: a_0+1\leq i\leq p-1\}\cup\{l_1p+i:\text{ }1\leq i\leq a_0\},
	\end{multline*}
	\begin{multline*}
	\mathcal{A'}:=\{\alpha+kp+i:\text{ }n_1+1\leq k\leq l_1-1,\text{ }0\leq i\leq p-1,\text{ }i\neq j\}\\\cup\{\alpha+n_1p+i: a_0\leq i\leq p-1,\text{ }i\neq j\}\cup\{\alpha+l_1p+i:\text{ }0\leq i\leq a_0-1\}.
	\end{multline*}
	Remember that $j$ is the unique integer in $\{1,\ldots, p-1\}$ such that $\alpha+j=p\alpha$.
	
	Then,
	\begin{equation}\label{egal}
	\frac{B^{(-1)}(n+mp^{s+1})B^{(0)}(\lfloor{n/p}\rfloor)}{B^{(0)}(\lfloor{n/p}\rfloor+mp^s)B^{(-1)}(n)}=\prod_{r\in\mathcal{A}'}(1-q^r)\prod_{t\in\mathcal{B}'}(1-q^t)^{-1}.
	\end{equation}
	To prove Equation~\eqref{reduction2} is equivalent to show that \begin{equation*}
	\prod_{r\in\mathcal{A}'}(1-q^r)\prod_{t\in\mathcal{B}'}(1-q^t)^{-1}-1\in\mathfrak{m}^{s+1}.
	\end{equation*}
	It is clear that if $r\in\mathcal{A}'$ and $t\in\mathcal{B}'$ then $r,t\notin p\mathbb{Z}_p$.  Now, we consider the following sets:
	$$\mathcal{B}:=\{n_1p+a_0+1,n_1p+a_0+2,\ldots, l_1p+a_0\},$$ $$\mathcal{A}:=\{\alpha+n_1p+a_0,\alpha+n_1p+a_0+1,\ldots,\alpha+l_1p+a_0-1\}.$$
	Note that $\mathcal{A}'\subset\mathcal{A}$ and $\mathcal{B}'\subset\mathcal{B}$. Moreover, $\mathcal{A}$ and $\mathcal{B}$ have the same cardinal which is equal to $mp^{s+1}$.  Since $\alpha \in\mathbb{Z}_p$, we have

	\begin{equation}\label{modulo}
	\mathcal{A}\mod mp^{s+1}=\mathbb{Z}/(mp^{s+1}\mathbb{Z})=\mathcal{B}\mod mp^{s+1}.
	\end{equation}
	Let $r$ be in $\mathcal{A}'$. On the one hand, it follows from~\eqref{modulo} that there exists a unique $\tau(r)\in\mathcal{B}$ such that $r-\tau(r)\in mp^{s+1}\mathbb{Z}_p$ and the other hand, as $r\notin p\mathbb{Z}_p$ then $\tau(r)\notin p\mathbb{Z}_p$. Therefore, $\tau(t)\in\mathcal{B}'$. We are going to see that $$\frac{1-q^{r}}{1-q^{\tau(r)}}\equiv 1\mod\phi_p(q)\cdots\phi_p(q^{p^{s}})\mathbb{Z}_{q,p}.$$  As $r-\tau(r)\in mp^{s+1}\mathbb{Z}_p$, then there is $\mu\in\mathbb{Z}_p$ such that $r-\tau(r)=mp^{s+1}\mu$ and consequently we have
	\begin{equation}\label{ab}
	1-q^{r}-(1-q^{\tau(r)})=q^{\tau(r)}-q^r=q^r(q^{r-\tau(r)}-1)=q^{r}(q^{mp^{s+1}\mu}-1).
	\end{equation}
	First of all, we prove that $q^{mp^{s+1}\mu}-1\in\mathfrak{m}^{s+1}\mathbb{Z}_{q,p}$.  Let $k$ be the $p$-adic valuation of $mp^{s+1}\mu$. Then $k\geq s+1$ and $$mp^{s+1}\mu=\sum_{n\geq k}s_np^n,$$
	where $s_k\in\{1,\ldots, p-1\}$ and $s_n\in\{0,1,\ldots,p-1\}$ for $n>k$. For every $n\geq k$, we set $r_n=~\sum_{j=k}^ns_np^n$. Therefore, $$q^{mp^{s+1}\mu}=q^{r_k}+q^{r_k}(q^{s_{k+1}p^{k+1}}-1)+q^{r_{k+1}}(q^{s_{k+2}}p^{k+2}-1)+\cdots+q^{r_{n-1}}(q^{s_{n}p^{n}}-1)+\cdots.$$
	Then, $$q^{mp^{s+1}\mu}-1=q^{s_kp^k}-1+q^{r_k}(q^{s_{r+1}p^{k+1}}-1)+q^{r_{k+1}}(q^{s_{r+2}}p^{k+2}-1)+\cdots+q^{r_{n-1}}(q^{s_{n}p^{n}}-1)+\cdots.$$
	From Lemma~\ref{norme} we have
	\begin{equation}\label{r}
	1-q^{p^k}=(1-q)\phi_p(q)\cdots\phi_p(q^{p^{k-1}}).
	\end{equation}
	Since for every $n\geq k$, $p^k$ divides $s_kp^k$, the polynomial $1-q^{p^k}$ divides the polynomial $q^{s_np^n}-~1$ in $\mathbb{Z}[q]$.  Then, for every $n\geq k$, there is $t_n(q)\in\mathbb{Z}[q]$ such that
	\begin{equation}\label{n}
	q^{s_np^n}-1=t_n(q)(1-q)\phi_p(q)\cdots\phi_p(q^{p^{k-1}}).
	\end{equation}
	So, it follows from~\eqref{ab} that $$1-q^r=1-q^{\tau(r)}+q^r(1-q)\phi_p(q)\cdots\phi_p(q^{p^{k-1}})\left[t_k(q)+q^{r_k}t_{k+1}(q)+\cdots+q^{r_{n-1}}t_n(q)+\cdots\right].$$
	As   $p$ does not divide $\tau(r)$ then $\frac{1}{1+q+\cdots+q^{\tau(r)-1}}\in\mathbb{Z}[q]_{\mathfrak{m}}$ and as $1-q^{\tau(r)}=(1-q)(1+q+\cdots+q^{\tau(r)-1})$ , then
	\begin{multline}\label{abc}
	\frac{1-q^r}{1-q^{\tau(r)}}=1+\frac{q^r\phi_p(q)\cdots\phi_p(q^{p^{r-1}})}{1+q+\cdots+q^{\tau(r)-1}}\left[t_k(q)+q^{r_k}t_{k+1}(q)+\cdots+q^{r_{n-1}}t_n(q)+\cdots\right].
	\end{multline}
	Following~\eqref{r} and~\eqref{n} we have $t_n(q)=\frac{q^{s_np^n}-1}{1-q^{p^k}}$.  The polynomial $q^{p^n}-1$ divides $q^{s_np^n}-1$ and, by Lemma~\ref{norme}, we know that the $\mathfrak{m}$-adic norm of $q^{p^n}-1$ is less than or equal to $1/p^{n+1}$,  then the $\mathfrak{m}$-adic norm of $q^{s_np^n}-1$ is less than or equal to $1/p^{n+1}$. Therefore, the $\mathfrak{m}$-adic norm of $t_n(q)$ is less than or equal to $1/p^{n+1-f}$, where $1/p^f$ is the $\mathfrak{m}$-adic norm of $1-q^{p^k}$. Since $\lim_{n\rightarrow\infty}1/p^{n+1-f}=0$ and the $\mathfrak{m}$-adic norm is ultrametric, we have  $$t_k(q)+q^{r_k}t_{k+1}(q)+\cdots+q^{r_{n-1}}t_n(q)+\cdots\in\mathbb{Z}_{q,p}.$$
	 Therefore, it follows from~\eqref{abc} that, $$\frac{1-q^r}{1-q^{\tau(r)}}\equiv 1\mod\phi_p(q)\cdots\phi_p(q^{p^{k-1}})\mathbb{Z}_{q,p}.$$
	Since $k\geq s+1$, $$\frac{1-q^r}{1-q^{\tau(r)}}\equiv 1\mod\phi_p(q)\cdots\phi_p(q^{p^{s}})\mathbb{Z}_{q,p}.$$ Since $\mathcal{A}'$ and $\mathcal{B}'$ have the same cardinal, and since moreover $\tau:\mathcal{A}'\rightarrow\mathcal{B}'$ is injective, we have 
	$$\prod_{r\in\mathcal{A}'}(1-q^r)\prod_{t\in\mathcal{B}'}(1-q^t)^{-1}-1\in\phi_p(q)\cdots\phi_p(q^{p^s})\mathbb{Z}_{q,p}.$$ Then, we get from~\eqref{egal} that,  for the arbitrary couple $(n,s)\in\mathbb{N}^2$ and the arbitrary positive integer $m$, \begin{equation}\label{eqn:produit}
	\frac{B^{(-1)}(n+mp^{s+1})B^{(0)}(\lfloor{n/p}\rfloor))}{B^{(0)}(\lfloor{n/p}\rfloor+mp^s))B^{(-1)}(n)}-1\in\phi_p(q)\cdots\phi_p(q^{p^s})\mathbb{Z}_{q,p}.
	\end{equation}  
	But, we have already showed that for every positive integer $j$, $\phi_p(q^{p^j})\in\mathfrak{m}$. Thus, we have $\phi_p(q)\cdots\phi_p(q^{p^s})\in\mathfrak{m}^{s+1}$. For this reason, Equation~\eqref{eqn:produit} implies that for every couple $(n,s)\in\mathbb{N}^2$ and for every positive integer $m$, $$\frac{B^{(-1)}(n+mp^{s+1})B^{(0)}(\lfloor{n/p}\rfloor))}{B^{(0)}(\lfloor{n/p}\rfloor+mp^s))B^{(-1)}(n)}\equiv1\mod\mathfrak{m}^{s+1}\mathbb{Z}_{q,p}.$$ 
	Therefore, Equation~\eqref{pasobase} is satisfied for every $(n,m,s)\in\mathbb{N}^3$.\\
	Finally,  we are going to show by induction on $i\geq-1$ that, for every $(n,m,s)\in\mathbb{N}^3$, $$\frac{B^{(i)}(n+mp^{s+1})}{B^{(i+1)}(\lfloor{n/p}\rfloor+mp^s)}\equiv\frac{B^{(i)}(n)}{B^{(i+1)}(\lfloor{n/p}\rfloor)}\mod\mathfrak{m}^{s+1}.$$
	We have already seen that it is true for $i=-1$ because \eqref{pasobase} is satisfied for every $(n,m,s)\in\mathbb{N}^3$. Suppose that for some $i\geq-1$ and for every $(n,m,s)\in\mathbb{N}^3$ we have
	\begin{equation}\label{pasoinductivo}
	\frac{B^{(i)}(n+mp^{s+1})}{B^{(i+1)})(\lfloor{n/p}\rfloor+mp^s)}-\frac{B^{(i)}(n)}{B^{(i+1)}(\lfloor{n/p}\rfloor)}\in\mathfrak{m}^{s+1}.
	\end{equation} 
	As $\textbf{F}_q(\mathfrak{m})\subset\mathfrak{m}$ then $\textbf{F}_q(\mathfrak{m}^{s+1})\subset\mathfrak{m}^{s+1}$ and therefore, by Equation~\eqref{pasoinductivo}, we obtain
	$$\frac{\textbf{F}_q(B^{(i)}(n+mp^{s+1}))}{\textbf{F}_q(B^{(i+1)})(\lfloor{n/p}\rfloor+mp^s)}-\frac{\textbf{F}_q(B^{(i)}(n))}{\textbf{F}_q(B^{(i+1)}(\lfloor{n/p}\rfloor))}\in\mathfrak{m}^{s+1}.$$ 
	So, $$\frac{B^{(i+1)}(n+mp^{s+1})}{B^{(i+2)}(\lfloor{n/p}\rfloor+mp^s)}\equiv\frac{B^{(i+1)}(n)}{B^{(i+2)}(\lfloor{n/p}\rfloor)}\mod\mathfrak{m}^{s+1}$$ 
	because by definition, for every positive integer $j$, $\textbf{F}_q(B^{(i)}(j))=B^{(i+1)}(j)$. \\
	Therefore, we have proved that the set of sequences $\{B^{(i)}(n)\}_{n\in\mathbb{N},i\geq-1}$, where $B^{(-1)}(n)=\frac{(q,q^{\alpha})_n}{(q,q)_n}$ and $B^{(i)}(n)=\textbf{F}^{i+1}_q(B^{(-1)}(n))$ satisfied the conditions (i)-(iv) from Lemma~\ref{dwork}. This puts us in a position to apply Lemma~\ref{dwork} to deduce that, for every $s\in\mathbb{N}$,
	\begin{small}
		$$\left(\sum_{k\geq0}B^{(-1)}(k)z^k\right)\left(\sum_{k=0}^{p^s-1}B^{(0)}(k)z^{pk}\right)\equiv\left(\sum_{k\geq0}B^{(0)}(k)z^{kp}\right)\left(\sum_{k=0}^{p^{s+1}-1}B^{(-1)}z^{k}\right)\in\mathfrak{m}^{s+1}.$$	
	\end{small}
	This last equality is equivalent to say that, for every  $s\in\mathbb{N}$, $$f_{\alpha}\textbf{F}_q(F_s)\equiv \textbf{F}_q(f_{\alpha})F_{s+1}\mod\mathfrak{m}^{s+1}.$$
\end{proof}

\subsection{Open Question}

Given a prime number $p$ and $\underline{\alpha}=(\alpha_1,\ldots,\alpha_n)$, $\underline{\beta}=(\beta_1,\ldots,\beta_n)\in(\mathbb{Q}\setminus\mathbb{Z}_{\leq0})^n$, we asked the following question:  does the $q$-hypergeometric operator $\mathcal{H}_{\delta_q}(\underline{\alpha},\underline{\beta})$ have a strong Frobenius structure for $p$?\\
According to \cite[ Theorem 6.1]{vmsff}, if $\alpha_i-\beta_j\notin\mathbb{Z}$ for all $i,j\in\{1,\ldots,n\}$ and if for all $i,j\in\{1,\ldots,n\}$, the $p$-adic norm of $\alpha_i$ and $\beta_j$ is less than or equal to 1, then  $\mathcal{H}(\underline{\alpha},\underline{\beta})$  has a strong Frobenius structure for $p$. So,  we expect that the $q$-hypergeometric operator $\mathcal{H}_{\delta_q}(\underline{\alpha},\underline{\beta})$ to have a strong Frobenius structure for almost every prime number $p$ when $\alpha_i-\beta_j\notin\mathbb{Z}$  for all $i,j\in\{1,\ldots,n\}$.\\
The proof of Theorem 6.1 of \cite{vmsff} relies essentially on the rigidity of the operator $\mathcal{H}(\underline{\alpha},\underline{\beta})$. In \cite{julien} Roques gave a definition of rigidity for $q$-difference operators  and he showed that under this definition the $q$-hypergeometric operator $\mathcal{H}_{\delta_q}(\underline{\alpha},\underline{\beta})$ is rigid.  This leads us to think that a way to determine whether  $\mathcal{H}_{\delta_q}(\underline{\alpha},\underline{\beta})$ has a strong Frobenius structure for almost all prime number $p$ when  $\alpha_i-\beta_j\notin\mathbb{Z}$ for all $i,j\in\{1,\ldots,n\}$ would be to use the notion of rigidity given in \cite{julien}.

\section{The work of André and Di Vizio}\label{sec_andre_divizio}

In \cite{anredivzio}, André and Di Vizio introduced a definition of a strong Frobenius structure for $q$-difference operators in the case where $q$ is an element of an ultrametric field satisfying certain conditions. The purpose of this section is to see that this definition does not seem to be appropriate to answer our questions asked in the introduction,  even though the results obtained in \cite{anredivzio} through their definition are deep. 

\smallskip

Let $K$ be a field of characteristic zero complete with respect to an ultrametric norm $|\cdot|$ and let $k$ be the residue field of $K$. We suppose that the characteristic of $k$ is $p>0$.  For every $\epsilon\in(0,1)$, put $I_{\epsilon}:=(1-\epsilon,1)$ and $$\mathcal{A}_{K}(I_{\epsilon})=\left\{\sum_{i\in\mathbb{Z}}a_iz^i:\text{ tel que }a_i\in K,\lim\limits_{i\rightarrow\pm\infty}|a_i|r^i=0, \text{ pour tout } r\in I_{\epsilon}\right\}.$$

The Robba ring is the ring $\mathcal{R}_K=\bigcup_{\epsilon>0}\mathcal{A}_K(I_{\epsilon})$. Let $q$ be in $K$ such that $q$ is not a root of unity and $|1-q|<1$. Let $\tau$ be a Frobenius endomorphism of $K$ verifying $\tau(q)=q$. In \cite{anredivzio}, André and Di Vizio gave a definition of strong Frobenius structure for the $\sigma_q$-modules with coefficients in $\mathcal{R}_K$.  Let  $\sigma_q-Mod^{(\phi)}_{\mathcal{R}_K}$ be the category of $\sigma_q$-modules with coefficients in $\mathcal{R}_K$ having a strong Frobenius structure and let $d-Mod^{(\phi)}_{\mathcal{R}_K}$ be the category whose objects are the differential modules with coefficients in $\mathcal{R}_K$ having a strong Frobenius structure. André and Di Vizio showed in \cite{anredivzio}  a local monodromy theorem which turns out  to be the $q$-analog of Crew's quasi-unipotence conjecture. This  allows them to show that if $k$ is algebraically closed, then there is a functor  $\text{Conf}: \sigma_q-Mod^{(\phi)}_{\mathcal{R}_K}\longrightarrow d-Mod^{(\phi)}_{\mathcal{R}_K}$ and they show that the functor $ \text{Conf}$ is an equivalence. Further,   Pulita \cite{pulita} showed that the functor $\text{Conf}$ can be extended to the category  of \emph{admissibles} modules, which shows that the construction made by André and Di Vizio does not depend \textit{a posteriori} on the strong Frobenius structure. Moreover, following Pulita \cite{pulita},  we see that the functor  $\text{Conf}$  sends  a $q$-difference operator with coefficients in $\mathcal{R}_K$ to a differential operator with coefficients in $\mathcal{R}_K$ having the same solutions in $\widetilde{\mathcal{R}_K}[Log z]$, where $\widetilde{\mathcal{R}_K}$ is an extension of $\mathcal{R}_K$, that is to say $\text{Conf}$ preserves the solutions. The fact that this functor preserves the solutions is an obstacle to answer to question regarding the confluence property. In fact, on the one hand, as the $\text{Conf}$ preserves the solutions then the strong Frobenius structure of  $M$ and the strong Frobenius structure of $\text{Conf}(M)$ are the same because the Frobenius transitions matrix for $M$ and $ \text{Conf}(M)$ are the same. On the other hand, even if we do not know if the operator $\mathcal{ H}_{\delta_q}(\underline{\alpha},\underline{\beta})$ has a strong Frobenius structure under the definition proposed in \cite{anredivzio}, we will  have  \textit{a posteriori}  $\text{Conf}(\mathcal{ H}_{\delta_q}(\underline{\alpha},\underline{\beta}))\neq\mathcal{ H}(\underline{\alpha},\underline{\beta})$ because $\mathcal{ H}_{\delta_q}(\underline{\alpha},\underline{\beta})$ and $\mathcal{ H}(\underline{\alpha},\underline{\beta})$ do not have the same solutions. Therefore, by using this approach, we cannot answer the open question asked in the previous section.  Here, we remind the reader that under our construction we have $ \textbf{ev}(\mathcal{H}_q(\underline{\alpha},\underline{\beta}))=\mathcal{H}(\underline{\alpha},\underline{\beta})$.\\
 The question concerning the congruences modulo the $p$-th cyclotomic polynomial also seems difficult to answer in a way affirmative for the following reason:   if $f(q,z)\in\mathbb{Z}[q][[z]]$ is a solution of a $q$-difference operator belonging to $\sigma_q-Mod^{(\phi)}_{\mathcal{R}_K}$, then $f(q,z)\in\mathbb{Z}[q][[z]]$ is  a solution of a differential operator having a strong Frobenius structure for $p$ because $\text{Conf}$ preserves the solutions. So,  it follows from  \cite[Theorem 2.1]{vmsff} and from the equality $\tau(q)=q$ that there are $a_0(z),\ldots, a_{n^2}(z)\in\mathbb{Z}[z]$ such that $a_0(z)Y_0+a_1(z)Y_1+\cdots+a_{n^2}Y_{n^2}\mod p$ is not zero and $$\sum_{i=0}^{n^2}a_i(z)f(1,z^{p^{ih}})\equiv0\mod p\mathbb{Z}[[z]].$$
Therefore, we recover a congruence of type  \eqref{eq:thm2_1} but we do not get a congruence of type  $\sum_{i=0}^{n^2}d_i(q,z)f(q^{p^{ih_p}},z^{p^{ih_p}})\equiv0\mod[p]_q\mathbb{Z}[q][[z]]$.  The limitation to have a congruence of this type is the supposition $\tau(q)=q$.

Institute of Mathematics of the Polish Academy of Sciences.\\
\emph{Email address:} devargasmontoya@impan.pl
% Important: Do not put any empty line here.
%\affiliationtwo{% in this example, one author has two addresses}
%   T. Hird\\
%   Previous postal address where
%     the research was performed and\\
%   Country
%   \email{hird@university.ac.uk}}
% Important: Do not put any empty line here.
% Use \affiliationthree{} for any address positioned under \affiliationone
% Use \affiliationfour{}  for any address positioned under \affiliationtwo
%\affiliationthree{~} %inserts a space to make this field empty
%\affiliationfour{%
%   Current address:\\
%   Present long-term address\\
%   Country
 %  \email{t.hird@institution.edu}}
%

%% French abstract
%\begin{altabstract}
%Ceci est le r\'esum\'e fran\c cais.
%\end{altabstract}


\begin{thebibliography}{xx}% Replace 9 by 99 if 10 or more references
%
% Please note the use of "\and" between author names below
%
\bibitem{cyclotomique}
{\sc B. Adamczewski, J. Bell, E. Delaygue and F. Jouhet,}
{\it Congruences modulo cyclotomic polynomials and algebraic independece for q-series,}
Sém. Lothar. Combin., \textbf{78B} (2017), Proceedings of the 29th International Conference FPSAC, 12 pages.

\bibitem{ABD2019}
	{\sc B. Adamczewski, J. Bell, and E. Delaygue,}
	{\it Algebraic independence of G-functions and congruences "à la Lucas'',}
	Ann. Sci. \'Ec. Norm. Sup\'er \textbf{52} (2019), 515--559.

\bibitem{atiyah}
{\sc M.F. Atiyah and I.G. Macdonald,}
{\it An introduction to commutative algebra.}



\bibitem{andre}
	{\sc Y. Andr\'e,}
	{G-functions and geometry,}	
	Aspects of Mathematics, E13. Friedr. Vieweg \& Sohn, Braunschweig, 1989.
	
	

\bibitem{anredivzio}
{\sc Y. André and L. Di Vizio,}
{\it q-difference equations and p-adic local monodromy,}
Astéristique \textbf{296} (2004), 55--111.


\bibitem {crew}
	{\sc R. Crew,}
	{\it Rigidity and Frobenius Structure,}
	Documenta Mathematica. \textbf{22} (2017), 287-296.

\bibitem{lucia}
{\sc L. Di Vizio,}
{\it  On the arithmetic theory of q-difference equations. The q-analogue of the Grothendieck-Katz's conjecture on p-curvatures,}
Invent. Math \textbf{150} (2002), 517--578.


\bibitem{padiccycles}
{\sc 	B. Dwork}
{\it p-adic cycles,}
Publications mathématiques de l’I.H.É.S. \textbf{37} (1969), p. 27-115.

\bibitem {DworksFf}
{\sc B.M Dwork,}
{\it On p-adic differential equations I. The Frobenius structure of differential equations,}
Bull. Soc. Math. France \textbf{39-40} (1974), 27--37.



\bibitem{Kedlaya}
	{\sc K. Kedlaya,}
	{\it $p$-adic differential equations,}
	Cambridge studies in advanced mathematics \textbf{125}, Cambridge University Press, 2010. 


\bibitem{pulita}
{\sc A. Pulita,}
{\it p-Adic Confluence of q-Difference Equations,}
Compos. Math \textbf{144} (2008), 867--919.

\bibitem{julien}
{\sc J. Roques,}
{\it Birkhoff matrices, residues and rigidity for q-difference equations.,}
J. Reine Angew. Math \textbf{706} (2015), 215--244.



\bibitem{vmsff}
{\sc D. Vargas-Montoya,}
{\it Algébricité modulo $p$, séries hypergéométriques et structure de Frobenius forte.}
Bull. Soc. Math. France \textbf{149} (2021), 439--477.
\end{thebibliography}
\end{document}